\pdfoutput=1
\documentclass[10pt,a4]{article}
\usepackage{amsthm,amsmath,amssymb}
\usepackage{ulem}
\theoremstyle{definition}
\newtheorem{definition}{Definition}[section]
\newtheorem{lemdef}[definition]{Lemma \& Definition}
\newtheorem{construction}[definition]{Construction}
\newtheorem{subconstruction}{Subconstruction}[definition]
\theoremstyle{theorem}
\newtheorem{theorem}[definition]{Theorem}
\newtheorem{corollary}[definition]{Corollary}
\newtheorem{lemma}[definition]{Lemma}
\newtheorem{proposition}[definition]{Proposition}
\newtheorem{example}[definition]{Example}
\newtheorem{Remark}[definition]{Remark}
\newtheorem{claim}[definition]{Claim}
\newtheorem*{remark}{Remark}
\newtheorem*{ex}{Example}
\theoremstyle{theorem}
\usepackage{fullpage}
\usepackage{mathbbol}
\usepackage{calc}
\usepackage{bm}
\usepackage{mathdots}
\usepackage{mathrsfs}
\usepackage{tikz}
\usetikzlibrary{intersections, calc, math}
\usepackage{tikz-cd}
\usepackage{amsmath}
\usepackage{mathtools}
\DeclareMathOperator{\supp}{supp}
\newcommand{\bs}[1]{\supp\left({#1}\right)}

\newcommand{\sijto}{\mathrel{\circ \hspace{-1.8mm} = \hspace{-2.2mm} > \hspace{-2.2mm} = \hspace{-1.8mm} \circ}}
\newcommand{\nxarrow}{{\nwarrow \hspace{-1em} \nearrow}}
\newcommand{\expnxarrow}{{\nwarrow \hspace{-0.8em} \nearrow}}
\newcommand{\sxarrow}{{\searrow \hspace{-1em} \swarrow}}
\newcommand{\expsxarrow}{{\searrow \hspace{-0.8em} \swarrow}}
\newcommand{\sint}[2]{\underline{\left[{#1},{#2}\right)}}
\newcommand{\insupp}{\overset{\circ}{\in}}
\renewcommand{\qed}{\unskip\nobreak\quad\qedsymbol}

\begin{document}

\title{
New bijective proofs pertaining to alternating sign matrices
}
\author{
Takuya Inoue
\footnote{
Graduate School of Mathematical Sciences, the University of Tokyo,
3-8-1 Komaba, Meguro-ku, 153-8914 Tokyo, Japan.
Email: inoue@ms.u-tokyo.ac.jp
}
}
\date{}
\maketitle

\begin{abstract}
The alternating sign matrices-descending plane partitions (ASM-DPP) bijection problem is one of the most intriguing open problems in bijective combinatorics, which is also relevant to integrable combinatorics.
The notion of a signed set and a signed bijection is used in [Fischer, I. \& Konvalinka, M., Electron. J. Comb., 27 (2020) 3-35.] to construct a bijection between $\text{ASM}_n \times \text{DPP}_{n-1}$ and $\text{DPP}_n \times \text{ASM}_{n-1}$.
Here, we shall construct a more natural alternative to a signed bijection between alternating sign matrices and shifted Gelfand-Tsetlin patterns which is presented in that paper, based on the notion of compatibility which we introduce to measure the naturalness of a signed bijection.
In addition, we give a bijective proof for the refined enumeration of an extension of alternating sign matrices with $n+3$ statistics, first proved in [Fischer, I. \& Schreier-Aigner, F., Advances in Mathematics 413 (2023)  108831.].
\end{abstract}

\tableofcontents

\section{Introduction}
An alternating sign matrix with rank $n$ is defined as an $n \times n$ matrix that satisfies the conditions:
\begin{itemize}
\item each entry is $-1$, $0$ or $1$,
\item for each row and column the sum is $1$,
\item for each row and column the nonzero entries alternate in sign.
\end{itemize}
Below is an example of a rank $4$ alternating sign matrix:
\begin{equation} \label{intro::ex}
	\begin{pmatrix}
	0 & 0 & 1 & 0 \\
	1 & 0 & -1 & 1 \\
	0 & 1 & 0 & 0 \\
	0 & 0 & 1 & 0
	\end{pmatrix}.
\end{equation}
Combinatorics concerning alternating sign matrices provides one of the most intriguing open problems in bijective combinatorics, which is called the ASM-DPP bijection problem \cite{Francesco1,Stan}.
Here, DPP stands for descending plane partitions, which are plane partitions subject to some conditions and equipped with the notion of rank.
It has been known for some time that for any $n \in \mathbb{Z}_{>0}$ the number of alternating sign matrices with rank $n$ is equal to that of descending plane partitions with rank $n$
(conjectured in 1983 by W. H. Mills, David P. Robbins and Howard Rumsey, Jr. \cite{MRR}, and proved in 1996 first by Doron Zeilberger \cite{Zeil} and then by Greg Kuperberg \cite{Kupe} independently).
However, no explicit bijections between them have been found so far.

Ilse Fischer and Matja\v{z} Konvalinka use the notions of a signed set and a signed bijection (which we shall call a \textit{sijection} in the style of Fischer and Konvalinka) to tackle the problem in \cite{FK1,FK2}.
These are defined as follows:
\begin{itemize}
\item A signed set is a pair of disjoint finite sets,
\item A sijection from a signed set $S=(S^+,S^-)$ to a signed set $T=(T^+,T^-)$ is a bijection between $S^{+} \sqcup T^{-}$ and $S^{-} \sqcup T^{+}$.
\end{itemize}
For more details, see Section \ref{sec::prel}.
Fischer and Konvalinka did not construct the desired bijection, but they did construct a bijection between $\text{ASM}_n \times \text{DPP}_{n-1}$ and $\text{DPP}_n \times \text{ASM}_{n-1}$, where $\text{ASM}_n$ and $\text{DPP}_n$ are the set of alternating sign matrices and descending plane partitions with rank $n$, respectively.
In Section \ref{secofSGT}, we will investigate the sijections constructed in \cite{FK1}, between alternating sign matrices and so-called shifted Gelfand-Tsetlin patterns, and obtain some new combinatorial results.
A shifted Gelfand-Tsetlin pattern is a combinatorial object defined in \cite{FK1} to clarify the combinatorial meaning of the operator formula for the enumeration of alternating sign matrices.
For more details, see Section \ref{secofSGT}.

Before explaining our method of investigation, let us introduce some other previous works on the ASM-DPP bijection problem, done by Roger E. Behrend, Philippe Di Francesco and Paul Zinn-Justin \cite{BFZ1,BFZ2}.
Their research concerns refined enumerations of these two objects, and they prove that the two refined enumerations with respect to certain quadruplets of statistics coincide with each other. Note that the results on three of these statistics were conjectured by W. H. Mills, David P Robbins and Howard Rumsey Jr. \cite{MRR}
The statistics on alternating sign matrices they consider are:
\begin{itemize}
\item the column number of the $1$ in the first row (note that only one $1$ exists in that row),
\item the (generalized) inversion number (for a detailed definition see Section \ref{secofSGT}),
\item the number of $-1$ in the matrix,
\item the number of $0$'s to the right of the $1$ in the last row
\end{itemize}
For example, the statistics are $3$, $2$, $1$ and $1$ of the example (\ref{intro::ex}).

The main idea of the present paper is the notion of compatibility, which is defined as follows:
\begin{quote}
	A sijection $\varphi$ from $S=(S^+,S^-)$ to $T=(T^+,T^-)$ is \textit{compatible} with a statistic function $\eta \colon S^+ \sqcup S^- \sqcup T^+ \sqcup T^- \to \mathbb{Z}^d$
	if it holds that:
	\[
		\forall s \in S^+ \sqcup T^-,\, \eta(\varphi(s)) = \eta(s).
	\]
\end{quote}
This means that the statistic $\eta$ is preserved under the action of $\varphi$.
In Section \ref{secofGT}, we interpret the sijections relevant to Gelfand-Tsetlin patterns that are constructed in \cite{FK1} in the light of the notion of compatibility.
This interpretation gives us some new knowledge about these objects.
With the notion of integrable systems of combinatorial objects, which will also be defined in that section, we will describe the combinatorial structure of Gelfand-Tsetlin patterns.
This result leads us to a new computational proof of the enumeration of Gelfand-Tsetlin patterns with a bottom row that is not strictly increasing.
Furthermore, inspired by this new proof, we obtain a new generalization of Gelfand-Tsetlin patterns and a signed enumeration of the objects.

The notion of compatibility also helps us to import Di Francesco et al.'s results to Fischer and Konvalinka's work.
We know the meaningful statistics according to the refined enumerations, so with these statistics and the notion of compatibility, we are able to measure the naturalness of constructed sijections.
In Section \ref{secofSGT}, we will extend the definition of two of the four statistics to shifted Gelfand-Tsetlin patterns and discuss the compatibility with these two statistics.
In fact, Fischer and Konvalinka's construction is not compatible with one of these statistics.
We will construct a new sijection between the two objects which is compatible with both of the statistics.
Therefore, according to the notion of compatibility, our construction can be considered more natural than Fischer and Konvalinka's.

In fact, in \cite{AF} Fischer and Schreier-Aigner obtain a more elaborate result to do with a refined enumeration of alternating sign matrices and descending plane partitions.
Amongst others, they provide a refined enumeration of some extensions of alternating sign matrices with $n+3$ statistics.
With a slight modification, this result can be recognized as a refined enumeration of generalized monotone triangles, which is defined in Section \ref{secofSGT}.
It is expressed as follows (for detailed definitions, see Section \ref{secofSGT}):
\[
	\prod_{i=1}^n \left( uX_i + vX_i^{-1} + w \right) \prod_{1 \leq p < q \leq n} \left( u E_{k_p} + v E_{k_q}^{-1} + w E_{k_p} E_{k_q}^{-1} \right) \tilde{s}_{(k_n,k_{n-1},\ldots,k_1)}(X_1,X_2,\ldots,X_n).
\]
We recognize this expression as a refined enumeration of the Cartesian product of arrow rows and shifted Gelfand-Tsetlin patterns, where arrow rows will be defined in Section \ref{secofSGT}.
Then, we give a bijective proof for the refined enumeration of generalized monotone triangles by appropriately defining the $n+3$ statistics on the Cartesian product and constructing a sijection between generalized monotone triangles and the Cartesian product that is compatible with these statistics. For more details, see Subsection \ref{ssec::gmt-ar_sgt}.

As a last remark in this section we explain the choice of title captions in this paper, in which we construct many sijections.
On the one hand, such constructions constitute proof of the existence of the sijections, but  on the other hand the constructions themselves also become objects that will be used in later propositions and/or theorems.
Therefore, not only the existence but also the explicit construction of each sijection is important.
Thus, when we describe a construction of a sijection in detail we will use ``Construction'' as a title caption instead of  ``Proposition'', ``Theorem'' or ``Proof''.
This convention is partially derived from Fischer and Konvalinka's papers \cite{FK1,FK2}.

\section{Preliminaries}\label{sec::prel}
\subsection{Signed Sets}
A \textit{signed set} is a pair of disjoint finite sets.
For a signed set $S$, we call its first (resp. second) element the \textit{plus part} (resp. \textit{minus part}) of $S$ and denote it by $S^{+}$ (resp. $S^{-}$).
Namely, $S=(S^{+},S^{-})$.
For a signed set $S$, we call the set $S^{+} \sqcup S^{-}$ the support of $S$ and denote it by $ \bs{S} $.
The \textit{size} of a signed set $S$ is defined by $ \#S= \#S^{+} - \#S^{-}$, where we denote the size of a set $X$ by $\#X$.

Next, we define some basic notions relevant to signed sets. (See also \cite{FK1}.)
\begin{itemize}
\item The \textit{opposite} of a signed set $S$ is $-S := (S^-,S^+)$.
\item The \textit{disjoint union} of two signed sets $S$ and $T$ is $S \sqcup T := (S^{+} \sqcup T^{+},S^{-} \sqcup T^{-})$.
\item The \textit{Cartesian product} of two signed sets $S$ and $T$ is \[S \times T := (S^{+} \times T^{+} \sqcup S^{-} \times T^{-},S^{+} \times T^{-} \sqcup S^{-} \times T^{+}).\]
\end{itemize}
We define the disjoint union and the Cartesian product of a finite number of signed sets in the same manner.
In addition, we can define a \textit{disjoint union with signed index} of a family of signed sets indexed by a signed set (to be rigorous, indexed by the support of the signed set) in the following way.
\begin{definition}
	Let $T$ be a signed set and $\{S_t\}_{t \in \bs{T}}$ a family of signed sets.
	The disjoint union with signed index in $T$ of $\{S_t\}_{t \in \bs{T}}$ is
	\[
		\bigsqcup_{t \in T} S_t := \left( \bigsqcup_{t \in T^{+}}(S_t^+) \sqcup \bigsqcup_{t \in T^{-}}(S_t^-), \bigsqcup_{t \in T^{+}}(S_t^-) \sqcup \bigsqcup_{t \in T^{-}}(S_t^+) \right).
	\]
\end{definition}
We denote an element of $\bigsqcup_{t \in T} S_t$ as $(s_t,t)$ for $s_t \in \bs{S_t}$.
For simplicity, we will sometimes denote an element of $\bigsqcup_{u \in U} \bigsqcup_{t \in T} S_{t,u}$ as $(s_{t,u},t,u)$ instead of $((s_{t,u},t),u)$.

Let us define a \textit{signed interval}, which is the most basic example of a signed set.
For any two integers $a$ and $b$, a signed interval $\sint{a}{b}$ is defined by,
\[
	\sint{a}{b}=\begin{cases}
		([a,b)\cap\mathbb{Z},\emptyset) & (a < b), \\
		(\emptyset,\emptyset) & (a=b), \\
		(\emptyset,[b,a)\cap\mathbb{Z}) & (a>b).
	\end{cases}
\]
It is noteworthy that we use half-open intervals to describe signed intervals instead of closed intervals as in \cite{FK1,FK2}.
Because of this, we have $\sint{b}{a}=-\sint{a}{b}$.
This relation is crucial and many properties of Gelfand-Tsetlin patterns are easier to establish using this notation.
For more information, see Section \ref{secofGT}.

\subsection{Sijections}
A \textit{sijection} $ \varphi $ from a signed set $S$ to a signed set $T$ is an involution on $\bs{S} \sqcup \bs{T}$ such that $ \varphi(S^{+} \sqcup T^{-}) = S^{-} \sqcup T^{+}$,
namely a sijection $ \varphi $ is a bijection between $S^{+} \sqcup T^{-}$ and $S^{-} \sqcup T^{+}$.
We denote it by $S \sijto T$.
If there is a sijection from $S$ to $T$, then $\#S = \#T$ holds.
This relation is an analogy of the relation between ordinary sets and a bijection.
In particular, if $ S^{-}=T^{-}=\emptyset $, then we can interpret $S$ and $T$ as ordinary sets and a sijection between them as an ordinary bijection.

A sijection $ \varphi $ can be also recognized as a triplet of a sign-preserving bijection from a subset of $S$ to a subset of $T$, a sign-reversing involution on the remaining part of $S$, and a sign-reversing involution on the remaining part of $T$.
Figure \ref{illust_sij} below illustrates this interpretation of a sijection.
In the figure, the upper-left square represents the set $S^{+}$, and similarly for the other three squares.
The symbol $\sijto$ was inspired by Figure \ref{illust_sij}.

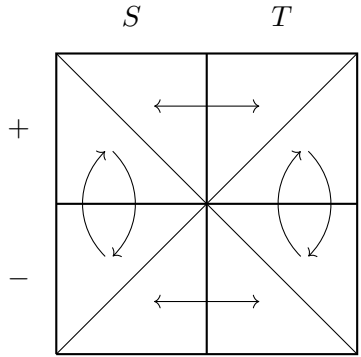
\begin{figure}[h]
\centering
\begin{tikzpicture}
\draw[thick] (0,0) -- (4,0) -- (4,4) -- (0,4) -- (0,0);
\draw[thick] (0,2) -- (4,2);
\draw[thick] (2,0) -- (2,4);
\draw (0,0) -- (4,4);
\draw (4,0) -- (0,4);
\node (S) at (1,4.5) {$S$};
\node (T) at (3,4.5) {$T$};
\node (+) at (-0.5,3) {$+$};
\node (-) at (-0.5,1) {$-$};
\draw[<->] (1.3,3.3) -- (2.7,3.3);
\draw[<->] (1.3,0.7) -- (2.7,0.7);
\draw[->] (0.75,2.7) .. controls (1.15,2.3) and (1.15,1.7) .. (0.75,1.3);
\draw[->] (0.65,1.3) .. controls (0.25,1.7) and (0.25,2.3) .. (0.65,2.7);
\draw[->] (3.35,2.7) .. controls (3.75,2.3) and (3.75,1.7) .. (3.35,1.3);
\draw[->] (3.25,1.3) .. controls (2.85,1.7) and (2.85,2.3) .. (3.25,2.7);
\end{tikzpicture}
\caption{An illustration of a sijection.}
\label{illust_sij}
\end{figure}

Since a sijection $\varphi \colon S \sijto T$ is an involution, we can use the notion ``$\varphi(s)$'' for an element $s$ not only of $S$ but also of $T$.
On the other hand, it is sometimes important to distinguish between the domain and the codomain, especially when we consider a composition of sijections.
Thus, we define an inversion $\varphi^{-1} \colon T \sijto S$ of $\varphi$ by
\[
	\varphi^{-1} = \varphi \quad (\textrm{as an involution}),
\]
and distinguish $\varphi^{-1}$ from $\varphi$ as a sijection.

The simplest example of a sijection is the \textit{identity sijection} $\textrm{id}_S$, defined by
\[\begin{tikzcd}[row sep=tiny]
S \arrow[r,phantom, "\sijto"] 	& S \\
s \arrow[r,leftrightarrow]		& s
\end{tikzcd}\]
for any signed set $S$. Let us introduce a few more examples which we shall use in later sections.
%
\begin{example}\label{ex_sij}
For any integers $a$, $b$ and $c$, there exists a sijection $\varphi \colon \sint{a}{b} \sijto \sint{a}{c} \sqcup \sint{c}{b}$.
In fact, each integer in $[\min \{a,b,c\}, \max \{a,b,c\}) \cup \mathbb{Z}$ appears twice as an element, and when we let one correspond to the other, we obtain the sijection.
\end{example}

\begin{example}
Let $A$, $B$ be ordinary sets and $f \colon A \to B$ a bijection.
In this situation, we can interpret $f$ as a sijection between $S=(A,B)$ and $(\emptyset,\emptyset)$.
\end{example}

\begin{example}\label{sij_oppo}
Let $S$, $T$ be signed sets and $\varphi \colon S \sijto T$ a sijection.
Then, we have a sijection $S \sqcup -T \sijto (\emptyset,\emptyset)$.
In fact, $\varphi$ is also a bijection between $S^+ \sqcup T^-$ and $S^- \sqcup T^+$.
Therefore, this is also a sijection between $(S^+ \sqcup T^-, S^- \sqcup T^+)=S \sqcup -T$ and $(\emptyset,\emptyset)$ by the previous example.
Conversely, a sijection between $S \sqcup -T$ can be interpreted as a sijection between $S$ and $T$.
\end{example}
In particular, we have $$S \sqcup -S \sijto (\emptyset,\emptyset)$$ derived from the identity sijection on $S$ for any signed set $S$.
\subsubsection{Composition of sijections}
A composition of sijections is defined as follows \cite{FK1}.
\begin{lemdef}[{\cite[Proposition 2 (1)]{FK1}}]\label{lem:comp}
Let $S$, $T$, $U$ be signed sets and $ \varphi \colon S \sijto T $, $ \psi \colon T \sijto U $ sijections.
For $ s \in \bs{S} $, we define $ \psi \circ \varphi(s) $ as the last well-defined element in the sequence
\[
	s,\, \varphi(s),\, \psi(\varphi(s)),\, \varphi(\psi(\varphi(s))),\, \ldots.
\]
Similarly, for $ u \in \bs{U} $, we define $ \psi \circ \varphi(u) $ as the last well-defined element in the sequence
\[
	u,\, \psi(u),\, \varphi(\psi(u)),\, \psi(\varphi(\psi(u))),\, \ldots.
\]
Then, $ \psi \circ \varphi $ is a sijection from $S$ to $U$, and we call it the composition of $ \varphi $ and $ \psi $.
\end{lemdef}
First we should clarify what we mean by \textit{``the last well-defined element in the sequence \\
$ s,\, \varphi(s),\, \psi(\varphi(s)),\, \varphi(\psi(\varphi(s))),\, \ldots $''}.
For example if $\varphi(s)$ belongs to $\bs{S}$, then $\psi(\varphi(s))$ is not defined because the domain of $\psi$ is $\bs{T} \sqcup \bs{U}$.
Therefore, $ (\psi \circ \varphi)(s) $ is $\varphi(s)$ in this case.
For the original proof of the lemma, see \cite{Doyle}.
Here, we give an alternative proof of this fact using the language of graphs.
A graph in the proof might have multiple edges, so we fix notations relevant to multisets.
We use double braces to describe a multiset, and use the symbol ``$+$'' to describe a sum of multisets.
\begin{proof}[Proof of Lemma \& Definition \ref{lem:comp}]
Consider a graph with vertices $V=\bs{S} \sqcup \bs{T} \sqcup \bs{U}$ and edges $E=R+B$, where
\begin{align*}
	R&=\{\!\{\, \{v,\varphi(v)\} \mid v \in S^{+} \sqcup T^{-}\,\}\!\}, \\
	B&=\{\!\{\, \{v,\psi(v)\} \mid v \in T^{+} \sqcup U^{-}\,\}\!\}.
\end{align*}
Note that the graph is an undirected bipartite finite graph, where one part is $S^{+} \sqcup T^{-} \sqcup U^{+}$ and the other part is $S^{-} \sqcup T^{+} \sqcup U^{-}$.
We consider painting edges in R with red and those in B with blue.
Then, by the definitions, each vertex belongs to at most one red edge and at most one blue edge.
Because of the degrees of the vertices and the finiteness of the graph, the graph consists of finitely many line graphs and cycles.
Moreover, all the endpoints of these line graphs belong to $\bs{S} \sqcup \bs{U}$ and for each line graph or cycle the color of the edges alternate.
Please refer to Figure \ref{fig::comp}  below for an example:
dashed edges represent those in $B$ and the others those in $R$. In addition, the shape of a node corresponds to which part it belongs to.
\begin{figure}[h]
\centering
\begin{tikzpicture}[main/.style = {draw, circle}]
\node (S) at (0,7) {$S$};
\node (phi) at (1.5,7.1) {$\overset{\varphi}{\sijto}$};
\node (T) at (3,7) {$T$};
\node (phi) at (4.5,7.1) {$\overset{\psi}{\sijto}$};
\node (U) at (6,7) {$U$};
\node (+) at (-2,4.5) {$+$};
\node (-) at (-2,1) {$-$};
\node[main] (1) at (0,5.5) {$+$};
\node[main] (2) at (0,3.5) {$+$};
\node[draw,rectangle] (3) at (0,1) {$-$};
\node[draw,rectangle] (4) at (3,6) {$+$};
\node[draw,rectangle] (5) at (3,5) {$+$};
\node[draw,rectangle] (6) at (3,4) {$+$};
\node[draw,rectangle] (7) at (3,3) {$+$};
\node[main] (8) at (3,2) {$-$};
\node[main] (9) at (3,1) {$-$};
\node[main] (10) at (3,0) {$-$};
\node[main] (11) at (6,5.5) {$+$};
\node[main] (12) at (6,3.5) {$+$};
\node[draw,rectangle] (13) at (6,1) {$-$};
\draw[-,thick,red] (1) -- (4);
\draw[-,thick,red] (2) -- (5);
\draw[-,thick,red] (3) -- (10);
\draw[-,thick,red] (6) to [out=180,in=180] (9);
\draw[-,thick,red] (7) to [out=180,in=180] (8);
\draw[-,dashed,thick,blue] (4) -- (11);
\draw[-,dashed,thick,blue] (5) to [out=0,in=0] (9);
\draw[-,dashed,thick,blue] (6) to [out=0,in=0] (10);
\draw[-,dashed,thick,blue] (7) to [out=0,in=0] (8);
\draw[-,dashed,thick,blue] (12) to [out=180,in=180] (13);
\draw (-2.5,-0.5) to (7.5,-0.5);
\draw (-2.5,2.5) to (7.5,2.5);
\draw (-2.5,6.5) to (7.5,6.5);
\end{tikzpicture}
\caption{A composition of sijections.}
\label{fig::comp}
\end{figure}
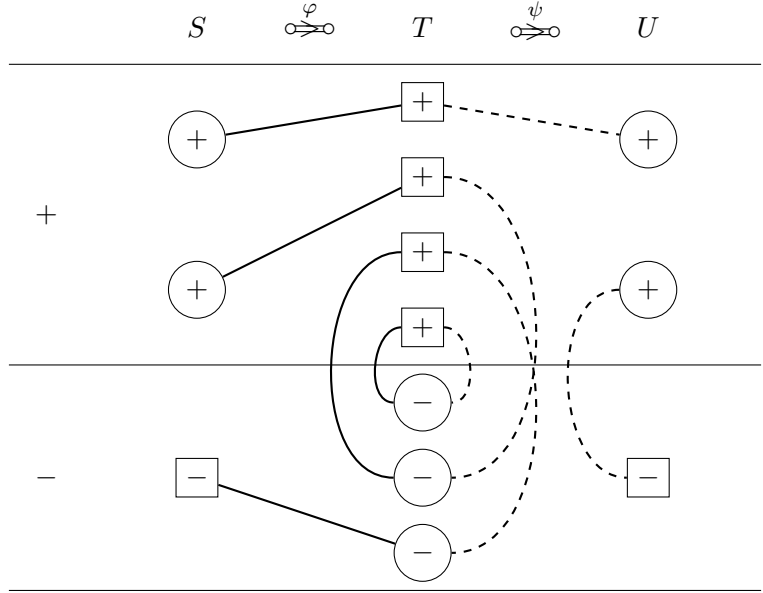

Consider $s \in \bs{S}$.
The degree of $s$ is $1$, so it is an endpoint of a line graph.
Therefore, $ \psi \circ \varphi(s) $ is the other endpoint of the line graph.
If $ \psi \circ \varphi(s) \in \bs{S} $, the first edge and the last edge in the line graph are both red, so its length is odd.
Thus, $s$ and $ \psi \circ \varphi(s) $ belong to different parts of $G$ as a bipartite graph.
Therefore, the sign of $\psi \circ \varphi(s)$ is different from that of $s$.
In the other case, namely when $ \psi \circ \varphi(s) \in \bs{U} $, the length of the line graph is even and $ \psi \circ \varphi(s) $ belongs to the same part as $s$,
so the sign of $\psi \circ \varphi(s)$ is equal to that of $s$.

For $ u \in \bs{U} $, we can give a similar proof. Therefore, $\psi \circ \varphi$ is a well-defined sijection from $S$ to $U$.
\end{proof}

Relating to this proof we define a graph $G=(V,E)$ of a sijection $\varphi \colon S \sijto T$ as follows:
\begin{align*}
	V &:= \bs{S} \sqcup \bs{T}, \\
	E &:= \{\!\{\, \{v,\varphi(v)\} \mid v \in S^{+} \sqcup T^{-}\,\}\!\}.
\end{align*}
This notion is useful to prove some lemmas and to understand properties of sijections.
For example, with this language, the sijection $\varphi \colon \sint{1}{3} \sqcup \sint{3}{2} \sijto \sint{1}{2}$ in Example \ref{ex_sij} is described as:
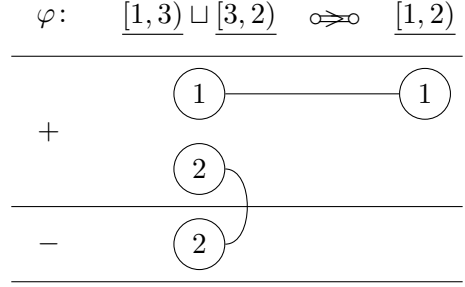
\[
\begin{tikzpicture}[main/.style = {draw, circle}]
\node (S) at (0,3) {$\sint{1}{3} \sqcup \sint{3}{2}$};
\node (T) at (3,3) {$\sint{1}{2}$};
\node (phi) at (-1.9,3) {$\varphi \colon$};
\node (sij) at (1.8,3) {$\sijto$};
\node (+) at (-2,1.5) {$+$};
\node (-) at (-2,0) {$-$};
\draw (-2.5,-0.5) to (3.5,-0.5);
\draw (-2.5,0.5) to (3.5,0.5);
\draw (-2.5,2.5) to (3.5,2.5);
\node[main] (1) at (0,2) {$1$};
\node[main] (2) at (0,1) {$2$};
\node[main] (3) at (0,0) {$2$};
\node[main] (4) at (3,2) {$1$};
\draw[-] (1) -- (4);
\draw[-] (2) to [out=0,in=0] (3);
\end{tikzpicture}
\]

\begin{remark}
There are two reasons why we impose finiteness on signed sets.
One of them is because, without the finiteness it could simply be meaningless since $(\mathbb{N},\mathbb{N})$ is in sijection to arbitrary signed sets (and any infinite ``signed set'').
The other reason is because it is needed for the well-definedness of compositions of sijections.
For example, let $S=(\{1\},\emptyset)$, $T=(2\mathbb{N}, 2\mathbb{N}+1)$, $U=(\emptyset,\emptyset)$, $\varphi(2x-1)=2x$ and $\psi(2x)=2x+1$.
If we define a graph in the same way as in finite cases, the connected component that $1 \in S^+$ belongs to is still a line graph but not finite since we have
\[
	\varphi(1) = 2,\quad \psi(\varphi(1))=3,\quad \varphi(\psi(\varphi(1)))=4,\quad \ldots.
\]
Thus, the proof and the definition are broken in infinite cases.
\end{remark}

Last, we prove that compositions of sijections have the associative property.
\begin{proposition}[\cite{Doyle}, Corollary 3]
Let $S$, $T$, $U$ and $V$ be signed sets and $ \varphi \colon S \sijto T $, $ \psi \colon T \sijto U $, and $ \xi \colon U \sijto V $ sijections. Then,
\[
	\xi \circ ( \psi \circ \varphi ) = ( \xi \circ \psi ) \circ \varphi .
\]
\end{proposition}
\begin{proof}
Consider a graph with vertices $ \bs{S} \,\sqcup\, \bs{T} \,\sqcup\, \bs{U} \,\sqcup\, \bs{V} $ and edges
$ \{\!\{\, \{v,\varphi(v)\} \mid v \in S^{+} \sqcup T^{-}\,\}\!\} + \{\!\{\, \{v,\psi(v)\} \mid v \in T^{+} \sqcup U^{-}\,\}\!\} + \{\!\{\, \{v,\xi(v)\} \mid v \in U^{+} \sqcup V^{-}\,\}\!\} $.
This graph consists of line graphs and cycles.
Moreover, for any endpoint $v$ of a line graph the other endpoint of the line graph is $(\xi \circ ( \psi \circ \varphi ))(v) = (( \xi \circ \psi ) \circ \varphi)(v)$.
\end{proof}

\subsubsection{Cartesian products of sijections} \label{sec::Cart_sij}
The definition of the Cartesian product of sijections is a bit more complicated.
For sijections $\varphi_i \colon S_i \sijto T_i$ ($i=1,2$),
we can construct a sijection $\varphi_1 \times \varphi_2$ between $S_1 \times S_2$ and $T_1 \times T_2$ as follows:
\begin{quote}
	For $s=(s_1,s_2) \in \bs{S_1 \times S_2} = \bs{S_1} \times \bs{S_2}$,
	\[
		(\varphi_1 \times \varphi_2)(s)
		= \begin{cases}
		(\varphi_1(s_1),s_2) &\qquad \varphi_1(s_1) \in \bs{S_1}, \\
		(s_1,\varphi_2(s_2)) &\qquad \varphi_1(s_1) \in \bs{T_1} \text{ and } \varphi_2(s_2) \in \bs{S_2}, \\
		(\varphi_1(s_1),\varphi_2(s_2)) &\qquad \varphi_1(s_1) \in \bs{T_1} \text{ and } \varphi_2(s_2) \in \bs{T_2},
		\end{cases}
	\]
	and for $t=(t_1,t_2) \in \bs{T_1 \times T_2} = \bs{T_1} \times \bs{T_2}$,
	\[
		(\varphi_1 \times \varphi_2)(t)
		= \begin{cases}
		(\varphi_1(t_1),t_2) &\qquad \varphi_1(t_1) \in \bs{T_1}, \\
		(t_1,\varphi_2(t_2)) &\qquad \varphi_1(t_1) \in \bs{S_1} \text{ and } \varphi_2(t_2) \in \bs{T_2}, \\
		(\varphi_1(t_1),\varphi_2(t_2)) &\qquad \varphi_1(t_1) \in \bs{S_1} \text{ and } \varphi_2(t_2) \in \bs{S_2}.
		\end{cases}
	\]
\end{quote}
It is easy to check that this sijection is indeed an involution on $\bs{S_1 \times S_2} \sqcup \bs{T_1 \times T_2}$ and that it meets the sign conditions:
\begin{itemize}
\item if $(\varphi_1 \times \varphi_2)(s) \in \bs{S_1 \times S_2}$, the sign of $(\varphi_1 \times \varphi_2)(s)$ is different from the sign of $s$,
\item if $(\varphi_1 \times \varphi_2)(s) \in \bs{T_1 \times T_2}$, the sign of $(\varphi_1 \times \varphi_2)(s)$ is equal to the sign of $s$,
\item and similar results are true for $t \in \bs{T_1 \times T_2}$.
\end{itemize}
Therefore, it is indeed a sijection.
Additionally, this definition has the associative property.
\begin{proposition}\label{prop_asso_Cart}
Let $ \varphi_i \colon S_i \sijto T_i $ be a sijection for $i=1,2,3$. Then, we have
\[
	(\varphi_1 \times \varphi_2) \times \varphi_3 = \varphi_1 \times (\varphi_2 \times \varphi_3).
\]
\end{proposition}
\begin{proof}
Regardless of how one puts the parentheses, if all $\varphi_i$ send $s_i$ to $T_i$, then $s=(s_1,s_2,s_3)$ corresponds to $(\varphi_1(s_1),\varphi_2(s_2),\varphi_3(s_3))$.
Otherwise, only the leftmost element $s_i$ of $s$ such that $\varphi_i(s_i) \in S_i$ is replaced with $\varphi_i(s_i)$.
\end{proof}

Thus, we can define the Cartesian product of a finite number of sijections as follows. (See also Proposition 2(2) of \cite{FK1}.)
\begin{definition}
Let $S_1,S_2,\ldots,S_k$, $T_1,T_2,\ldots,T_k$ be signed sets and let $ \varphi_i \colon S_i \sijto T_i $ be a sijection for $i=1,2,\ldots,k$.
We define
\[
	\varphi_1 \times \varphi_2 \times \cdots \times \varphi_k = (\cdots((\varphi_1 \times \varphi_2) \times \varphi_3) \times \cdots ) \times \varphi_k.
\]
\end{definition}
Let $s=(s_1,s_2,\ldots,s_k) \in \bs{S_1 \times S_2 \times \cdots \times S_k} = \bs{S_1} \times \bs{S_2} \times \cdots \times \bs{S_k}$.
Then, according to the proof of Proposition \ref{prop_asso_Cart},
when $\varphi_i(s_i) \in \bs{T_i}$ for all $i = 1,2,\ldots,k$, we have
\[
	(\varphi_1 \times \varphi_2 \times \cdots \times \varphi_k)(s) = (\varphi_1(s_1),\varphi_2(s_2),\ldots,\varphi_k(s_k)).
\]
Otherwise, we have
\[
	(\varphi_1 \times \varphi_2 \times \cdots \times \varphi_k)(s) = (s_1,s_2,\ldots,s_{j-1},\varphi_j(s_j),s_{j+1},\ldots,s_k),
\]
where $j$ is the minimum index such that $\varphi_j(s_j) \in \bs{S_j}$.
For an element of $\bs{T_1 \times T_2 \times \cdots \times T_k}$, a similar expression can be given.

\begin{Remark} \label{rem-2.2.2}
We remark on a disadvantageous feature of this Cartesian product of sijections.
Let $\varphi_i \colon S_i \sijto T_i$ and $\psi_i \colon T_i \sijto U_i$ be sijections for $i=1,2$.
In general, it does not hold that
\begin{equation} \label{cartesian-counterex}
	(\psi_1 \times \psi_2) \circ (\varphi_1 \times \varphi_2) = (\psi_1 \circ \varphi_1) \times (\psi_2 \circ \varphi_2).
\end{equation}
For example, let $S=(\{A\},\emptyset)$, $T=(\{A,B\},\{B^\dag\})$ and define a sijection $\varphi$ as
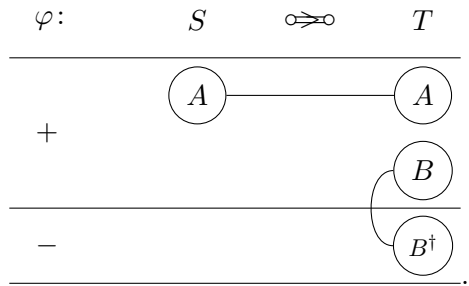
\[
\begin{tikzpicture}[main/.style = {draw, circle}]
\node (S) at (0,3) {$S$};
\node (T) at (3,3) {$T$};
\node (phi) at (-1.9,3) {$\varphi \colon$};
\node (sij) at (1.5,3) {$\sijto$};
\node (+) at (-2,1.5) {$+$};
\node (-) at (-2,0) {$-$};
\draw (-2.5,-0.5) to (3.5,-0.5);
\draw (-2.5,0.5) to (3.5,0.5);
\draw (-2.5,2.5) to (3.5,2.5);
\node[main] (1) at (0,2) {$A$};
\node[main] (2) at (3,1) {$B$};
\node[scale=0.84,main] (3) at (3,0) {$B^\dag$};
\node[main] (4) at (3,2) {$A$};
\draw[-] (1) -- (4);
\draw[-] (2) to [out=180,in=180] (3);
\end{tikzpicture}.
\]

If the relation (\ref{cartesian-counterex}) is always true, it must hold that
\[
	(\text{id}_T \times \varphi) \circ (\varphi \times \text{id}_S) = \varphi \times \varphi =  (\varphi \times \text{id}_T) \circ (\text{id}_S \times \varphi).
\]
However, we have by a simple calculation that
\[
	((\text{id}_T \times \varphi) \circ (\varphi \times \text{id}_S))((B,B)) = (B,B^\dag) \ne (B^\dag,B) = ((\varphi \times \text{id}_T) \circ (\text{id}_S \times \varphi))((B,B)),
\]
so we have a contradiction. Thus, the relation (\ref{cartesian-counterex}) is not always true.

This feature is inevitable, namely it is not due to a defect in the definition.
For instance, in the above situation, we cannot determine which of $(B,B^\dag)$ and $(B^\dag,B)$ should be $(\varphi \times \varphi)((B,B))$ a priori because of symmetry.
\end{Remark}

\subsubsection{Disjoint unions of sijections}
For sijections $\varphi_i \colon S_i \sijto T_i$ ($i=1,2$), the disjoint union $\varphi_1 \sqcup \varphi_2 \colon S_1 \sqcup S_2 \sijto T_1 \sqcup T_2$ is defined by
\[ (\varphi_1 \sqcup \varphi_2)(s) = \begin{cases} \varphi_1(s) & s \in \bs{S_1} \sqcup \bs{T_1}, \\ \varphi_2(s) & s \in \bs{S_2} \sqcup \bs{T_2}. \end{cases}\]
We define the disjoint union of a finite number of sijections in the same manner.
Considering the graph of sijections, the disjoint union is just a juxtaposition, so we have $(\varphi_1 \sqcup \varphi_2) \sqcup \varphi_3 = \varphi_1 \sqcup (\varphi_2 \sqcup \varphi_3)$.

Next, we define the disjoint union with signed index of a family of sijections.
\begin{definition}[Proposition 2(3) in \cite{FK1}] \label{def_sij_DIS}
Let $T$, $\tilde{T}$ be signed sets and $\psi \colon T \sijto \tilde{T}$ a sijection.
Assume that we have signed sets $\{ S_t \}_{t \in \bs{T} \sqcup \bs{\tilde{T}}}$ and sijections
$ \{ \varphi_t \colon S_t \sijto S_{\psi(t)} \}_{ t \in T^+ \sqcup \tilde{T}^-}  $.
Let $\varphi_t = \varphi_{(\psi(t))}^{-1}$ for $t \in T^- \sqcup \tilde{T}^+$.
Then, we have a sijction $\varphi$ from $\bigsqcup_{t \in T} S_t$ to $\bigsqcup_{t \in \tilde{T}} S_t$ defined by
\[
	\varphi\left((s_t,t)\right)=\begin{cases} (\varphi_t(s_t),t) & \text{if } \varphi_t(s_t) \in S_t, \\ (\varphi_t(s_t),\psi(t)) & \text{if } \varphi_t(s_t) \in S_{\psi(t)}. \end{cases}
\]
We denote it by $\bigsqcup_{t \in T \sqcup \tilde{T}} \varphi_t$.
\end{definition}

Considering the graphs of sijections, the disjoint union of this type is also a juxtaposition. Therefore, it is indeed a sijection.
When $T=\tilde{T}$ and $\psi=\mathrm{id}_T$ hold, the situation becomes simple.

Let $\{ \varphi_t \colon S_t \sijto \tilde{S}_t \}_{t \in \bs{T}}$ be sijections, then we have
\[
	\bigsqcup_{t \in T} \varphi_t \left( := \bigsqcup_{t \in T \sqcup T} \varphi_t \right)\colon \bigsqcup_{t \in T} S_t \sijto \bigsqcup_{t \in T} \tilde{S}_t.
\]


We shall construct another four canonical sijections.
\begin{construction}\label{sij_dist}
Let $S$, $T_1$ and $T_2$ be signed sets. Then, we have a canonical sijection
\[
	S \times (T_1 \sqcup T_2) \sijto (S \times T_1) \sqcup (S \times T_2),
\]
which is derived from the identities
\begin{align*}
	\left( S \times (T_1 \sqcup T_2) \right)^+ &= \left( (S \times T_1) \sqcup (S \times T_2) \right)^+ \\
	&= (S^+ \times T_1^+) \sqcup (S^+ \times T_2^+) \sqcup (S^- \times T_1^-) \sqcup (S^- \times T_2^-),\\
	\left( S \times (T_1 \sqcup T_2) \right)^- &= \left( (S \times T_1) \sqcup (S \times T_2) \right)^- \\
	&= (S^+ \times T_1^-) \sqcup (S^+ \times T_2^-) \sqcup (S^- \times T_1^+) \sqcup (S^- \times T_2^+).
\end{align*}
\end{construction}
\begin{construction}\label{sij_dist_family}
Let $S$ and $T$ be signed sets and $\{S_t\}_{t \in \bs{T}}$ a family of signed sets. Then, we have a canonical sijection
\[
	S \times \bigsqcup_{t \in T} S_t \sijto \bigsqcup_{t \in T} \left( S \times S_t \right)
\]
which is derived from the identities.
\end{construction}
\begin{construction}\label{du_dist}
Let $T$, $U$ be signed sets and $\{S_v\}_{v \in \bs{T} \sqcup \bs{U}}$ a family of signed sets.
Then, we have a canonical sijection
\[
	\bigsqcup_{t \in T} S_t \sqcup \bigsqcup_{u \in U} S_u \sijto \bigsqcup_{v \in T \sqcup U} S_v,
\]
which is derived from the identities.
\end{construction}
\begin{construction}\label{du_ord}
Let $S_1$ and $S_2$ be signed sets and $\{ T_{s_1,s_2} \}_{s_1 \in \bs{S_1}, s_2 \in \bs{S_2}}$ a family of signed sets.
Then, we have a canonical sijection
\[
	\bigsqcup_{s_1 \in S_1} \bigsqcup_{s_2 \in S_2} T_{s_1,s_2} \sijto \bigsqcup_{s_2 \in S_2} \bigsqcup_{s_1 \in S_1} T_{s_1,s_2},
\]
which is derived from bijections of the form
\[
	\bigsqcup_{a \in A} \bigsqcup_{b \in B} C_{a,b} \to \bigsqcup_{b \in B} \bigsqcup_{a \in A} C_{a,b} \colon ((c,b),a) \mapsto ((c,a),b),
\]
where $A$, $B$ and $C_{a,b}$ (for each $a \in A$ and $b \in B$) are ordinary sets.
\end{construction}
\begin{construction}\label{du_nest}
Let $U$ be a signed set and $\{T_u\}_{u\in\bs{U}}$, $\{S_{t,u}\}_{u \in\bs{U},t \in\bs{T_u}}$ families of signed sets.
Then, we have a canonical sijection
\[
	\bigsqcup_{u \in U} \bigsqcup_{t \in T_u} S_{t,u} \sijto \bigsqcup_{t \in \bigsqcup_{u \in U} T_u} S_{t,u},
\]
which is derived from bijections of the form
\[
	\bigsqcup_{a \in A} \bigsqcup_{b \in B_a} C_{a,b} \to \bigsqcup_{b \in \bigsqcup_{a \in A} B_a} C_{a,b} \colon ((c,b),a) \mapsto (c,(b,a)).
\]
\end{construction}

These constructions can be decomposed to bijections, which are almost identities.
In these cases, we shall identify the domain with the codomain and  use ``$=$'' to describe these sijections.

Lastly, we explain the operator precedence for signed sets.
The operators should be evaluated in the order:
\begin{enumerate}
\item taking the opposite ($-$),
\item the Cartesian products ($\times$),
\item disjoint unions ($\sqcup$).
\end{enumerate}
For example, we can write $S \times T_1 \sqcup S \times T_2$ instead of $(S \times T_1) \sqcup (S \times T_2)$,
but not $S \times T_1 \sqcup T_2$ instead of $S \times (T_1 \sqcup T_2)$.

\section{Compatibility}
We introduce the notion of \textit{compatibility}, the most important idea in this paper.
\begin{definition}
A \textit{statistic} of a signed set is a function on (at least) its support.
Let $S$, $T$ be signed sets and let $\eta$ be a statistic of $S \sqcup T$. 
A sijection $\phi \colon S \sijto T$ is \textit{compatible with $\eta$} if
\[
	\forall s \in \bs{S \sqcup T},\,\eta(\phi(s))=\eta(s).
\]
\end{definition}
The codomain of statistics does not matter, so we will not pay attention to it.
For simplicity, we will sometimes write $ s \insupp  S $ instead of $s \in \bs{S}$.

Compatibility is a generalization of the notion of \textit{normality} which is used in \cite{FK1}.
For example, let us consider the sijection $\varphi$ in Example \ref{ex_sij}, which is normal in the language of \cite{FK1}.
We can define a canonical $\mathbb{Z}$-valued statistic of $\sint{a}{b} \sqcup (\sint{a}{c} \sqcup \sint{c}{b})$,
because we can recognize each element of the support of  $\sint{a}{b} \sqcup (\sint{a}{c} \sqcup \sint{c}{b})$ as an integer in the canonical way.
The sijection $\varphi$ is compatible with this statistic.
We call statistics of this type \textit{normal} statistics.
The precise definition is given in Definition \ref{defofnormal}.

Compatibility is preserved under compositions and so is normality.
\begin{lemma}\label{lem:comp_comp}
Let $\phi \colon S \sijto T$, $\psi \colon T \sijto U$ be sijections and $\eta$ a statistic of $S \sqcup T \sqcup U$.
If $\phi$ and $\psi$ are compatible with $\eta$, then $\psi \circ \phi$ is compatible with $\eta$.
\end{lemma}
\begin{proof}
For any $s \insupp S$, $\eta(s)=\eta(\phi(s))=\eta(\psi(\phi(s)))=\cdots$ while they are well-defined.
Therefore, $\eta(s) = \eta(\psi \circ \phi(s))$.
The same is true for an element of $\bs{U}$.
\end{proof}

As well as for compositions, compatibility is preserved under disjoint unions and Cartesian products.
\begin{lemma}\label{lemofsqcup}
Let $ \varphi_i \colon S_i \sijto T_i $ be a sijection compatible with a statistic $\eta_i$ for $i=1,2$.
When we define a statistic $\eta=\eta_1\sqcup\eta_2$ by
\begin{align*}
	\eta(u)=(\eta_1\sqcup\eta_2)(u)=\begin{cases}
		\eta_1(u) & (u \insupp {S_1 \sqcup T_1}) \\
		\eta_2(u) & (u \insupp {S_2 \sqcup T_2}), \end{cases}
\end{align*}
$\varphi_1 \sqcup \varphi_2$ is compatible with $\eta_1\sqcup\eta_2$.
\end{lemma}
\begin{lemma}\label{lemoftimes}
Let $ \varphi_i \colon S_i \sijto T_i $ be a sijection compatible with a statistic $\eta_i$ for $i=1,2$.
When we define a statistic $\eta=\eta_1\times\eta_2$ by
\begin{align*}
	\eta((u_1,u_2))=(\eta_1\times\eta_2)(u)=(\eta_1(u_1),\eta_2(u_2)),
\end{align*}
$\varphi_1 \times \varphi_2$ is compatible with $\eta_1\times \eta_2$.
\end{lemma}
By these definitions, a disjoint union and the Cartesian product of more than two sijections can be treated in the same manner.
Furthermore, in the situation of Definition \ref{def_sij_DIS}, if for all $t \in T^+ \sqcup \tilde{T}^-$, a sijection $\varphi_t$ is 
 compatible with some statistic $\eta$,
then $\bigsqcup_{t \in T \sqcup \tilde{T}} \varphi_t$ is also compatible with $\eta$.
Proofs can be given either by considering the graphs of sijections or by using the definitions.

In addition, if we have a sijction $\psi \colon T \sijto \tilde{T}$ compatible with some statistic $\eta$ and a family of signed sets $\{ S_a \}_{a \in \eta\left(\bs{T} \sqcup \bs{\tilde{T}}\right)}$
then we have, by applying the Definition \ref{def_sij_DIS},  a sijection
\begin{equation} \label{du_id}
\bigsqcup_{t \in T \sqcup \tilde{T}} \text{id}_{S_\eta(t)} \colon \bigsqcup_{t \in T} S_{\eta(t)} \sijto \bigsqcup_{t \in \tilde{T}} S_{\eta(t)}.
\end{equation}
Furthermore, this sijection is compatible with the `forgetful' statistic:
\[\eta_f \colon \bs{ \bigsqcup_{t \in T} S_{\eta(t)} \sqcup \bigsqcup_{t \in \tilde{T}} S_{\eta(t)}} \to \bigcup_{a \in \eta(\bs{T} \sqcup \bs{\tilde{T}})} \bs{S_a} \colon 
\begin{cases}
((s,t),0) \mapsto (s,\eta(t)) \\
((s,t),1) \mapsto (s,\eta(t)).
\end{cases}\]
\begin{definition}\label{defofnormal}
A \textit{normal} signed set is a signed interval or a signed set that is made by finite operations of disjoint union or Cartesian product of signed intervals.
For a signed interval, its \textit{normal statistic} is the inclusion map to $\mathbb{Z}$.
For two pairs of normal signed sets and their normal statistics $(S_1,\eta_1)$, $(S_2,\eta_2)$,
we define the normal statistic of $S_1 \sqcup S_2$ by $\eta_1 \sqcup \eta_2$ as in Lemma \ref{lemofsqcup} and
the normal statistic of $S_1 \times S_2$ by $\eta_1 \times \eta_2$ as in Lemma \ref{lemoftimes}.
\end{definition}
\begin{remark}
These definitions are similar to those of elementary signed sets and normality in \cite{FK1}.
\end{remark}

We prepare one notation for later use.
\begin{definition}
Let $S$ be a signed set and $\eta$ a statistic of $S$. For an element $a \in \eta(\bs{S})$, we define the restriction of $S$ as follows:
\[
	S \mid_{\eta=a} \,= \left(\{ s \in S^+ \mid \eta(s)=a\}, \{ s \in S^- \mid \eta(s)=a\}\right).
\]
\end{definition}

Let us explain the motivation for introducing the notion of compatibility.
In bijective combinatorics, we like to think that the more natural a bijection (or a sijection for that matter), the better. 
However, whether one particular bijection (or sijection) is more natural than another can be controversial.
The notion of compatibility is intended to offer an answer to this problem.
More specifically, a bijection (or a sijection) is considered to be `more natural' if it is compatible with a finer statistic, where the size of the image of a statistic determines how fine it is.
Unfortunately however, for any given sijection $\Gamma \colon S \sijto T$, we can always construct one of the finest possible statistics which, at the same time, is also entirely trivial: $\eta \colon \bs{S \sqcup T} \to S^{+} \sqcup T^{-}$ such that
\begin{gather*}
	\forall s \in S^{+} \sqcup T^{-},\, \eta(s)=s, \\
	\forall s \in S^{-} \sqcup T^{+},\, \eta(s)=\Gamma(s).
\end{gather*}
Nevertheless, we believe that the notion of compatibility is useful because it allows us to translate a sense of `naturalness' (as opposed to `contrivedness') of a sijection into the naturalness of a statistic.
When a `natural looking' sijection is given, we can explain why it can truly be regarded as natural using the notion of compatibility.
This process sometimes leads us to novel combinatorial structures.
Our first application of compatibility concerns Gelfand-Tsetlin patterns.
It is an example of this type of application of compatibility.
For more details, see Section \ref{secofGT}.

In another interpretation, a construction of a sijection with compatibility gives a bijective proof of a refined enumeration.
In this context, first, a seemingly natural statistic is given according to a computational proof of a refined enumeration.
After that, we can try to construct a sijection compatible with that statistic.
In fact, the existence is guaranteed through the computational proof but we need to find an algorithm to construct it.
Our second application of compatibility is of this type.
We define statistics on alternating sign matrices and shifted Gelfand-Tsetlin patterns by referring to refined enumerations for alternating sign matrices \cite{BFZ1,BFZ2},
and construct a sijection which is compatible with the statistics.
The construction is similar to Fischer and Konvalinka's work \cite{FK1}, but our construction is more elementary.
For more details, see Section \ref{secofSGT}.

In the first type of application we make a statistic referring to a sijection and vice versa in the second type of application.
Either way, we have to obtain a pair of a natural looking statistic and a natural sijection compatible with the statistic to state that they are truly natural in the sense explained above.

\section{
About Gelfand-Tsetlin patterns}\label{secofGT}
\subsection{Definitions}
A Gelfand-Tsetlin pattern is a triangular array $(a_{i,j})_{1 \leq j \leq i \leq n}$ of integers such that for any $1 \leq j \leq i \leq n-1$, it holds that $a_{i+1,j} \leq a_{i,j} \leq a_{i+1,j+1}$.
When the elements are arranged as follows
\[
\begin{array}{ccccccccc}
&&&&a_{1,1}&&&&\\
&&&a_{2,1}&&a_{2,2}&&&\\
&&a_{3,1}&&a_{3,2}&&a_{3,3}&&\\
&\iddots&&\vdots&&\vdots&&\ddots&\\
a_{n,1}&\cdots&\cdots&\cdots&\cdots&\cdots&\cdots&\cdots&a_{n,n},
\end{array}
\]
each element is greater than or equal to the element left-below and is less than or equal to the element right-below.
We call the sequence $(a_{n,1},a_{n,2},\ldots,a_{n,n})$ the bottom row of a Gelfand-Tsetlin pattern $(a_{i,j})_{ij}$.
For reasons that will become clear afterwards, we shall modify this definition a bit.
We impose the condition
\begin{equation}
\label{condition_GT}
	a_{i+1,j} \leq a_{i,j} < a_{i+1,j+1},\quad \text{for all $1 \leq j \leq i \leq n-1$,}
\end{equation}
instead of $a_{i+1,j} \leq a_{i,j} \leq a_{i+1,j+1}$.
We call triangular arrays subject to this new conditions \textit{modified Gelfand-Tsetlin patterns}.
Clearly, Gelfand-Tsetlin patterns with bottom row $(k_1,k_2,\ldots,k_n)$ are in one-to-one correspondence with modified Gelfand-Tsetlin patterns with bottom row $(k_1,k_2+1,\ldots,k_n+n-1)$.
Indeed, a Gelfand-Tsetlin pattern $(a_{i,j})_{ij}$ corresponds to a modified Gelfand-Tsetlin pattern $(a_{i,j}+j-1)_{ij}$.
For example,
\[
	\begin{array}{ccccc}&&3&&\\&1&&3&\\1&&3&&4\end{array} \mapsto \begin{array}{ccccc}&&3&&\\&1&&4&\\1&&4&&6.\end{array}
\]

A generalization of Gelfand-Tsetlin patterns for bottom rows that are not monotonically increasing is introduced in \cite{FK1}.
The signed set of the generalization is defined as follows.
\begin{definition}[Definition 5 in \cite{FK1}]
For $k \in \mathbb{Z}$, define $\textrm{GT}^\textrm{FK}(k)=(\{\cdot\},\emptyset)$, and for $\mathbf{k}=(k_1,k_2,\ldots,k_n) \in \mathbb{Z}^n$, $n \geq 2$,
define $\textrm{GT}^\textrm{FK}(\mathbf{k})$ recursively as
\[
	\textrm{GT}^\textrm{FK}(\mathbf{k})=\textrm{GT}^\textrm{FK}(k_1,k_2,\ldots,k_n)=\bigsqcup_{\mathbf{l} \in \underline{[k_1,k_2]}\times\cdots\times\underline{[k_{n-1},k_n]}} \textrm{GT}^\textrm{FK}(\mathbf{l}).
\]
\end{definition}
Here, $\underline{[a,b]}$ is a variant of a signed interval and is equal to $\underline{[a,b+1)}$. (Please pay attention to the closing parenthesis.)
As well as this definition, the signed set of modified Gelfand-Tsetlin patterns is defined as follows.
\begin{definition} \label{def-GT}
For $k \in \mathbb{Z}$, define $\textrm{GT}(k)=(\{\cdot\},\emptyset)$, and for $\mathbf{k}=(k_1,k_2,\ldots,k_n) \in \mathbb{Z}^n$, $n \geq 2$,
define $\textrm{GT}(\mathbf{k})$ recursively as
\[
	\textrm{GT}(\mathbf{k})=\textrm{GT}(k_1,k_2,\ldots,k_n)=\bigsqcup_{\mathbf{l} \in \underline{[k_1,k_2)}\times\cdots\times\underline{[k_{n-1},k_n)}} \textrm{GT}(\mathbf{l}).
\]
\end{definition}
The only difference between the two definitions is the closing parenthesis.
As well as in the unsigned case,
for $\mathbf{k} = (k_1,k_2,\ldots,k_n) \in \mathbb{Z}^n$ and $\mathbf{k}' = (k_1,k_2+1,\ldots,k_n+n-1)$, two signed sets $\textrm{GT}^\textrm{FK}(\mathbf{k})$ and $\textrm{GT}(\mathbf{k}')$ are essentially the same.
In fact, we can make $$\mathbf{l} = (l_1,l_2,\ldots,l_{n-1}) \in \underline{[k_1,k_2]}\times\cdots\times\underline{[k_{n-1},k_n]}$$ correspond to 
$$\mathbf{l'} = (l_1,l_2+1,\ldots,l_{n-1}+n-2) \in \underline{[k_1,k_2+1)}\times\cdots\times\underline{[k_{n-1}+n-2,k_n+n-1)}.$$
Therefore, we can check recursively that $\# \textrm{GT}^\textrm{FK}(\mathbf{k})^+ = \# \textrm{GT}(\mathbf{k}')^+$
and that $\# \textrm{GT}^\textrm{FK}(\mathbf{k})^- = \# \textrm{GT}(\mathbf{k}')^- $.

\begin{example}
For example, $((2,(1,3)),(3,1,4))$ is an element of $\text{GT}((7,1,3,5))$ because
\[
	2 \insupp \sint{1}{3},\quad (1,3) \insupp \sint{3}{1} \times \sint{1}{4},\quad (3,1,4) \insupp \sint{7}{1} \times \sint{1}{3} \times \sint{3}{5}.
\]
This sign is plus because there is an even number of ` $>$' in the following:
\[
	1<3,\quad 3>1<4,\quad 7>1<3<5.
\]
We can interpret an element of $\text{GT}(\mathbf{k})$ (or $\text{GT}^\textrm{FK}(\mathbf{k})$) as a triangular array.
For the above example, we can draw it as
\[
	\begin{array}{ccccccc}&&&2&&&\\&&1&&3&&\\&3&&1&&4&\\7&&1&&3&&5.\end{array}
\]
Accordingly, it is clear that the formal definition of $\text{GT}(\mathbf{k})$ is indeed identical to the triangular array-based definition of modified Gelfand-Tsetlin patterns when the bottom row $\mathbf{k}$ is monotonically increasing.
From this correspondence, we call $\mathbf{l}$ the second bottom row of $T = (T', \mathbf{l}) \in \textrm{GT}(\mathbf{k})$.
For instance, the second bottom row of our running example is $(3,1,4)$.
\end{example}

\subsection{
Known sijections relevant to Gelfand-Tsetlin patterns
}
Fischer and Konvalinka construct some sijections relevant to Gelfand-Tsetlin pattern in \cite{FK1}.
In this subsection, we explain these known results.
Using our notations, it is easy to describe the results and the constructions become transparent,
which is why we choose to give considerable details for them although they are essentially the same as Fischer and Konvalinka's work.
\begin{construction}[Problem 4 in \cite{FK1}]\label{rho}
Let $\mathbf{a}=(a_1,a_2,\ldots,a_n) \in \mathbb{Z}^n$, $\mathbf{b}=(b_1,b_2,\ldots,b_n) \in \mathbb{Z}^n$ and $x \in \mathbb{Z}$.
We construct a sijection
\[
	\rho=\rho_{\mathbf{a},\mathbf{b},x} \colon \bigsqcup_{\mathbf{l} \in \sint{a_1}{b_1}\times\sint{a_2}{b_2}\times\cdots\times\sint{a_n}{b_n}}\text{GT}(\mathbf{l})
	\sijto \bigsqcup_{\mathbf{m} \in S_1 \times S_2 \times\cdots\times S_n} \text{GT}(\mathbf{m},x),
\]
where $S_i=(\{a_i\},\{b_i\})$, $\mathbf{m}=(m_1,m_2,\ldots,m_n)$ and $\text{GT}(\mathbf{m},x)=\text{GT}(m_1,m_2,\ldots,m_n,x)$.

First, we construct the following sijection.
\begin{subconstruction}[Problem 2. in \cite{FK1}]\label{beta}
We construct a sijection
\begin{multline*}
	\beta=\beta_{\mathbf{a},\mathbf{b},x} \colon  \sint{a_1}{b_1}\times\sint{a_2}{b_2}\times\cdots\times\sint{a_n}{b_n} \\
	\sijto \bigsqcup_{\mathbf{m} \in S_1 \times S_2 \times\cdots\times S_n} \sint{m_1}{m_2}\times\sint{m_2}{m_3}\times\sint{m_{n-1}}{m_n}\times\sint{m_n}{x},
\end{multline*}
which is compatible with the normal statistic.
The construction is by induction on $n$. If $n=1$, it coincides with Example \ref{ex_sij}. In fact, we have $\sint{a_1}{b_1} \sijto \sint{a_1}{x} \sqcup \sint{x}{b_1} = \sint{a_1}{x} \sqcup -\sint{b_1}{x}=\bigsqcup_{m_1 \in S_1} \sint{m_1}{m_2}$.
If $n>1$, we have, by induction,
\begin{align*}
	\sint{a_1}{b_1}&\times\sint{a_2}{b_2}\times\cdots\times\sint{a_n}{b_n} \\
	&\overset{ind.}{\quad\sijto\quad} \sint{a_1}{b_1}\times \bigsqcup_{\mathbf{(m_2,m_3,\ldots,m_n)} \in S_2 \times S_3 \times\cdots\times S_n} \sint{m_2}{m_3}\times\sint{m_{n-1}}{m_n}\times\sint{m_n}{x} \\
	&\overset{C. \ref{sij_dist_family}}{\quad\sijto\quad} \bigsqcup_{\mathbf{(m_2,m_3,\ldots,m_n)} \in S_2 \times S_3 \times\cdots\times S_n}\sint{a_1}{b_1}\times  \sint{m_2}{m_3}\times\sint{m_{n-1}}{m_n}\times\sint{m_n}{x}.
\end{align*}
Here, we have 
$\sint{a_1}{b_1} \sijto \bigsqcup_{m_1 \in S_1} \sint{m_1}{m_2}$ by substituting $m_2$ for $x$ in the result of base step,
then the construction is completed. The compatibility is induced in the obvious way.
\qed
\end{subconstruction}
We return to the construction of $\rho$.
By using the subconstruction and Construction \ref{du_nest}, and considering a disjoint union with signed index of $\text{id}_{\text{GT}(\mathbf{l})}$ (cf. (\ref{du_id})), we obtain a sijection
\begin{align*}
	\bigsqcup_{\mathbf{l} \in \sint{a_1}{b_1}\times\sint{a_2}{b_2}\times\cdots\times\sint{a_n}{b_n}}\text{GT}(\mathbf{l})
	\overset{\sqcup\text{id}_{\text{GT}(\mathbf{l})}}{\quad\sijto\quad}& \bigsqcup_{\mathbf{l} \in \bigsqcup_{\mathbf{m} \in S_1 \times S_2 \times\cdots\times S_n} \sint{m_1}{m_2}\times\sint{m_2}{m_3}\times\sint{m_{n-1}}{m_n}\times\sint{m_n}{x}}\text{GT}(\mathbf{l}) \\
	\overset{E. \ref{du_nest}}{\quad=\quad}& \bigsqcup_{\mathbf{m} \in S_1 \times S_2 \times\cdots\times S_n} \left(\bigsqcup_{ \mathbf{l} \in \sint{m_1}{m_2}\times\sint{m_2}{m_3}\times\sint{m_{n-1}}{m_n}\times\sint{m_n}{x}} \text{GT}(\mathbf{l}) \right) \\
	\overset{D. \ref{def-GT}}{\quad=\quad}& \bigsqcup_{\mathbf{m} \in S_1 \times S_2 \times\cdots\times S_n} \text{GT}(\mathbf{m},x),
\end{align*}
which completes the construction. \qed
\end{construction}

\begin{construction}[Problem 5 in \cite{FK1}]\label{pi}
We construct the following sijections:
\begin{enumerate}
\item for any $\mathbf{k}=(k_1,k_2,\ldots,k_n) \in \mathbb{Z}^n$ and $i \in \{1,2,\ldots,n-1\}$,
\[
	\pi=\pi_{\mathbf{k},i} \colon \text{GT}(k_1,k_2,\ldots,k_n) \sijto -\text{GT}(k_1,k_2,\ldots,k_{i-1},k_{i+1},k_i,k_{i+2},\ldots,k_n),
\]
\item for any $\mathbf{a}=(a_1,a_2,\ldots,a_n), \mathbf{b}=(b_1,b_2,\ldots,b_n) \in \mathbb{Z}^n$ satisfying $a_i=b_i$ and $a_{i+1}=b_{i+1}$ for some $i \in \{1,2,\ldots,n-1\}$,
\[
	\sigma=\sigma_{\mathbf{a},\mathbf{b},i} \colon \bigsqcup_{\mathbf{l} \in \sint{a_1}{b_1} \times \sint{a_2}{b_2} \times \cdots \times \sint{a_n}{b_n}} \text{GT}(\mathbf{l}) \sijto (\emptyset,\emptyset).
\]
\end{enumerate}

The construction is by induction on $n$. 
First, we construct $\sigma$ from $\pi$ for the same $n$ as follows:
\begin{align*}
	&\hspace{-4em}\bigsqcup_{\mathbf{l} \in \sint{a_1}{b_1} \times \sint{a_2}{b_2} \times \cdots \times \sint{a_n}{b_n}} \text{GT}(\mathbf{l}) \\
	\hspace{4em}\quad=\quad&\begin{multlined}[t]
	\bigsqcup_{\mathbf{l} \in \sint{a_1}{b_1} \times \sint{a_2}{b_2} \times \cdots \times \sint{a_n}{b_n} \text{ s.t. } l_i < l_{i+1}} \text{GT}(\mathbf{l}) \\
	\sqcup \bigsqcup_{\mathbf{l} \in \sint{a_1}{b_1} \times \sint{a_2}{b_2} \times \cdots \times \sint{a_n}{b_n} \text{ s.t. }  l_i > l_{i+1}} \text{GT}(\mathbf{l})
	\sqcup \bigsqcup_{\mathbf{l} \in \sint{a_1}{b_1} \times \sint{a_2}{b_2} \times \cdots \times \sint{a_n}{b_n} \text{ s.t. }  l_i = l_{i+1}} \text{GT}(\mathbf{l})
	\end{multlined} \\
	\hspace{4em}\overset{\textrm{id} \sqcup (\sqcup\pi) \sqcup =}{\sijto}&
	\begin{multlined}[t] \bigsqcup_{\mathbf{l} \in \sint{a_1}{b_1} \times \sint{a_2}{b_2} \times \cdots \times \sint{a_n}{b_n} \text{ s.t. }  l_i < l_{i+1}} \text{GT}(\mathbf{l}) \\
	\sqcup \bigsqcup_{\mathbf{l} \in \sint{a_1}{b_1} \times \sint{a_2}{b_2} \times \cdots \times \sint{a_n}{b_n} \text{ s.t. }  l_i < l_{i+1}} -\text{GT}(\mathbf{l}) \sqcup (\emptyset,\emptyset) \end{multlined}\\
	\hspace{4em}\quad=\quad& \bigsqcup_{\mathbf{l} \in \sint{a_1}{b_1} \times \sint{a_2}{b_2} \times \cdots \times \sint{a_n}{b_n} \text{ s.t. }  l_i < l_{i+1}}
	(\text{GT}(\mathbf{l}) \sqcup -\text{GT}(\mathbf{l})) \\ \hspace{4em}\overset{\text{E. \ref{sij_oppo}}}{\quad\sijto\quad}& (\emptyset,\emptyset).
\end{align*}

On the other hand, we can construct $\pi$ from $\sigma$ for one smaller $n$.
Note that when $n=1$ there is nothing to prove.
Because the construction is too long, we shall describe it in Appendix. We only describe the proof for the case $n=5$ and $i=3$.
For $x_1,x_2,x_3,x_4 \in \mathbb{Z}$, we have sijections compatible with the normal statistics
\begin{align*}
	\sint{x_1}{x_2} \times \sint{x_2}{x_3} \times \sint{x_3}{x_4}
	&\sijto (\sint{x_1}{x_3} \sqcup \sint{x_3}{x_2}) \times -\sint{x_3}{x_2} \times (\sint{x_3}{x_2} \sqcup \sint{x_2}{x_4}) \\
	&\sijto -\sint{x_1}{x_3} \times \sint{x_3}{x_2} \times \sint {x_3}{x_2} \sqcup -\sint{x_1}{x_3} \times \sint{x_3}{x_2} \times \sint {x_2}{x_4} \\
	&\qquad \qquad \sqcup -\sint{x_3}{x_2} \times \sint{x_3}{x_2} \times \sint {x_3}{x_2} \sqcup -\sint{x_3}{x_2} \times \sint{x_3}{x_2} \times \sint {x_2}{x_4}.
\end{align*}
Then, we have
\begin{align*}
	\text{GT}(\mathbf{l}) =& \bigsqcup_{\mathbf{m} \in \sint{l_1}{l_2} \times \sint{l_2}{l_3} \times \sint{l_3}{l_4} \times \sint{l_4}{l_5}} \text{GT}(\mathbf{m}) \\
	\overset{C. \ref{du_dist}}{\ \sijto\ }& \bigsqcup_{\mathbf{m} \in -\sint{l_1}{l_2} \times \sint{l_2}{l_4} \times \sint{l_4}{l_3} \times \sint{l_4}{l_3}} \text{GT}(\mathbf{m})
	\sqcup \bigsqcup_{\mathbf{m} \in -\sint{l_1}{l_2} \times \sint{l_2}{l_4} \times \sint{l_4}{l_3} \times \sint{l_3}{l_5}} \text{GT}(\mathbf{m}) \\
	&\qquad \qquad \sqcup \bigsqcup_{\mathbf{m} \in -\sint{l_1}{l_2} \times \sint{l_4}{l_3} \times \sint{l_4}{l_3} \times \sint{l_4}{l_3}} \text{GT}(\mathbf{m})
	\sqcup \bigsqcup_{\mathbf{m} \in -\sint{l_1}{l_2} \times \sint{l_4}{l_3} \times \sint{l_4}{l_3} \times \sint{l_3}{l_5}} \text{GT}(\mathbf{m}) \\
	\overset{\sigma}{\ \sijto\ }& (\emptyset,\emptyset) \sqcup -\bigsqcup_{\mathbf{m} \in \sint{l_1}{l_2} \times \sint{l_2}{l_4} \times \sint{l_4}{l_3} \times \sint{l_3}{l_5}} \text{GT}(\mathbf{m})
	\sqcup (\emptyset,\emptyset) \sqcup (\emptyset,\emptyset) \\
	=& -\bigsqcup_{\mathbf{m} \in \sint{l_1}{l_2} \times \sint{l_2}{l_4} \times \sint{l_4}{l_3} \times \sint{l_3}{l_5}} \text{GT}(\mathbf{m}) =-\text{GT}(l_1,l_2,l_4,l_3,l_5). \qed
\end{align*}
\end{construction}

\begin{construction}[Problem 6 in \cite{FK1}]\label{tau}
Let $\mathbf{k}=(k_1,k_2,\ldots,k_n) \in \mathbb{Z}^n$ and $x \in \mathbb{Z}$. We construct a sijection
\[
	\tau=\tau_{\mathbf{k},x} \colon \text{GT}(k_1,k_2,\ldots,k_n)
	\sijto \bigsqcup_{i=1}^n \text{GT}(k_1,k_2,\ldots,k_{i-1},x,k_{i+1},\ldots,k_n).
\]
When $n=1$, it is trivial. In the following we assume that $n \geq 2$.
First, we construct the following sijection.
\begin{subconstruction}[Problem 3. in \cite{FK1}]\label{gamma}
We construct a sijection compatible with the normal statistic
\begin{multline*}
	\gamma=\gamma_{\mathbf{k},x} \colon \sint{k_1}{k_2} \times \sint{k_2}{k_3} \times \cdots \times \sint{k_{n-1}}{k_n} \\
	\sijto \bigsqcup_{i=1}^n \sint{k_1}{k_2} \times \sint{k_2}{k_3} \times \cdots \times \sint{k_{i-2}}{k_{i-1}} \times \sint{k_{i-1}}{x} \times \sint{x}{k_{i+1}} \times \cdots \times \sint{k_{n-1}}{k_n} \\
	\sqcup \bigsqcup_{i=1}^{n-2} \sint{k_1}{k_2} 
	\times \cdots \times \sint{k_{i-1}}{k_i} \times \sint{k_{i+1}}{x} \times \sint{k_{i+1}}{x} \times \sint{k_{i+2}}{k_{i+3}} \times \cdots \times \sint{k_{n-1}}{k_n},
\end{multline*}
where we replaced $k_i$ with $x$ in the first term and $\sint{k_i}{k_{i+1}}$ and $\sint{k_{i+1}}{k_{i+2}}$ with $\sint{k_{i+1}}{x}$ in the second term.
The construction is by induction on $n$. If $n=2$, it coincides with Example \ref{ex_sij}.
For $n>2$, we shall calculate the difference of the right hand side from that with 1 smaller $n$ multiplied by $\sint{k_{n-1}}{k_n}$. First, we have
\begin{multline*}
	\bigsqcup_{i=1}^n \sint{k_1}{k_2} \times \sint{k_2}{k_3} \times \cdots \times \sint{k_{i-2}}{k_{i-1}} \times \sint{k_{i-1}}{x} \times \sint{x}{k_{i+1}} \times \cdots \times \sint{k_{n-1}}{k_n} \\
	\sqcup -\left(\bigsqcup_{i=1}^{n-1} \sint{k_1}{k_2} \times \sint{k_2}{k_3} \times \cdots \times \sint{k_{i-2}}{k_{i-1}} \times \sint{k_{i-1}}{x} \times \sint{x}{k_{i+1}} \times \cdots \times \sint{k_{n-2}}{k_{n-1}}\right) \times \sint{k_{n-1}}{k_n} \\
	\sijto \left(\prod_{i=1}^{n-3} \sint{k_i}{k_{i+1}}\right) \times \left( \sint{k_{n-2}}{x} \times \sint{x}{k_n} \sqcup \sint{k_{n-2}}{k_{n-1}} \times \sint{k_{n-1}}{x} \sqcup -\sint{k_{n-2}}{x} \times \sint{k_{n-1}}{k_n}\right)
\end{multline*}
and
\begin{multline*}
\bigsqcup_{i=1}^{n-2} \sint{k_1}{k_2} \times \cdots \times \sint{k_{i-1}}{k_i} \times \sint{k_{i+1}}{x} \times \sint{k_{i+1}}{x} \times \sint{k_{i+2}}{k_{i+3}} \times \cdots \times \sint{k_{n-1}}{k_n} \\
\sqcup -\left(\bigsqcup_{i=1}^{n-3} \sint{k_1}{k_2} \times \cdots \times \sint{k_{i-1}}{k_i} \times \sint{k_{i+1}}{x} \times \sint{k_{i+1}}{x} \times \sint{k_{i+2}}{k_{i+3}} \times \cdots \times \sint{k_{n-2}}{k_{n-1}} \right) \times \sint{k_{n-1}}{k_n} \\
\sijto \left(\prod_{i=1}^{n-3} \sint{k_i}{k_{i+1}}\right) \times \sint{k_{n-1}}{x} \times \sint{k_{n-1}}{x}.
\end{multline*}
Here, we have a sijection
\begin{multline*}
\sint{k_{n-2}}{x} \times \sint{x}{k_n} \sqcup \sint{k_{n-2}}{k_{n-1}} \times \sint{k_{n-1}}{x} \sqcup \sint{k_{n-1}}{x} \times \sint{k_{n-1}}{x} \\
\sijto \sint{k_{n-2}}{x} \times \sint{x}{k_n} \sqcup \sint{k_{n-2}}{x} \times \sint{k_{n-1}}{x} 
\sijto \sint{k_{n-2}}{x} \times \sint{k_{n-1}}{k_n},
\end{multline*}
which means $\sint{k_{n-2}}{x} \times \sint{x}{k_n} \sqcup \sint{k_{n-2}}{k_{n-1}} \times \sint{k_{n-1}}{x} \sqcup -\sint{k_{n-2}}{x} \times \sint{k_{n-1}}{k_n} \sqcup \sint{k_{n-1}}{x} \times \sint{k_{n-1}}{x}$. Therefore, the difference is in sijection to $(\emptyset,\emptyset)$.
By Example \ref{sij_oppo}, we thus obtain the result. The compatibility can be checked step by step.
\qed
\end{subconstruction}
We return to the construction of $\tau$. By using Construction \ref{du_dist}, the sijection $\gamma$ and the sijection $\sigma$ constructed in Construction \ref{pi}, we have
\begin{align*}
	\text{GT}(\mathbf{k})\ =\ &\bigsqcup_{\mathbf{l} \in \underline{[k_1,k_2)}\times\sint{k_2}{k_3} \times\cdots\times\underline{[k_{n-1},k_n)}} \text{GT}(\mathbf{l}) \\
	\overset{\sqcup\text{id}_{\text{GT}(\mathbf{l})}}{\quad\sijto\quad}& \bigsqcup_{\mathbf{l} \in \bigsqcup_{i=1}^n \sint{k_1}{k_2} \times \sint{k_2}{k_3} \times \cdots \times \sint{k_{i-2}}{k_{i-1}} \times \sint{k_{i-1}}{x} \times \sint{x}{k_{i+1}} \times \cdots \times \sint{k_{n-1}}{k_n} } \text{GT}(\mathbf{l}) \\
	&\qquad \sqcup \bigsqcup_{\mathbf{l} \in \bigsqcup_{i=1}^{n-2} \sint{k_1}{k_2} \times \cdots \times \sint{k_{i-1}}{k_i} \times \sint{k_{i+1}}{x} \times \sint{k_{i+1}}{x} \times \sint{k_{i+2}}{k_{i+3}} \times \cdots \times \sint{k_{n-1}}{k_n}} \text{GT}(\mathbf{l}) \\
	\overset{(C. \ref{du_dist}) \sqcup \sigma}{\quad\sijto\quad}& \bigsqcup_{i=1}^n \text{GT}(k_1,k_2,\ldots,k_{i-1},x,k_{i+1},\ldots,k_n) \sqcup (\emptyset,\emptyset) \\
	\ =\ & \bigsqcup_{i=1}^n  \text{GT}(k_1,k_2,\ldots,k_{i-1},x,k_{i+1},\ldots,k_n).
\end{align*}
The construction is thus completed. \qed
\end{construction}

\subsection{
New compatibility properties
}
We define a statistic $\eta_\text{row}$ of Gelfand-Tsetlin patterns.
First, we define a set of multisets
\[
	\text{MS}(\mathbb{Z},n)  = \{ \{\!\{k_1,k_2,\ldots,k_n\}\!\} \mid k_1 \leq k_2 \leq \cdots \leq k_n \}
	\simeq \mathbb{Z}^n / \mathfrak{S}_n,
\]
where we use double braces to describe a multiset. We define $\mathcal{A}_\text{row}$ by
\begin{gather*}
	\mathcal{A}_\text{row} = \bigcup_{i=1}^\infty \mathcal{A}_\text{row}^i,\qquad
	\mathcal{A}_\text{row}^i = \text{MS}(\mathbb{Z},1) \times \text{MS}(\mathbb{Z},2) \times \cdots \times \text{MS}(\mathbb{Z},i).
\end{gather*}
Then, we define $\eta_\text{row}$ as follows.
\begin{itemize}
\item For $k \in \text{GT}(k)$, $\eta_\text{row}(k) = \{\!\{k\}\!\} \in\text{MS}(\mathbb{Z},1) = \mathcal{A}_\text{row}^1$.
\item For $T \in \text{GT}(\mathbf{k})$, where $\mathbf{k} \in \mathbb{Z}^n$ and $n \geq 2$, we have $\mathbf{l} \in \mathbb{Z}^{n-1}$ and $T' \in \text{GT}(\mathbf{l})$ such that $T=(T',\mathbf{l})$.
Then, we define \[\eta_\text{row}(T) = ( \eta_\text{row}(T'), \{\!\{k_1,k_2,\ldots,k_n\}\!\}) \in \mathcal{A}_\text{row}^{n-1} \times \text{MS}(\mathbb{Z},n) = \mathcal{A}_\text{row}^{n}.\]
\end{itemize}
For example, let $T = (4,(2,5)) \in \textrm{GT}((2,5,7))$, then
\[
	\eta_\text{row}\left( T \right) = \left(\{\!\{4\}\!\},\{\!\{2,5\}\!\},\{\!\{2,5,7\}\!\}\right).
\]
The following theorem is the main result of this section.
\begin{theorem}\label{thm::GT}
The sijections $\pi$ and $\sigma$ in Construction \ref{pi} are compatible with the statistic $\eta_\text{row}$.
\end{theorem}
\begin{proof}
The proof is by tracing the constructions.
First, the compatibility of $\sigma$ follows from the compatibility of $\pi$.
Next, the compatibility of $\pi$ follows from the compatibility of $\sigma$ without the bottom row.
The sijection $\pi$ acts on the bottom row as swapping an element with the next element, so it does not affect the value of $\eta_\text{row}$.
Therefore, $\pi$ and $\sigma$ are compatible with the statistic $\eta_\text{row}$.
\end{proof}
The theorem means that the sijection $\pi$ acts on $\text{GT}(\mathbf{k})$ as a permutation of elements in each row.
This result is one of the most important reasons why we use half-open intervals to describe signed intervals and Gelfand-Tsetlin patterns.
Based on the theorem, we can decompose the equinumerous relation of Gelfand-Tsetlin patterns by using the statistic $\eta_\text{row}$.
Namely,
we have
\[
	\text{GT}(k_1,k_2,\ldots,k_n) \mid_{\eta_\text{row}=A} \sijto -\text{GT}(k_1,k_2,\ldots,k_{i-1},k_{i+1},k_i,k_{i+2},\ldots,k_n) \mid_{\eta_\text{row}=A}
\]
from the compatibility with $\eta_\text{row}$ of the sijection $\pi$.
When $\mathbf{l}$ is a permutation of $\mathbf{k}$, we have $\#\text{GT}(\mathbf{k})\mid_{\eta_\text{row}=A} = \text{sgn}\, s \cdot \#\text{GT}(\mathbf{l})\mid_{\eta_\text{row}=A}$,
where $s$ is a permutation of $\{1,2,\ldots,n\}$ such that $l_i = k_{s(i)}$ and $\text{sgn}\, s$ means the sign of $s$.
Let $(a_{i,j})_{1 \leq j \leq i \leq n}$ be the elements of $A$ as
\[
	A = (\{\!\{a_{1,1}\}\!\},\{\!\{a_{2,1},a_{2,2}\}\!\},\ldots,\{\!\{ a_{n,1},a_{n,2},\ldots,a_{n,n} \}\!\}), \qquad a_{i,j} \leq a_{i,j+1} (\forall 1 \leq j < i \leq n ),
\]
and let $\mathbf{a}=(a_{n,1},a_{n,2},\ldots,a_{n,n})$.
When $(a_{i,j})_{ij}$ form an unsigned modified Gelfand-Tsetlin pattern, namely when for any $1 \leq j \leq i \leq n-1$, $a_{i+1,j} \leq a_{i,j} < a_{i+1,j+1}$ holds, 
we have $\#\text{GT}(\mathbf{a})\mid_{\eta_\text{row}=A}=1$. Therefore, we have $\#\text{GT}(\tilde{\mathbf{a}})\mid_{\eta_\text{row}=A} = \text{sgn}\,s$, where $s \in \mathfrak{S}_n$ and $\tilde{\mathbf{a}}_i = \mathbf{a}_{s(i)}$.
Otherwise, when $(a_{i,j})_{ij}$ do not form a classical Gelfand-Tsetlin pattern, $\#\text{GT}(\tilde{\mathbf{a}})\mid_{\eta_\text{row}=A} = 0$ for any $\tilde{\mathbf{a}} \in \mathbb{Z}^n$.
To describe what occurs in this situation, we need some new notions.
\begin{definition}
\textit{A system of combinatorial objects} consists of
\begin{itemize}
\item a set $\{A_\lambda\}_{\lambda \in \Lambda}$ of combinatorial objects,
\item a set $\{ \varphi_\mu \}_{\mu \in M}$ of sijections between combinatorial objects,
\item a source function $s \colon M \to \Lambda$ and a target function $t \colon M \to \Lambda$ such that $\varphi_\mu \colon A_{s(\mu)} \sijto A_{t(\mu)}$ holds for any $\mu \in M$.
\end{itemize}
We define an \textit{integrability condition} for a pair of combinatorial objects $(A_\lambda,A_{\lambda'})$ as follows:
\begin{quote}
	There exists a path of sijections from $A_\lambda$ to $A_{\lambda'}$, namely $(\mu_1, \mu_2, \ldots, \mu_k) \in M^k$ such that $s(\mu_1) = \lambda$, $t(\mu_i) = s(\mu_{i+1})$ for any $1 \leq i \leq k-1$ and $t(\mu_k)=\lambda'$.
	In addition, the composition $\varphi_{\mu_k} \circ \varphi_{\mu_{k-1}} \circ \cdots \circ \varphi_{\mu_1}$ is independent to the choice of the path.
\end{quote}
When any pair of combinatorial objects meets the integrability condition, we say the system is \textit{integrable}, and when any pair $(A_\lambda,A_{\lambda'})$ of combinatorial objects such that $A_\lambda^- = A_{\lambda'}^- = \emptyset$ meets the integrability condition, we say the system is \textit{partially integrable}.
\end{definition}
The following lemma is useful to prove a system is partially integrable.
\begin{lemma}\label{lem::sys}
Let $(\{A_\lambda\}_{\lambda \in \Lambda},\{ \varphi_\mu \}_{\mu \in M},s,t)$ be a system of combinatorial objects.
When the system is strongly connected, namely for any pair $(\lambda,\lambda') \in \Lambda^2$ there exists a path of sijections from $A_\lambda$ to $A_{\lambda'}$, the system is partially integrable if it meets the following condition:
\begin{quote}
	There exists $\lambda \in \Lambda$ such that $A_\lambda^-=\emptyset$ and $(A_\lambda,A_\lambda)$ meets the integrability condition.
\end{quote}
\end{lemma}
\begin{proof}
Let $\lambda_1,\lambda_2 \in \Lambda$ such that $A_{\lambda_1}^- = A_{\lambda_2}^- = \emptyset$.
Since the system is strongly connected, there is a path of sijections from $A_{\lambda_1}$ to $A_{\lambda_2}$.
Let $\psi,\tilde{\psi}$ be compositions of sijections from $A_{\lambda_1}$ to $A_{\lambda_2}$ along two different paths.
Since $A_{\lambda_1}^- = A_{\lambda_2}^-=\emptyset$, they are ordinary bijections.
Similarly, we can take sijections $\varphi_1 \colon A_\lambda \sijto A_{\lambda_1}$, $\varphi_2 \colon A_\lambda \sijto A_{\lambda_2}$ along some paths in the system and they are also bijections.
From the integrability condition for $(A_\lambda,A_\lambda)$, we have $\varphi_2 \circ \psi \circ \varphi_1^{-1} = \varphi_2 \circ \tilde{\psi} \circ \varphi_1^{-1}$, namely $\psi=\tilde{\psi}$.
Since this means that $(A_{\lambda_1},A_{\lambda_2})$ meets the integrability condition, the system is partially integrable.
\end{proof}

\begin{remark}
In this lemma, if the system has an ordinary set as its objects, the converse of the lemma holds by definition.
\end{remark}
\begin{remark}
In general, it does not holds that
\[
	f \circ g = f \circ h \Rightarrow g=h
\]
for sijections $f$, $g$ and $h$, unlike the case of bijections. For example, let $f$ and $g$ be sijections from $(\emptyset,\emptyset)$ to $(\{a,b\},\{c,d\})$ such that $f(a)=c$, $f(b)=d$, $g(a)=d$, $g(b)=c$. Then, $f \circ g = f \circ f \colon (\emptyset,\emptyset) \sijto (\emptyset,\emptyset)$ holds, but $f \ne g$.
\end{remark}

In bijective (as well as in \textit{sijective}) combinatorics, our goal is to construct larger (partially) integrable systems of combinatorial objects.
In this sense, to construct a compatible sijection is a weak goal.
For a strongly connected system $(\{A_\lambda\}_{\lambda \in \Lambda},\{ \varphi_\mu \}_{\mu \in M},s,t)$ and a statistic $\eta$ on $\bigsqcup_{\lambda \in \Lambda} \bs{A_\lambda}$,
we say the system is compatible with $\eta$ when each $\varphi_\mu$ is compatible with $\eta$.
If the statistic $\eta$ is injective on $\bs{A_\lambda}$ (note that this requires $A_\lambda^- = \emptyset$), according to the lemma, the system is partially integrable.

For instance, in the case of Gelfand-Tsetlin patterns the system with $\mathbf{k} \in \mathbb{Z}^n$ is described as
\begin{itemize}
\item a set of combinatorial objects: $\{ \text{GT}(\mathbf{l}) \mid { \mathbf{l} \text{ is a permutation of } \mathbf{k}} \}$,
\item a set of sijections: $\{\pi_{\mathbf{l},i} \mid \mathbf{l} \text{ is a permutation of } \mathbf{k} \text{ and } i \in \{1,2,\ldots,n\} \}$.
\end{itemize}
Then the system is compatible with $\eta_\text{row}$.
Since $\eta_\text{row}$ is injective on $\bs{\text{GT}(\mathbf{k})}$ when $\mathbf{k}$ is increasing, the system is partially integrable.
Therefore, the above discussion results in the following corollary.
\begin{corollary}
Let $\mathbf{k} \in \mathbb{Z}^n$ and $\mathfrak{S}\mathbf{k} = \{ \mathbf{l} \mid \mathbf{l} \text{ is a permutation of } \mathbf{k} \}$.
Then, $\{\text{GT}(\mathbf{l})\}_{\mathbf{l} \in \mathfrak{S}\mathbf{k}}$ and $\{\pi_{\mathbf{l},i}\}_{\mathbf{l} \in \mathfrak{S}\mathbf{k}, i \in \{1,2,\ldots,n\}}$ form a partially integrable system of combinatorial objects.
\end{corollary}

\begin{remark}
In this remark, we explain why the definition of ``partially integrable'' is needed, or in other words why ``integrable'' is too strict.
For this purpose, we give an example of a partially integrable system of combinatorial objects which is not integrable.
Let $\varphi \colon S \sijto T$ be the sijection constructed in Remark \ref{rem-2.2.2}.
We consider the system described as
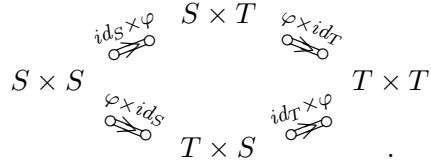
\[\begin{tikzcd}
&S \times T \arrow[rd,phantom,"\overset{\varphi \times \text{id}_T}{\sijto}",sloped,yshift=.5em] & \\[-1em]
S \times S \arrow[rd,phantom,"\overset{\varphi \times \text{id}_S}{\sijto}",sloped] \arrow[ru,phantom,"\overset{\text{id}_S \times \varphi}{\sijto}",sloped,yshift=.5em] && T \times T \\[-1em]
& T \times S \arrow[ru,phantom,"\overset{\text{id}_T \times \varphi}{\sijto}",sloped] & .
\end{tikzcd}\]
Here, when we define a system by a diagram like as above, ``$\psi \colon U \sijto V$'' in a diagram implies that the system includes not only $\psi \colon U \sijto V$ but also $\psi^{-1} \colon V \sijto U$ .
We showed that $(\text{id}_T \times \varphi) \circ (\varphi \times \text{id}_S) \ne  (\varphi \times \text{id}_T) \circ (\text{id}_S \times \varphi)$ in the remark.
Therefore, the system is not integrable.
However, by Lemma \ref{lem::sys}, this system is partially integrable because of $\#(S \times S)^+ = 1$ and $\#(S \times S)^- = 0$.

This example is very simple and a similar situation occurs in general when we take the Cartesian product of two or more sijection that have some cancellations.
Therefore, there are no integrable systems that do not involve at least some trivial cases.
\end{remark}

Last, we give two proofs of the signed enumeration of the Gelfand-Tsetlin patterns, one of which is combinatorial and the other is partially computational.
The enumeration for a strictly increasing $\mathbf{k} \in \mathbb{Z}^n$ is known to be (for example, see \cite{EC2})
\[
	\#\text{GT}(\boldsymbol{k})
	 = \prod_{1 \leq i<j \leq n} \frac{k_j-k_i}{j-i},
\]
and we shall prove that the result is true for general cases. 
First, let us prepare some definitions.
For $\mathbf{k} \in \mathbb{Z}^n$, we define $\text{sgn}(\mathbf{k})$ as the sign of a permutation $s \in \mathfrak{S}_n$ such that $k_{s(1)}<k_{s(2)}<\cdots<k_{s(n)}$ when all elements of $\mathbf{k}$ are distinct. Otherwise, we define $\text{sgn}(\mathbf{k})=0$.
Let $\mathbf{k}^\text{inc}$ be a weakly increasing sequence such that $\{\!\{\mathbf{k}\}\!\} = \{\!\{\mathbf{k}^\text{inc}\}\!\}$.
For a signed set $S$ and numbers $\{a_s\}_{s \in \bs{S}}$, we define
\[
	\sum_{s \in S} a_s = \sum_{s \in S^+} a_s - \sum_{s \in S^-} a_s.
\]

\begin{proof}[Combinatorial proof of the signed enumeration]
When $\text{sgn}(\mathbf{k}) \ne 0$, we have a sijection $\text{GT}(\mathbf{k}) \sijto \text{sgn}(\mathbf{k}) \text{GT}(\mathbf{k}^\text{inc})$ according to Construction \ref{pi}.
Therefore,
\[
\#\text{GT}(\mathbf{k}) = \text{sgn}(\mathbf{k}) \cdot \#\text{GT}(\mathbf{k}^\text{inc}) = \prod_{1 \leq i<j \leq n} \frac{k_j-k_i}{j-i}.
\]
When $\text{sgn}(\mathbf{k})=0$, similarly we have a sijection $\text{GT}(\mathbf{k}) \sijto \text{GT}(\mathbf{k}^\text{inc})$ and since $\mathbf{k}$ has duplicate elements, we have 
\[
\#\text{GT}(\mathbf{k}) = 0 = \prod_{1 \leq i<j \leq n} \frac{k_j-k_i}{j-i}. \qedhere
\]
\end{proof}

\begin{proof}[Computational proof of the signed enumeration]
It is sufficient to prove the following claim.
\begin{claim}\label{prop::GT_det}
Let $A=(A_1,A_2,\ldots,A_n)\in \mathcal{A}_\text{row}^n$
 such that $\{\!\{\mathbf{k}\}\!\} = A_n$.
Then, we have
\[
	\#\text{GT}(\mathbf{k})\mid_{\eta_\text{row}=A} = \begin{cases}
		\text{sgn}(\mathbf{k}) & \text{ when $A$ forms an unsigned modified Gelfand-Tsetlin pattern},\\
		0 & \text{ otherwise}.
	\end{cases}
\]
\end{claim}
The proof is given by induction on $n$. For $n=1$, this is trivial.
For $n>1$, let $A' \in \mathcal{A}_\text{row}^{n-1}$ such that $A = (A',\{\!\{\mathbf{k}\}\!\})$. Then, we have
\[
	\#\text{GT}(\mathbf{k})\mid_{\eta_\text{row}=A} = \sum_{\mathbf{l} \in \sint{k_1}{k_2} \times \sint{k_2}{k_3} \times \cdots \times \sint{k_{n-1}}{k_n}} \#\text{GT}(\mathbf{l}) \mid_{\eta_\text{row}=A'}.
\]
Let $a_1\leq a_2\leq \cdots\leq a_{n-1}$ be the elements of $A_{n-1}$.
If they are not distinct, $A'$ does not form a classical Gelfand-Tsetlin pattern and neither does $A$.
Therefore, in this case we obtain the result that $\#\text{GT}(\mathbf{k})\mid_{\eta_\text{row}=A} = 0$.
In the following, we assume that $a_1 < a_2 < \cdots < a_{n-1}$.

Now, with the assumption of the induction, we have
\begin{align*}
	\#\text{GT}(\mathbf{k})\mid_{\eta_\text{row}=A} &= \sum_{\mathbf{l} \in \sint{k_1}{k_2} \times \sint{k_2}{k_3} \times \cdots \times \sint{k_{n-1}}{k_n}} \#\text{GT}(\mathbf{l}) \mid_{\eta_\text{row}=A'} \\
	&=\sum_{\mathbf{l} \in \sint{k_1}{k_2} \times \sint{k_2}{k_3} \times \cdots \times \sint{k_{n-1}}{k_n} \text{ s.t. } \{\!\{ \mathbf{l} \}\!\} = A_{n-1}} \#\text{GT}(\mathbf{l}) \mid_{\eta_\text{row}=A'} \\
	&=1_\text{$A'$ forms a classical Gelfand-Tsetlin pattern} \\
	&\hspace{100pt}\times \sum_{\mathbf{l} \in \sint{k_1}{k_2} \times \sint{k_2}{k_3} \times \cdots \times \sint{k_{n-1}}{k_n} \text{ s.t. } \{\!\{ \mathbf{l} \}\!\} = A_{n-1}} \text{sgn}(\mathbf{l}).
\end{align*}

For $i \in \{1,2,\ldots,n\}$, let $t_i$ be the integer such that $a_{t_i-1} < k_i \leq a_{t_i}$, where we define $a_0 = -\infty$ and $a_n = +\infty$.
Then, it is sufficient to show that
\[
	\sum_{\mathbf{l} \in \sint{k_1}{k_2} \times \sint{k_2}{k_3} \times \cdots \times \sint{k_{n-1}}{k_n} \text{ s.t. } \{\!\{ \mathbf{l} \}\!\} = A_{n-1}} \text{sgn}(\mathbf{l})
	=\begin{cases}
	\text{sgn}(\mathbf{k}) & \text{if }t_1,t_2,\ldots,t_n \text{ are distinct}, \\
	0 & \text{otherwise},
	\end{cases}
\]
because $A$ forms a classical Gelfand-Tsetlin pattern if and only if so does $A'$ and $t_1,t_2,\ldots,t_n$ are distinct.
Let $e_1,e_2,\ldots,e_{n-1}$ be the canonical basis of $\mathbb{R}^{n-1}$ and $f_i = \sum_{j=1}^{i-1} e_j$ for $i \in \{1,2,\ldots,n\}$.
Then, we have
\[
	\sum_{\mathbf{l} \in \sint{k_1}{k_2} \times \sint{k_2}{k_3} \times \cdots \times \sint{k_{n-1}}{k_n} \text{ s.t. } \{\!\{ \mathbf{l} \}\!\} = A_{n-1}} \text{sgn}(\mathbf{l})
	= \text{det}\left( f_{t_2}-f_{t_1}, f_{t_3}-f_{t_2}, \ldots, f_{t_n}-f_{t_{n-1}}\right).
\]

The reason why this determinant formula holds is the following:
\begin{quote}
Taking an element $\mathbf{l}$ of $\sint{k_1}{k_2} \times \sint{k_2}{k_3} \times \cdots \times \sint{k_{n-1}}{k_n}$ such that $\{\!\{ \mathbf{l} \}\!\} = A_{n-1}$, is equivalent to assigning each element of $A_{n-1} = \{\!\{a_1,a_2,\ldots,a_{n-1}\}\!\}$ to an element of $\mathbf{l} = (l_1,l_2,\ldots,l_{n-1})$ one-to-one, such that $a_i$ corresponds to $l_j$ with $a_i \in \sint{k_j}{k_{j+1}}$, including its sign. By the definition of $t_j$ and $t_{j+1}$, $a_i \in \sint{k_j}{k_{j+1}}$ means that $i \in \sint{t_j}{t_{j+1}}$.
From the viewpoint of $l_j$ this assignment is identical to a bijection $\sigma \in  \mathfrak{S}_{n-1} ; j \mapsto i = \sigma(j)$ such that $\sigma(j) \in \sint{t_j}{t_{j+1}}$. Furthermore, $\text{sgn}(\mathbf{l})$ is equal to $\text{sgn}(\sigma)$ by the definition of $\text{sgn}(\mathbf{l})$, and we obtain the determinant formula.
\end{quote}

If $t_i = t_j$ for some $1 \leq i < j \leq n$, then $f_{t_{i+1}}-f_{t_i},f_{t_{i+2}}-f_{t_{i+1}},\ldots,f_{t_{j}}-f_{t_{j-1}}$ are not linearly independent.
Thus, if $t_1,t_2,\ldots,t_n$ are not distinct, the determinant is equal to zero.
Otherwise, we will show that the determinant is equal to the sign of $\mathbf{T} = (t_1,t_2,\ldots,t_n)$, which is equal to $\text{sgn}(\mathbf{k})$.
First, let
\[
	F(x_1,x_2,\ldots, x_n) = \text{det}\left( f_{x_1}-f_{x_2}, f_{x_2}-f_{x_3}, \ldots, f_{x_{n-1}}-f_{x_n}\right),
\]
and then $F$ has the anti-symmetric property. In fact, we have
\begin{align*}
F(x_1,x_2,\ldots,x_n)
&=\text{det}\left(f_{x_1}-f_{x_2},f_{x_2}-f_{x_3},\ast,\ldots,\ast\right) \\
&=\text{det}\left(f_{x_1}-f_{x_2},(f_{x_1}-f_{x_2})+(f_{x_2}-f_{x_3}),\ast,\ldots,\ast\right)\\
&=-\text{det}\left(f_{x_2}-f_{x_1},f_{x_1}-f_{x_3},\ast,\ldots,\ast\right) \\
&=-F(x_2,x_1,x_3,\ldots,x_n),
\end{align*}
\begin{align*}
F(x_1,x_2,\ldots,x_n)
&=\text{det}\left(\ast,\ldots,\ast,f_{x_{n-2}}-f_{x_{n-1}},f_{x_{n-1}}-f_{x_n}\right) \\
&=\text{det}\left(\ast,\ldots,\ast,(f_{x_{n-2}}-f_{x_{n-1}})+(f_{x_{n-1}}-f_{x_n}),f_{x_{n-1}}-f_{x_n}\right) \\
&=-\text{det}\left(\ast,\ldots,\ast,f_{x_{n-2}}-f_{x_n},f_{x_n}-f_{x_{n-1}}\right) \\
&=-F(x_1,\ldots,x_{n-2},x_{n},x_{n-1}),
\end{align*}
and for all $i \in \{2,3,\ldots, n-2\}$
\begin{align*}
&F(x_1,x_2,\ldots,x_n) \\
&\quad=\text{det}\left(\ast,\ldots,\ast,f_{x_{i-1}}-f_{x_{i}},f_{x_{i}}-f_{x_{i+1}},f_{x_{i+1}}-f_{x_{i+2}},\ast,\ldots,\ast\right) \\
&\quad= \text{det}\left(\ast,\ldots,\ast,(f_{x_{i-1}}-f_{x_{i}})+(f_{x_{i}}-f_{x_{i+1}}),f_{x_{i}}-f_{x_{i+1}},(f_{x_{i+1}}-f_{x_{i+2}})+(f_{x_{i}}-f_{x_{i+1}}),\ast,\ldots,\ast\right) \\
&\quad= \text{det}\left(\ast,\ldots,\ast,f_{x_{i-1}}-f_{x_{i+1}},f_{x_{i}}-f_{x_{i+1}},f_{x_{i}}-f_{x_{i+2}},\ast,\ldots,\ast\right) \\
&\quad= -\text{det}\left(\ast,\ldots,\ast,f_{x_{i-1}}-f_{x_{i+1}},f_{x_{i+1}}-f_{x_{i}},f_{x_{i}}-f_{x_{i+2}},\ast,\ldots,\ast\right) \\
&\quad=-F(x_1,\ldots,x_{i-1},x_{i+1},x_i,x_{i+2},\ldots,x_n).
\end{align*}
Now, the determinant is evaluated as follows:
\begin{align*}
\text{det}\left( f_{t_2}-f_{t_1}, f_{t_3}-f_{t_2}, \ldots, f_{t_n}-f_{t_{n-1}}\right) &= (-1)^{n-1} F(t_1,t_2,\ldots,t_n) \\
&=\text{sgn}(\mathbf{T}) \cdot(-1)^{n-1}  F(1,2,\ldots,n) \\
&=\text{sgn}(\mathbf{T}) \cdot \text{det}\left( f_{2}-f_{1}, f_{3}-f_{2}, \ldots, f_{n}-f_{n-1}\right)\\
&=\text{sgn}(\mathbf{T}) \cdot \text{det}\left( e_1, e_2, \ldots, e_{n-1} \right)\\
&=\text{sgn}(\mathbf{T}). \qedhere
\end{align*}

\end{proof}

\subsection{
Gelfand-Tsetlin Graph sequences}
In this subsection, we will define a generalization of modified Gelfand-Tsetlin patterns and we extend the scope of the results in the previous subsection to these objects.
The generalization is also a generalization of the notion of Gelfand-Tsetlin tree sequences, which is introduced in \cite{Fis}.
In particular, we prove a generalization of Claim \ref{prop::GT_det} and a signed enumeration of the generalized Gelfand-Tsetlin patterns,
which leads us to the notion of ``Gelfand-Tsetlin graph sequences'', a more generalized version of Gelfand-Tsetlin tree sequences.

\begin{definition}
Let $n$ be a positive integer and $\mathbf{k} \in \mathbb{Z}^n$, and for $1 \leq j \leq i \leq n-1$ let $p_{i,j},q_{i,j}$ be positive integers less than or equal to $i+1$.
Then, we define the signed set $\text{GGT}\left(\mathbf{k};\{p_{i,j}\},\{q_{i,j}\}\right)$ of \textit{generalized Gelfand-Tsetlin patterns} with bottom row $\mathbf{k}$ and parameters $\{p_{i,j}\},\{q_{i,j}\}$ as:
\begin{itemize}
\item when $n=1$, $\text{GGT}\left(\mathbf{k};\{p_{i,j}\},\{q_{i,j}\}\right) = (\{\cdot\},\emptyset)$,
\item otherwise, \begin{multline*}
\displaystyle \text{GGT}\left(\mathbf{k};\{p_{i,j}\},\{q_{i,j}\}\right) \\
=\bigsqcup_{\mathbf{l} \in K} 
\text{GGT}\left(\mathbf{l};\{p_{i,j}\}_{1 \leq j \leq i \leq n-2},\{q_{i,j}\}_{1 \leq j \leq i \leq n-2}\right), \\
\text{where $K=\sint{k_{p_{n-1,1}}}{k_{q_{n-1,1}}} \times \sint{k_{p_{n-1,2}}}{k_{q_{n-1,2}}} \times \cdots \times \sint{k_{p_{n-1,n-1}}}{k_{q_{n-1,n-1}}}$.}
\end{multline*}
\end{itemize}
\end{definition}
When we set $p_{i,j} = j$ and $q_{i,j} = j+1$ for all $1 \leq i \leq j \leq n-1$, then this definition coincides with the definition of modified Gelfand-Tsetlin patterns $\text{GT}(\mathbf{k})$.
This definition is also a generalization of (modified) Gelfand-Tsetlin tree sequences.
We consider the graph $G_i$ for $i \in \{1,2,\ldots,n\}$ with vertices $\{ 1,2,\ldots,i+1 \}$ and (named) edges $\{ j \colon p_{i,j} \to q_{i,j} \mid 1 \leq j \leq i\}$, where $s \to t$ denotes a directed edge from $s$ to $t$.
When all $G_i$ are trees (note that here the directions of edges do not matter), the definition of the generalized Gelfand-Tsetlin patterns coincides with a signed set of Gelfand-Tsetlin tree sequences with some parameters, or its opposite.
To explain the relation between them, we define the sign of parameters $\{p_{i,j}\},\{q_{i,j}\}$.

\begin{definition} \label{def::sgn_param}
Let $n$ be a positive integer and for $1 \leq j \leq i \leq n-1$ let $p_{i,j}$ and $q_{i,j}$ be positive integers less than or equal to $i+1$.
Let $G_i$ be the graph described above. We define the sign of $(\{p_{i,j}\},\{q_{i,j}\})$ as follows and denote it by $\text{sgn}(\{p_{i,j}\},\{q_{i,j}\})$:
\begin{itemize}
\item when $G_i$ is not a tree for some $i \in \{1,2,\ldots,n-1\}$, the sign is $0$.
\item otherwise, the sign is $\prod_{i=1}^{n-1} (-1)^{\#\{j \mid 1 \leq j \leq i,\,p_{i,j}=r_{i}(j) \}} \cdot \text{sgn}(r_i)$, where $r_i$ is a bijection from $\{0,1,\ldots,i\}$ to $\{1,2,\ldots,i+1\}$ constructed as below and its sign is defined by $\text{sgn}(r_i) = \text{sgn}(r_i-1)$ (note that $r_i-1$ is a permutation of $\{0,1,\ldots,i\}$):
\begin{quote}
We choose $r_i(0)\in \{1,2,\ldots,i+1\}$ arbitrarily. For $j \in \{1,2,\ldots,i\}$, if $p_{i,j}$ is further from $r_i(0)$ than $q_{i,j}$ in $G_i$ with respect to the shortest path distance, let $r_i(j)=p_{i,j}$ and otherwise let $r_i(j)=q_{i,j}$.
\end{quote}
In fact, the sign of $r_i$ is independent of the choice of $r_i(0)$. For a proof, see \cite[after equation (2.1)]{Fis}.
\end{itemize}
\end{definition}

When the sign is $+1$, the definition of the generalized Gelfand-Tsetlin patterns coincides with the signed set of Gelfand-Tsetlin tree sequences and when the sign is $-1$ it coincides with the opposite.
Conversely, Gelfand-Tsetlin tree sequences are always expressed as generalized Gelfand-Tsetlin patterns with some parameters.
The notion of generalized Gelfand-Tsetlin patterns is a true generalization of Gelfand-Tsetlin tree sequences, but we will prove that $\#\text{GGT}\left(\mathbf{k};\{p_{i,j}\},\{q_{i,j}\}\right) =0$ when some $G_i$ is not a tree (See proposition \ref{prop::GGT}).

We define the statistic $\eta_\text{row}$ in a similar way for Gelfand-Tsetlin patterns, namely
\begin{itemize}
\item for $k \in \text{GGT}(k;\emptyset,\emptyset)$, $\eta_\text{row}(k) = \{\!\{k\}\!\} \in\text{MS}(\mathbb{Z},1) = \mathcal{A}_\text{row}^1$.
\item for $T \in \text{GGT}\left(\mathbf{k};\{p_{i,j}\},\{q_{i,j}\}\right)$, where $\mathbf{k} \in \mathbb{Z}^n$ and $n \geq 2$,
we have $T' \in \text{GGT}\big(\mathbf{l};\{p_{i,j}\}_{1 \leq j \leq i \leq n-2},$ $\{q_{i,j}\}_{1 \leq j \leq i \leq n-2}\big)$ and $\mathbf{l} \in \mathbb{Z}^{n-1}$ such that $T=(T',\mathbf{l})$.
Then, we define \[\eta_\text{row}(T) = ( \eta_\text{row}(T'), \{\!\{k_1,k_2,\ldots,k_n\}\!\}) \in \mathcal{A}_\text{row}^{n-1} \times \text{MS}(\mathbb{Z},n) = \mathcal{A}_\text{row}^{n}.\]
\end{itemize}

\begin{proposition}[cf. Claim \ref{prop::GT_det} and the enumeration of modified Gelfand-Tsetlin patterns] \label{prop::GGT}
Let $n$ be a positive integer and $\mathbf{k} \in \mathbb{Z}^n$, and for $1 \leq j \leq i \leq n-1$ let $p_{i,j}$ and $q_{i,j}$ be positive integers less than or equal to $i+1$.
Moreover, let $A$ be an element of $\mathcal{A}_\text{row}^n$ such that $\{\!\{\mathbf{k}\}\!\} = A_n \in \text{MS}(\mathbb{Z},n)$.
Then, we have
\begin{multline*}
	\#\text{GGT}\left(\mathbf{k};\{p_{i,j}\},\{q_{i,j}\}\right) \mid_{\eta_\text{row}=A} 
	\\
	= \begin{cases}
		\text{sgn}(\mathbf{k}) \cdot \text{sgn}(\{p_{i,j}\},\{q_{i,j}\}) & \text{ when $A$ forms an unsigned modified Gelfand-Tsetlin pattern},\\
		0 & \text{ otherwise}.
	\end{cases}
\end{multline*}
In particular,
\begin{align*}
\#\text{GGT}\left(\mathbf{k};\{p_{i,j}\},\{q_{i,j}\}\right) &= \text{sgn}(\mathbf{k}) \cdot \text{sgn}(\{p_{i,j}\},\{q_{i,j}\}) \cdot \#\text{GT}(\mathbf{k}^\text{inc}) \\ &= \text{sgn}(\{p_{i,j}\},\{q_{i,j}\}) \cdot \prod_{1 \leq i<j \leq n} \frac{k_j-k_i}{j-i}.
\end{align*}
\end{proposition}
\begin{proof}
When all $G_i$ are trees, the proof runs parallel to the proof of Claim \ref{prop::GT_det}. Therefore, it is sufficient to show that for $n \geq 2$
\begin{multline} \label{prop::GGT::eq}
	\sum_{\mathbf{l} \in \sint{k_{p_{n-1,1}}}{k_{q_{n-1,1}}} \times \sint{k_{p_{n-1,2}}}{k_{q_{n-1,2}}} \times \cdots \times \sint{k_{p_{n-1,n-1}}}{k_{q_{n-1,n-1}}} \text{ s.t. } \{\!\{ \mathbf{l} \}\!\} = A_{n-1}} \text{sgn}(\mathbf{l})
	\\ =\begin{cases}
	\text{sgn}(\mathbf{k}) \cdot \frac{\text{sgn}(\{p_{i,j}\}_{1 \leq j \leq i \leq n-1},\{q_{i,j}\}_{1 \leq j \leq i \leq n-1})}{\text{sgn}(\{p_{i,j}\}_{1 \leq j \leq i \leq n-2},\{q_{i,j}\}_{1 \leq j \leq i \leq n-2})} & \text{if } t_1,t_2,\ldots,t_n \text{ are distinct}, \\
	0 & \text{otherwise},
	\end{cases}
\end{multline}
instead of 
\[
	\sum_{\mathbf{l} \in \sint{k_1}{k_2} \times \sint{k_2}{k_3} \times \cdots \times \sint{k_{n-1}}{k_n} \text{ s.t. } \{\!\{ \mathbf{l} \}\!\} = A_{n-1}} \text{sgn}(\mathbf{l})
	=\begin{cases}
	\text{sgn}(\mathbf{k}) & \text{if } t_1,t_2,\ldots,t_n \text{ are distinct}, \\
	0 & \text{otherwise}.
	\end{cases}
\]
(For the definition of the invariants, see the proof of Claim \ref{prop::GT_det}.)
Here, the LHS is equal to 
\begin{equation} \label{prop::GGT::det-formula}
	\text{det}\left(f_{t_{q_{n-1,1}}}-f_{t_{p_{n-1,1}}},f_{t_{q_{n-1,2}}}-f_{t_{p_{n-1,2}}},\ldots,f_{t_{q_{n-1,n-1}}}-f_{t_{p_{n-1,n-1}}}\right).
\end{equation}
Here, each row of the matrix in (\ref{prop::GGT::det-formula}) is of the form $\pm(f_{\text{par}(v)} - f_v)$, where $\text{par}(v)$ is the parent of the vertex $v$ in G with root $0$. Furthermore, each $v \in \{1,2,\ldots,n-1\}$ appears exactly once.
Therefore, when we consider treating vertices of $G$ in postorder (with root $0$), the determinant (\ref{prop::GGT::det-formula}) is equal to
\begin{align*}
	&(-1)^{\#\{j \mid 1 \leq j \leq n-1,\, p_{n-1,j} = r(j)\}} \cdot \text{det}\left( f_{t_{r(1)}},f_{t_{r(2)}},\ldots,f_{t_{r(n-1)}}\right) \\
	&\qquad= (-1)^{\#\{i \mid 1 \leq j \leq n-1,\, p_{n-1,j} = r(j)\}} \cdot \text{sgn}(\mathbf{T}) \cdot \text{sgn}(r).
\end{align*}

Note that we defined $\text{sgn}(\mathbf{T})=0$ when $\mathbf{T}$ has duplicates.
It immediately implies $(\ref{prop::GGT::eq})$ by the definition of the sign of parameters.

When some $G_i$ is not a tree, let $i$ be the minimum such index.
When $i < n-1$, by the induction, the statement becomes trivial because it is holds that $\#\text{GGT}(\mathbf{l};\{p_{i,j}\}_{1 \leq j \leq i \leq n-2},\{q_{i,j}\}_{1 \leq j \leq i \leq n-2})=0$ independent of $\mathbf{l}$. Otherwise, it is sufficient to show that the determinant (\ref{prop::GGT::det-formula})
is equal to $0$.
Now, $G_{n-1}$ has $n$ vertices and $n-1$ edges but it is not a tree. Therefore, $G_{n-1}$ has at least one cycle. (Here, we forget the direction of the edges.)
When the edges $j_1,j_2,\ldots,j_\ell$ form a cycle, then $f_{t_{q_{n-1,j_1}}}-f_{t_{p_{n-1,j_1}}},f_{t_{q_{n-1,j_2}}}-f_{t_{p_{n-1,j_2}}},\ldots,f_{t_{q_{n-1,j_\ell}}}-f_{t_{p_{n-1,j_\ell}}}$ are not linearly independent, which implies that the determinant is $0$.
\end{proof}
This proposition is an answer to Fischer's question relevant to the notion of \textit{Gelfand-Tsetlin Graph sequences} in \cite{Fis}:
\begin{quote}
\textit{
``
we want to mention an obvious generalization of Gelfand-Tsetlin tree sequences, which we do not consider in this article, but might be interesting to look at: the notion of admissibility makes perfect sense if the tree T is replaced by any other graph. Are there any nice assertions to be made on ``Gelfand-Tsetlin graph sequences''?
''
}
\end{quote}
We can also interpret the proposition via the notions of signed sets, sijections and compatibility. This interpretation results in the following corollary.
\begin{corollary}
Let $n$, $\mathbf{k}$, $\{p_{i,j}\}$, $\{q_{i,j}\}$ and $A$ be the parameters defined in the above proposition.
Then, there exists a sijection between $\text{GGT}\left(\mathbf{k};\{p_{i,j}\},\{q_{i,j}\}\right) \mid_{\eta_\text{row}=A}$ and the signed set $S$ defined by
\[
	S=\begin{cases} (\emptyset,\emptyset) & \text{when sgn}(\{p_{i,j}\}, \{q_{i,j}\})=0, \\
	\text{GT}(\mathbf{k}) \mid_{\eta_\text{row}=A} & \text{when sgn}(\{p_{i,j}\}, \{q_{i,j}\})=+1, \\
	-\text{GT}(\mathbf{k}) \mid_{\eta_\text{row}=A} & \text{when sgn}(\{p_{i,j}\}, \{q_{i,j}\})=-1, \end{cases}
\]
which is compatible with $\eta_\text{row}$.
\end{corollary}
This corollary states only the existence of such sijections.
It should be possible to construct them explicitly by generalizing the constructions in this section with the combinatorial operations illustrated in \cite{Fis}.

\subsection{Preparation for Section \ref{secofSGT}}
For use in Section \ref{secofSGT} we check the compatibilities of the sijections constructed in the present section.
First, we define statistics $\eta_{\text{row},i}$ as the $i$-th component of $\eta_\text{row}$ on Gelfand-Tsetlin patterns with the length of the bottom row $\geq i$. 
Moreover, we define a $\mathbb{Z}$-valued statistic $\eta_\text{top}$ of Gelfand-Tsetlin patterns as the value of the unique element of $\eta_{\text{row},1}$.

For example,
\begin{align*}
	\eta_{\text{row},1}\left( \begin{array}{ccccc}&&4&&\\&2&&5&\\2&&5&&7\end{array} \right) = \{\!\{4\}\!\},&&
	\eta_{\text{row},3}\left( \begin{array}{ccccc}&&4&&\\&2&&5&\\2&&5&&7\end{array} \right) = \{\!\{2,5,7\}\!\},
\end{align*}
and
\[
	\eta_\text{top}\left( \begin{array}{ccccc}&&4&&\\&2&&5&\\2&&5&&7\end{array} \right) = 4.
\]
In addition, we define $\eta_\text{top}$ for a signed set which is in the form of $\bigsqcup_{s \in S} \text{GT}(\mathbf{k}_s)$ by
\[
	\text{for } (T,s) \in \bigsqcup_{s \in S} \text{GT}(\mathbf{k}_s),\,\, \eta_\text{top}((T,s)) = \eta_\text{top}(T).
\]
For $\eta_{\text{row},i}$, we define these values on such signed sets similarly.
Note that these definitions are compatible with that for Gelfand-Tsetlin patterns and the relation
$
	\textrm{GT}(\mathbf{k})=\bigsqcup_{\mathbf{l} \in \underline{[k_1,k_2)}\times\cdots\times\underline{[k_{n-1},k_n)}} \textrm{GT}(\mathbf{l})
$.
\begin{proposition}\label{prop::rho}
The sijection $\rho$ in Construction \ref{rho} is compatible with $\eta_\text{top}$, $\eta_{\text{row},1}$, $\eta_{\text{row},2}$, $\ldots$, $\eta_{\text{row},n}$.
\end{proposition}
\begin{proof}
The sijection is constructed as a disjoint union of $\text{id}_{\text{GT}(\mathbf{l})}$.
Therefore, the top $n$ rows are preserved under $\rho$. In particular, $\rho$ is compatible with the statistics.
\end{proof}
\begin{proposition}\label{prop::pi}
The sijections $\pi$, $\sigma$ in Construction \ref{pi} are compatible with $\eta_\text{top}$, $\eta_{\text{row},1}$, $\eta_{\text{row},2}$, $\ldots$, $\eta_{\text{row},n}$.
\end{proposition}
\begin{proof}
It follows directly from Theorem \ref{thm::GT}.
\end{proof}
\begin{proposition}\label{prop::tau}
The sijection $\tau$ in Construction \ref{tau} is compatible with $\eta_{\text{row},1}$, $\eta_{\text{row},2}$, $\ldots$, $\eta_{\text{row},n-1}$.
In particular, it is compatible with $\eta_\text{top}$ when $n \geq 2$.
\end{proposition}
\begin{proof}
As in the proof of Proposition \ref{prop::rho}, a disjoint union of $\text{id}_{\text{GT}(\mathbf{l})}$ is compatible with the statistics.
The sijection $\sigma$ is also compatible, so the sijection $\tau$, which is constructed by composing them, is also compatible.
\end{proof}

\section{
New sijections pertaining to monotone triangles}\label{secofSGT}
In this section, we simplify the sijection between alternating sign matrices and shifted Gelfand-Tsetlin patterns that is constructed in \cite{FK1},
and using the notion of compatibility we argue that it is more natural than the conventional one.
First, in \ref{ssec_def} we give definitions of some combinatorial objects. After that we define \textit{generalized inversion numbers} for these combinatorial objects in \ref{ssec_GIN}.
In \ref{ASM-SGT::KR} and \ref{ASM-SGT::MR} we construct a sijection and prove the compatibility of our construction with two statistics, one of which is the generalized inversion numbers we defined.

\subsection{Definitions}\label{ssec_def}
\subsubsection{Monotone Triangles}\label{def_MT}
A monotone triangle is a Gelfand-Tsetlin pattern such that for any $1 \leq j < i \leq n-1$, $a_{i,j} < a_{i,j+1}$ holds.
Note that we do not impose the condition on the bottom row.
We define a modified monotone triangle as a modified Gelfand-Tsetlin pattern such that for any $1 \leq j < i \leq n-1$, $a_{i,j} < a_{i,j+1}-1$ holds.
Here, the ``$-1$'' in the RHS is needed for consistency with our notation using half-open intervals.

There is a well-known bijection between alternating sign matrices and monotone triangles. (For example, see \cite{FK1}.)
Here, we illustrate a bijection between alternating sign matrices and modified monotone triangles without bothering with the details of the proof.
First, for an alternating sign matrix, we calculate the cumulative sum along columuns from above.
Next, row by row, we pick up the indices of columns where the element is $1$.
Last, we add to each element the number of elements on the left of it, to use half-open intervals
(cf. the bijection between Gelfand-Tsetlin patterns and modified Gelfand-Tsetlin patterns).
\begin{multline*}
	\hspace{2em}
	\underset{\text{alternating sign matrices}}{\begin{pmatrix} 0&0&1&0 \\ 1&0&-1&1 \\ 0&0&1&0 \\ 0&1&0&0 \end{pmatrix}}
	\hspace{2em} \to \hspace{2em}
	\underset{\text{(cumulative sum)}}{\begin{pmatrix} 0&0&1&0 \\ 1&0&0&1 \\ 1&0&1&1 \\ 1&1&1&1 \end{pmatrix}}
	\\ \to \hspace{2em}
	\underset{\text{monotone triangles}}{\begin{array}{ccccccc} &&&3&&& \\ &&1&&4&& \\ &1&&3&&4& \\ 1&&2&&3&&4 \end{array}}
	\hspace{2em} \to \hspace{2em}
	\underset{\text{modified monotone triangles}}{\begin{array}{ccccccc} &&&3&&& \\ &&1&&5&& \\ &1&&4&&6& \\ 1&&3&&5&&7 \end{array}}
	\hspace{2em}
\end{multline*}
In general, there is a one-to-one correspondence between alternating sign matrices with rank $n$ and monotone triangles with bottom row $(1,2,\ldots,n)$
and therefore between them and modified monotone triangles with bottom row $(1,3,\ldots,2n-1)$.

Last, we define the top of a modified monotone triangle $A=(a_{i,j})_{ij}$ as
$
	\eta_\text{top}(A) = a_{1,1}.
$
This statistic corresponds to Behrend, Di Francesco and Zinn-Justin's first statistic for alternating sign matrices, the column number of the $1$ in the first row.

\subsubsection{Generalized Monotone Triangles}\label{GMT}
We denote the signed set of modified monotone triangles with bottom row $\mathbf{k}$, by $\textrm{MT}(\mathbf{k})$.
Namely, we define
\begin{align*}
	\text{MT}(\mathbf{k})^+ = \{ \text{modified monotone triangles with bottom row } \mathbf{k} \}, && \text{MT}(\mathbf{k})^- = \emptyset,
\end{align*}
for a strictly increasing sequence $\mathbf{k} \in \mathbf{Z}^n$. For the case that $\mathbf{k}$ is not strictly increasing, see Remark \ref{uniqMT}.

Before explaining our definition of a generalized monotone triangle, some preparation is needed.
An \textit{arrow row} is a sequence of $\{\nwarrow,\nearrow,\nxarrow\}$, and as an element of a signed set its sign is $(-1)^{\# of \expnxarrow}$.
More formally, the signed set of arrow rows with length $n$ is $\mathrm{AR}_n=(\{\nwarrow,\nearrow\},\{\nxarrow\})^n$.
An arrow row $\mu$ with length $n$ acts on $\mathbf{k} \in \mathbb{Z}^n$ as follows:
\[
	\mu(\mathbf{k})=\sint{k_1+\delta_{\nearrow}(\mu_1)}{k_2-\delta_{\nwarrow}(\mu_2)}\times\cdots\times\sint{k_{n-1}+\delta_{\nearrow}(\mu_{n-1})}{k_n-\delta_{\nwarrow}(\mu_n)}
\]
where $\delta_{\nwarrow}(\nwarrow)=\delta_{\nwarrow}(\nxarrow)=1$, $\delta_{\nwarrow}(\nearrow)=0$ and $\delta_{\nearrow}(\nearrow)=\delta_{\nearrow}(\nxarrow)=1$, $\delta_{\nearrow}(\nwarrow)=0$.
Fischer and Konvalinka prove some recursive relation of monotone triangles (with their generalized definition of (modified) monotone triangles \cite[Problem 7]{FK1}.
This result can be translated to the following recursive relation:
\[
	\exists\, \Xi_\mathbf{k}^{(\textrm{FK})} \colon \textrm{MT}(\mathbf{k}) \sijto \bigsqcup_{\mu \in \textrm{AR}_n} \bigsqcup_{\mathbf{l} \in \mu(\mathbf{k})} \textrm{MT}(\mathbf{l}).
\]
We shall adopt the RHS of this relation as a definition of generalized monotone triangles.
Namely, our variant of the definition is described as follows.
\begin{definition}
For $\mathbf{k} \in \mathbb{Z}^n$, a generalized monotone triangle with bottom row $\mathbf{k}$ is a sequence
\[\left(\mu^{(n)},\mathbf{k}^{(n-1)},\mu^{(n-1)},\ldots,\mathbf{k}^{(1)},\mu^{(1)}\right) \in \mathbb{Z}^n \times \mathrm{AR}_n \times \mathbb{Z}^{n-1} \times \mathrm{AR}_{n-1} \times \cdots \times \mathbb{Z}^1 \times \mathrm{AR}_1,\]
such that for $1 \leq i \leq n-1$, $\mathbf{k}^{(i)} \in \mu^{(i+1)}(\mathbf{k}^{(i+1)})$. More formally, the signed set of generalized monotone triangles with bottom row $\mathbf{k}$ is 
\[
	\textrm{GMT}(\mathbf{k}) := \bigsqcup_{\mu^{(n)} \in \textrm{AR}_n} \bigsqcup_{\mathbf{k}^{(n-1)} \in \mu^{(n)}(\mathbf{k})} \bigsqcup_{\mu^{(n-1)} \in \textrm{AR}_{n-1}} \bigsqcup_{\mathbf{k}^{(n-2)} \in \mu^{(n-1)}(\mathbf{k}^{(n-1)})}
	\cdots \bigsqcup_{\mu^{(2)} \in \textrm{AR}_{2}} \bigsqcup_{\mathbf{k}^{(1)} \in \mu^{(2)}(\mathbf{k}^{(2)})} \textrm{AR}_1.
\]
\end{definition}
Another interpretation of a generalized monotone triangle is given by the (extented) arrowed monotone triangles which are introduced in \cite{AF}.
An arrowed monotone triangle is a triangular array of integers such that each entry is decorated by $\nwarrow$, $\nearrow$ or $\nxarrow$ with some condition.
The condition used in our variant is described as follows:
\begin{quote}
	For each entry $a$ that is not in bottom row, let $b$, $c$ be the left-below and the right-below entry and let $\mu_b$, $\mu_c$ be the decoration of $b$, $c$.
	Then, $a \in \sint{b+\delta_{\nearrow}(\mu_b)}{c-\delta_{\nwarrow}(\mu_c)}$ holds.
\end{quote}
For $\left(\mu^{(n)},\mathbf{k}^{(n-1)},\mu^{(n-1)},\ldots,\mathbf{k}^{(1)},\mu^{(1)}\right) \in \textrm{GMT}(\mathbf{k})$,
$\left(\mathbf{k}^{(1)},\mathbf{k}^{(2)},\ldots,\mathbf{k}^{(n-1)},\mathbf{k}\right)$ is an arrowed monotone triangle when it is decorated by $\left(\mu_1,\mu_2,\ldots,\mu_n\right)$.
In particular, $\left(\mathbf{k}^{(1)},\mathbf{k}^{(2)},\ldots,\mathbf{k}^{(n-1)},\mathbf{k}\right)$ is a monotone triangle with bottom row $\mathbf{k}$ when $\mathbf{k}$ is strictly increasing.
\begin{ex}
Let $\mathbf{k}=(1,3,5)$ and $\mu=(\nwarrow,\nxarrow,\nearrow)$.
Then, we have
\[
	\mu(\mathbf{k}) = \sint{1}{2} \times \sint{4}{5} = (\{(1,4)\},\emptyset).
\]
Next, let $\mathbf{k}^{(2)} = (1,4)$ and $\mu^{(2)} = (\nxarrow,\nxarrow)$, we have $\mu^{(2)}(\mathbf{k}^{(2)}) = \sint{2}{3}$. From the above, for example we have
\[
	((\nwarrow),(2),(\nxarrow,\nxarrow),(1,4),(\nwarrow,\nxarrow,\nearrow)) \insupp \text{GMT}((1,3,5)),
\]
and its sign is $-1$. With the interpretation as an arrowed monotone triangle, we sometimes describe it by
\[
	\left( \quad\begin{array}{ccccc}&&2&&\\&1&&4&\\1&&3&&5\end{array}\quad,
	\setlength{\arraycolsep}{0pt} 
	\quad\begin{array}{ccccc}&&\nwarrow&&\\&\nxarrow&&\nxarrow&\\\nwarrow&&\nxarrow&&\nearrow \end{array}\quad \right) \insupp \text{GMT}((1,3,5)).
\]
\end{ex}

To see that our generalized monotone triangle is indeed a generalization of a (modified) monotone triangle, we first construct a sijection
$
	\iota_{\textrm{MT}(\mathbf{k})} \colon \textrm{MT}(\mathbf{k}) \sijto \textrm{GMT}(\mathbf{k})
$
for a strictly increasing $\mathbf{k} \in \mathbb{Z}^n$.
\begin{construction}\label{iotamt}
\textit{
	We construct a sijection $\iota_{\textrm{MT}(\mathbf{k})}  \colon \textrm{MT}(\mathbf{k}) \sijto \textrm{GMT}(\mathbf{k})$ for a strictly increasing $\mathbf{k} \in \mathbb{Z}^n$.
}
For $n=1$, $\textrm{MT}(\mathbf{k})=(\{k_1\},\emptyset)$ and $\textrm{GMT}(\mathbf{k})=\textrm{AR}_1$. We define $\iota_{\textrm{MT}(\mathbf{k})}$ as
\[
\begin{tikzpicture}[main/.style = {draw, circle}]
\node (S) at (0,3) {$\text{MT}(\mathbf{k})$};
\node (T) at (3,3) {$\text{GMT}(\mathbf{k})$};
\node (phi) at (-1.9,3) {$\varphi \colon$};
\node (sij) at (1.5,3) {$\sijto$};
\node (+) at (-2,1.5) {$+$};
\node (-) at (-2,0) {$-$};
\draw (-2.5,-0.5) to (3.5,-0.5);
\draw (-2.5,0.5) to (3.5,0.5);
\draw (-2.5,2.5) to (3.5,2.5);
\node[scale=0.84,main] (1) at (0,2) {$k_1$};
\node[scale=0.84,main] (2) at (3,1) {$\nearrow$};
\node[scale=0.84,main] (3) at (3,0) {$\nxarrow$};
\node[scale=0.84,main] (4) at (3,2) {$\nwarrow$};
\draw[-] (1) -- (4);
\draw[-] (2) to [out=180,in=180] (3);
\end{tikzpicture}.
\]
For $n>1$, it is sufficient to construct a sijection $\Xi_\mathbf{k} \colon \textrm{MT}(\mathbf{k}) \sijto \bigsqcup_{\mu \in \textrm{AR}_n} \bigsqcup_{\mathbf{l} \in \mu(\mathbf{k})} \textrm{MT}(\mathbf{l})$.
Indeed, the construction is completed by induction on $n$:
\begin{equation} \label{iotamt::rec}
	\textrm{MT}(\mathbf{k})
	\sijto \bigsqcup_{\mu \in \textrm{AR}_n} \bigsqcup_{\mathbf{l} \in \mu(\mathbf{k})} \textrm{MT}(\mathbf{l})
	\sijto \bigsqcup_{\mu \in \textrm{AR}_n} \bigsqcup_{\mathbf{l} \in \mu(\mathbf{k})} \textrm{GMT}(\mathbf{l})
	= \textrm{GMT}(\mathbf{k}).
\end{equation}
Note that for any $\mu \in \mathrm{AR}_n$ and a strictly increasing $\mathbf{k} \in \mathbb{Z}^n$, every element of $\mu(\mathbf{k})$ is weakly increasing.
For an $\mathbf{l} \in \mu(\mathbf{k})$ that is not strictly increasing, it holds that $l_i = l_{i+1}$ for some $1 \leq i \leq n-1$, and thus $\text{MT}(\mathbf{l}) = (\emptyset,\emptyset)$ holds.
Therefore, we can obtain the result for general $n$ by the induction.

Let us now construct $\Xi_{\mathbf{k}}$.
First, we decompose $\textrm{MT}(\mathbf{k})$ with respect to the second bottom row. Let $\textrm{MT}(\mathbf{k},\mathbf{l})$ be a modified signed set of monotone triangles whose bottom row is $\mathbf{k}$ and whose second bottom row is $\mathbf{l}$,
then $$\mathrm{MT}(\mathbf{k})=\bigsqcup_{\mathbf{l} \in \mathbb{Z}^{n-1}} \mathrm{MT}(\mathbf{k},\mathbf{l}),$$
where $\mathbb{Z}^{n-1}$ is identified with the signed set $(\mathbb{Z}^{n-1},\emptyset)$, and the same identification shall be used in the following without mentioning.
Here, in case $\mathbf{l} \in \bigcup_{\mu \in \mathrm{AR}_n} \bs{\mu(\mathbf{k})}$ we have $\text{MT}(\mathbf{k},\mathbf{l})=\text{MT}(\mathbf{l})$,
and otherwise $\mathrm{MT}(\mathbf{k},\mathbf{l}) = (\emptyset,\emptyset)$.
Consequently, we have
\[
	\text{MT}(\mathbf{k})=\bigsqcup_{\mathbf{l} \in \bigcup_{\mu \in \mathrm{AR}_n} \bs{\mu(\mathbf{k})}} \mathrm{MT}(\mathbf{l}).
\]
Let $\mathscr{M}_\mathbf{l}$ be a signed set defined by $\mathscr{M}_\mathbf{l}^\pm=\{ \mu \in \mathrm{AR}_n^\pm \mid \mathbf{l} \in \bs{\mu(\mathbf{k})}\}$.
Then, we have
\[
	\bigsqcup_{\mu \in \textrm{AR}_n} \bigsqcup_{\mathbf{l} \in \mu(\mathbf{k})} \textrm{MT}(\mathbf{l})
	= \bigsqcup_{\mathbf{l} \in \mathbb{Z}^{n-1}} \bigsqcup_{\mu \in \mathscr{M}_\mathbf{l}} \text{MT}(\mathbf{l})
	= \bigsqcup_{\mathbf{l} \in \bigcup_{\mu \in \mathrm{AR}_n} \bs{\mu(\mathbf{k})}} \bigsqcup_{\mu \in \mathscr{M}_\mathbf{l}} \text{MT}(\mathbf{l}).
\]
Therefore, it is sufficient to construct a sijection
$(\{\cdot\},\emptyset) \sijto \mathscr{M}_\mathbf{l}$ for $\mathbf{l} \in \bigcup_{\mu \in \mathrm{AR}_n} \bs{\mu(\mathbf{k})}$.
Now, $\bs{\mathscr{M}_\mathbf{l}}$ is described as follows, when it is not empty:
\[
	\bs{\mathscr{M}_\mathbf{l}} = \left\{ \mu \in \mathrm{AR}_n \middle\vert \begin{array}{l} \mu_i=\nwarrow \text{when $i \leq n-1$ and $k_i=l_i$},\\ \mu_i=\nearrow \text{when $i \geq 2$ and $k_i=l_{i-1}+1$} \\\text{otherwise, $\mu_i \in \{ \nwarrow, \nearrow, \nxarrow\}$} \end{array}\right\}.
\]
Let $\mu_\mathbf{l}$ be an arrow row such that
\[
	(\mu_\mathbf{l})_i = \begin{cases} \nearrow & \text{when $i \geq 2$ and $k_i=l_{i-1}+1$}, \\ \nwarrow & \text{otherwise}, \end{cases}
\]
and the following rule then defines a sijection $(\{\cdot\},\emptyset) \sijto \mathscr{M}_\mathbf{l}$:
\begin{quote}
	An arrow row $\mu_\mathbf{l} \in \mathscr{M}_\mathbf{l}^+$ corresponds to a unique element of LHS of the sijection.
	For any other element $\mu$ of $\mathscr{M}_\mathbf{l}$, there exists a minimum index $i$ such that $\mu_i \ne (\mu_\mathbf{l})_i$.
	The counterpart of $\mu$ is given by adding/removing $\nwarrow$ to/from $\mu_i$. \qed
\end{quote}
\end{construction}

\begin{remark}
This construction is similar to Problem 7 in \cite{FK1}.
The most important difference with that reference is the definition of $(\mu_\mathbf{l})_n$, which is crucial from the viewpoint of compatibility.
For more details, see Proposition \ref{prop::GIN_iota}.
\end{remark}

Next, we define some statistics.
\begin{definition}\label{def:eta_TA}
Let $\mathbf{k} \in \mathbb{Z}^n$. Define $\eta_\text{TA}$ as the triangular array-valued statistic of $\text{GMT}(\mathbf{k})$ such that
\[
\forall s = \left(\mu^{(n)},\mathbf{k}^{(n-1)},\mu^{(n-1)},\ldots,\mathbf{k}^{(1)},\mu^{(1)}\right) \in \textrm{GMT}(\mathbf{k}),\,\, \eta_\text{TA}(s)=\left(\mathbf{k}^{(1)},\mathbf{k}^{(2)},\ldots,\mathbf{k}^{(n-1)},\mathbf{k}\right).
\]
In addition, for a triangular array $T$, define $\eta_\text{TA}(T) = T$.
\end{definition}

Last, the following proposition shows that the generalized monotone triangles are indeed a generalization of monotone triangles.
\begin{proposition}\label{eta_MT}
Let $\mathbf{k} \in \mathbb{Z}^n$ such that $\mathbf{k}$ is strictly increasing.
Then, $\iota_{\textrm{MT}(\mathbf{k})} \colon \textrm{MT}(\mathbf{k}) \sijto \textrm{GMT}(\mathbf{k})$ in Construction \ref{iotamt} is compatible with $\eta_\text{TA}$.
\end{proposition}
\begin{proof}
For $n=1$, this is trivial. For $n>1$, the second bottom row is preserved, because both sides are first decomposed by the second bottom row in the construction. Therefore, according to (\ref{iotamt::rec}), the proof is completed by induction on $n$.
\end{proof}

The top of a generalized monotone triangle $T$ is defined by $\eta_\text{top} = \eta_\text{top}(\eta_\text{TA}(T))$.
According to Proposition \ref{eta_MT}, the sijection $\iota_{\textrm{MT}(\mathbf{k})}$ in Construction \ref{iotamt} is compatible with $\eta_\text{top}$.

\begin{Remark} \label{uniqMT}
In \cite{FK1}, the definition of a monotone triangle with general bottom row is given as a triangular array.
Accordingly, one can extend the definition of a modified monotone triangle for general bottom row as follows.
First one defines a binary relation $ \mathbf{l} \prec \mathbf{k} $ for $\mathbf{l} \in \mathbb{Z}^{n-1}$ and $\mathbf{k} \in \mathbb{Z}^n$ as:
\[
\mathbf{l} \prec \mathbf{k} \Leftrightarrow 
\begin{cases}
\text{for $i \in \{1,2,\ldots,n-1\}$, $l_i$ is in the closed interval between $k_i$ and $k_{i+1}-1$,} \\
\text{if $k_{i-1} < k_i < k_{i+1}$ for some $i \in \{2,\ldots,n-1\}$, then $l_{i-1} \ne k_i-1$ or $l_i \ne k_i$,} \\
\text{if $k_i > l_i = k_{i+1}-1$ for some $i \in \{1,\ldots,n-1\}$, then $i \leq n-2$ and $l_{i+1}=k_{i+1}$,} \\
\text{if $k_i = l_i \geq k_{i+1}$ for some $i \in \{1,\ldots,n-1\}$, then $i \geq 2$ and $l_{i-1}=k_i-1$.}
\end{cases}
\]
Then, a modified monotone triangle with bottom row $\mathbf{k}$ is a triangular array $(k_1,k_2,\ldots,k_n) \in \mathbb{Z} \times \mathbb{Z}^2 \times \cdots \times \mathbb{Z}^n$ such that $k_1 \prec k_2 \prec \cdots \prec k_n=k$.
The sign of a modified monotone triangle $(k_1,k_2,\ldots,k_n)$ is $(-1)^r$, where $r$ is the sum of:
\begin{itemize}
\item the number of $(i,j)$ such that $1 \leq j < i \leq n$ and $k_{i,j} \geq k_{i,j+1}$,
\item the number of $(i,j)$ such that $1 \leq j+2 \leq i \leq n$, $k_{i,j} \geq k_{i,j+1} \geq k_{i,j+2}$, $k_{i-1,j} = k_{i,j+1}-1$ and $k_{i-1,j+1} = k_{i,j+1}$.
\end{itemize}
Let $\text{MT}^\text{(FK)}(\mathbf{k})$ be the signed set of modified monotone triangles with bottom row $\mathbf{k}$ according to this definition.
This definition might seem strange but it is in fact correct.
The purpose of this remark is to explain why, by showing uniqueness in a certain sense.
Let us extend the definition of modified monotone triangles independently of Fischer and Konvalinka's work, but rather according to Construction \ref{iotamt}.
First, we define $\mathscr{M}_{\mathbf{k},\mathbf{l}}$ as:
\begin{itemize}
\item $\bs{\mathscr{M}_{\mathbf{k},\mathbf{l}}} = \{ \mu \in \bs{\mathrm{AR}_n} \mid \mathbf{l} \in \bs{\mu(\mathbf{k})}\}$,
\item The sign of $\mu \in \bs{\mathscr{M}_{\mathbf{k},\mathbf{l}}}$ is the product of:
\begin{itemize}
\item the sign of $\mu$ as an element of $\text{AR}_n$,
\item the sign of $\mathbf{l}$ as an element of $\mu(\mathbf{k})$.
\end{itemize}
\end{itemize}
In addition, we define $\mathbf{l} \prec' \mathbf{k} \Leftrightarrow \#\mathscr{M}_{\mathbf{k},\mathbf{l}} \ne 0$ and define a modified monotone triangle as a triangular array $(k_1,k_2,\ldots,k_n) \in \mathbb{Z} \times \mathbb{Z}^2 \times \cdots \times \mathbb{Z}^n$ such that $k_1 \prec' k_2 \prec' \cdots \prec' k_n=k$.
Note that when $\mathbf{k}$ is strictly increasing, this definition coincides with the original one. Let $X(\mathbf{k})$ be a set of modified monotone triangles for $\mathbf{k} \in \mathbb{Z}^n$ according to this definition.
If $\#\mathscr{M}_{\mathbf{k},\mathbf{l}} = 0, \pm 1$ holds for any $\mathbf{l} \in \mathbb{Z}^{n-1}$ and $\mathbf{k} \in \mathbb{Z}^n$ such that $X(\mathbf{l}) \ne \emptyset$,
then there exists a way to assign signs to the monotone triangles such that one can construct the sijection
\[
	\textrm{MT}(\mathbf{k}) \sijto \bigsqcup_{\mu \in \textrm{AR}_n} \bigsqcup_{\mathbf{l} \in \mu(\mathbf{k})} \textrm{MT}(\mathbf{l}) \quad \left( = \bigsqcup_{\mathbf{l} \in \mathbb{Z}^{n-1}} \bigsqcup_{\mu \in \mathscr{M}_{\mathbf{k},\mathbf{l}}} \text{MT}(\mathbf{l}) \right),
\]
in the way similar to Construction \ref{iotamt}, where $\text{MT}(\mathbf{k})$ is a signed set of monotone triangles with that signs.
We prove that when $X(\mathbf{l}) \ne \emptyset$ it holds that $\#\mathscr{M}_{\mathbf{k},\mathbf{l}} = 0, \pm 1$.
First, we define a function $f \colon \bs{\text{AR}_n} \to \{0,1\}^{2n}$ by
\[
	\mu \mapsto \left(\delta_\nwarrow(\mu_1),\,\delta_\nearrow(\mu_1),\,\delta_\nwarrow(\mu_2),\,\delta_\nearrow(\mu_2),\,\ldots,\,\delta_\nwarrow(\mu_n),\,\delta_\nearrow(\mu_n)\right).
\]
Then, the image of $f$ is
\[
	\left\{ (a_{1,0},a_{1,1},a_{2,0},a_{2,1},\ldots,a_{n,0},a_{n,1}) \in \{0,1\}^{2n} \mid \forall i \in \{1,2,\ldots,n\},\, a_{i,0}+a_{i,1} \geq 1 \right\}.
\]

We fix $\mathbf{l} \in \mathbb{Z}^{n-1}$ and $\mathbf{k} \in \mathbb{Z}^n$.
For $\mu \in \bs{\text{AR}_n}$ to be included $\bs{\mathscr{M}_{\mathbf{k},\mathbf{l}}}$ we have to impose some conditions on it.
In detail, the conditions are described as follows:
\begin{quote}
\centering
\begin{tabular}{llcc}
\hline
& & conditions on $a$ & trans. mat. \\ \hline
$k_{i+1} \geq k_i+2$, & $l_i = k_i$ & $a_{i,1}=0$ & $A_1$ \\
 & $l_i = k_{i+1}-1$ & $a_{i+1,0}=0$ & $A_2$ \\
 & $k_i < l_i < k_{i+1}-1$ & True & $A_3$ \\
 & otherwise & False & $O$ \\\hline
$k_{i+1} = k_i+1$, & $l_i = k_i = k_{i+1}-1$ & $a_{i,1}=a_{i+1,0}$ & $A_4$ \\
 & otherwise & False & $O$ \\\hline
$k_{i+1} \leq k_i$, & $l_i = k_i$ & $a_{i,1}=1$ & $-A_5$ \\
 & $l_i = k_{i+1}-1$ & $a_{i+1,0}=1$ & $-A_6$ \\
 & $k_{i+1}-1 < l_i < k_i$ & True & $-A_3$ \\
 & otherwise & False & $O$ \\\hline
\end{tabular}
\end{quote}
where $i \in \{1,2,\ldots,n-1\}$, $(a_{1,0},a_{1,1},a_{2,0},a_{2,1},\ldots,a_{n,0},a_{n,1}) = f(\mu)$ and the data in the last column will be explained later.
For $j \in \{1,2,\ldots,n\}$ and $t \in \{0,1\}$, we define signed sets $B_{j,t}$ by
\begin{multline*}
	\bs{B_{j,t}} = \#\bigl\{ (a_{1,0},a_{1,1},a_{2,0},a_{2,1},\ldots,a_{j,0}) \in \{0,1\}^{2j-1}  \mid a_{j,0}=t \\ \text{ and for $i\in \{1,2,\ldots,j-1\}$}, \text{ it meets the condition described above and $a_{i,0}+a_{i,1} \geq 1$  } \bigr\},
\end{multline*}
with the sign of $(a_{1,0},a_{1,1},a_{2,0},a_{2,1},\ldots,a_{j,0})$ being $(-1)^r$, where $r$ is the sum of:
\begin{itemize}
\item the number of integers $i\in \{1,2,\ldots,j-1\}$ such that $a_{i,0}=a_{i,1}=1$,
\item the number of integers $i\in \{1,2,\ldots,j-1\}$ such that $k_i+a_{i,1} > k_{i+1} - a_{i+1,0}$.
\end{itemize}
We denote the size of $B_{j,t}$ by $b_{j,t}$.
Furthermore, we define the transition matrix by the relation
\[
	\begin{pmatrix} b_{j+1,0} \\ b_{j+1,1} \end{pmatrix} = T_j \begin{pmatrix} b_{j,0} \\ b_{j,1} \end{pmatrix},
\]
where $j \in \{1,2,\ldots,n-1\}$.
Then the value of $\#\mathscr{M}_{\mathbf{k},\mathbf{l}}$ is expressed as
\[
	\#\mathscr{M}_{\mathbf{k},\mathbf{l}} = \begin{pmatrix} 1 & 0 \end{pmatrix} T_{n-1} T_{n-2} \cdots T_1 \begin{pmatrix} 1 \\ 1 \end{pmatrix}.
\]
We can calculate $T_j$ from the value of $k_j$, $k_{j+1}$ and $l_j$.
The results are described in the last column of the above table, where
\begin{align*}
A_1 &= \begin{pmatrix} 0 & 1 \\ 0 & 1 \end{pmatrix}, &
A_2 &= \begin{pmatrix} 1 & 0 \\ 0 & 0 \end{pmatrix}, &
A_3 &= \begin{pmatrix} 1 & 0 \\ 1 & 0 \end{pmatrix}, \\
A_4 &= \begin{pmatrix} 0 & 1 \\ -1 & 1 \end{pmatrix}, &
A_5 &= \begin{pmatrix} 1 & -1 \\ 1 & -1 \end{pmatrix}, &
A_6 &= \begin{pmatrix} 0 & 0 \\ 1 & 0 \end{pmatrix}.
\end{align*}

Here, in order to use the condition $X(\mathbf{l}) \ne \emptyset$, we prove the following lemma.
For $\mathbf{m} = (m_i)_i \in \mathbb{Z}^d$, if there exists an integer $1 \leq i \leq d-2$ such that $m_i = m_{i+1}-1 = m_{i+2}-2$, we say it is partially successive.
\begin{lemma}
For a partially successive sequence $\mathbf{m} \in \mathbb{Z}^d$, we have $X(\mathbf{m}) = \emptyset$.
\end{lemma}
\begin{proof}
We assume, for the sake of contradiction, that there exists a sequence $\mathbf{m}^{(1)} \prec' \mathbf{m}^{(2)} \prec' \cdots \prec' \mathbf{m}^{(d)} = \mathbf{m}$.
Let $i$ be the minimum index such that $\mathbf{m}^{(i)}$ is partially successive and $j$ the minimum index such that $m^{(i)}_j = m^{(i)}_{j+1}-1 = m^{(i)}_{j+2}-2$.
We calculate $\#\mathscr{M}_{\mathbf{m}^{(i)},\mathbf{m}^{(i-1)}}$. We set $(\mathbf{k},\mathbf{l}) = (\mathbf{m}^{(i)},\mathbf{m}^{(i-1)})$ in the above context and consider the transition matrices.
First, we have $T_j=T_{j+1}=A_4$ if they are not $O$ because of the choice of $j$.
Furthermore, we have $m^{(i-1)}_j = m^{(i)}_j = m^{(i-1)}_{j+1}-1 = m^{(i)}_{j+1}-1$.
Therefore, $m^{(i-1)}_{j+2} \ne m^{(i)}_{j+2}$ if it exists because of the choice of $i$.
Thus, we have $T_{j+2} = A_2, \pm A_3, -A_6, O$ or $j+2=i$.
Similarly, if $j \ne 1$ then $T_{j-1}=A_1,\pm A_3, -A_5, O$.
Here, we can decompose $A_2,A_3,A_6,O$ as
\begin{align*}
A_2 = \begin{pmatrix} 1 \\ 0 \end{pmatrix} \begin{pmatrix} 1 & 0 \end{pmatrix},&&
A_3 = \begin{pmatrix} 1 \\ 1 \end{pmatrix} \begin{pmatrix} 1 & 0 \end{pmatrix},&&
A_6 = \begin{pmatrix} 0 \\ 1 \end{pmatrix} \begin{pmatrix} 1 & 0 \end{pmatrix},&&
O = \begin{pmatrix} 0 \\ 0 \end{pmatrix} \begin{pmatrix} 1 & 0 \end{pmatrix},&&
\end{align*}
and $A_1,A_3,A_5,O$ as
\begin{align*}
A_1 = \begin{pmatrix} 1 \\ 1 \end{pmatrix} \begin{pmatrix} 0 & 1 \end{pmatrix},&&
A_3 = \begin{pmatrix} 1 \\ 1 \end{pmatrix} \begin{pmatrix} 1 & 0 \end{pmatrix},&&
A_5 = \begin{pmatrix} 1 \\ 1 \end{pmatrix} \begin{pmatrix} 1 & -1 \end{pmatrix},&&
O = \begin{pmatrix} 1 \\ 1 \end{pmatrix} \begin{pmatrix} 0 & 0 \end{pmatrix}.&&
\end{align*}
Therefore, we have
\[
	\#\mathscr{M}_{\mathbf{m}^{(i)},\mathbf{m}^{(i-1)}} = \begin{pmatrix} 1 & 0 \end{pmatrix} T_{i-1} T_{i-2} \cdots T_1 \begin{pmatrix} 1 \\ 1 \end{pmatrix} = 0.
\]
because $\begin{pmatrix} 1 & 0 \end{pmatrix} A_4^2 \begin{pmatrix} 1 \\ 1 \end{pmatrix} = 0$.
This means that $\mathbf{m}^{(i-1)} \not\prec' \mathbf{m}^{(i)}$, and it contradicts the assumption.
Finally, we obtain $X(\mathbf{m}) = \emptyset$, for a partially successive sequence $\mathbf{m}$.
\end{proof}
Coming back to the purpose of this remark, according to this lemma, it is sufficient to show that $\#\mathscr{M}_{\mathbf{k},\mathbf{l}} = 0, \pm 1$ when $\mathbf{l}$ is not partially successive.
We assume that $\mathbf{l}$ is not partially successive and $T_j=A_4$ for some $j$.
Then, we have $k_j = l_j = k_{j+1}-1$, so $j=1,n-1$ or $l_{j-1} \ne k_j-1$ or $l_{j+1} \ne k_{j+1}$.
Therefore, we have $j=1,n-1$ or $T_{j-1}=A_1,\pm A_3, -A_5, O$ or $T_{j+1} = A_2, \pm A_3, -A_6, O$.
Here, we have
\[
	\begin{pmatrix} 1 & 0 \end{pmatrix} A_4 = \begin{pmatrix} 0 & 1 \end{pmatrix},\qquad
	A_4 \begin{pmatrix} 1 \\ 1 \end{pmatrix} = \begin{pmatrix} 1 \\ 0 \end{pmatrix}
\]
and
\begin{align*}
A_2A_4 &= \begin{pmatrix} 0 & 1 \\ 0 & 0 \end{pmatrix} =: A_7,&
A_3A_4 &= \begin{pmatrix} 0 & 1 \\ 0 & 1 \end{pmatrix} = A_1,&
A_6A_4 &= \begin{pmatrix} 0 & 0 \\ 0 & 1 \end{pmatrix} =: A_8,\\
A_4A_1 &= \begin{pmatrix} 0 & 1 \\ 0 & 0 \end{pmatrix} = A_7,&
A_4A_3 &= \begin{pmatrix} 1 & 0 \\ 0 & 0 \end{pmatrix} = A_2,&
A_4A_5 &= \begin{pmatrix} 1 & -1 \\ 0 & 0 \end{pmatrix} =: A_9.
\end{align*}
Therefore, because the set $\left\{ \begin{pmatrix} 0 \\ 0 \end{pmatrix}, \pm\begin{pmatrix} 0 \\ 1 \end{pmatrix}, \pm\begin{pmatrix} 1 \\ 0 \end{pmatrix}, \pm\begin{pmatrix} 1 \\ 1 \end{pmatrix}\right\}$ is closed under the action of $\pm A_1,\pm A_2, \pm A_3,\pm A_5$, $\pm A_6,\pm A_7, \pm A_8,\pm A_9,O$, we obtain 
\[
	\#\mathscr{M}_{\mathbf{k},\mathbf{l}} = \begin{pmatrix} 1 & 0 \end{pmatrix} T_{n-1} T_{n-2} \cdots T_1 \begin{pmatrix} 1 \\ 1 \end{pmatrix} = 0,\pm 1.
\]

From the above, we can define the signed set $\text{MT}(\mathbf{k})$ of modified monotone triangles with bottom row $\mathbf{k}$ as follows:
\begin{itemize}
\item $\bs{\text{MT}(\mathbf{k})} = X(\mathbf{k})$, 
\item the sign of an element $(k_1,k_2,\ldots,k_n) \insupp \text{MT}(\mathbf{k})$ is $\#\mathscr{M}_{k_n,k_{n-1}} \#\mathscr{M}_{k_{n-1},k_{n-2}} \cdots \#\mathscr{M}_{k_2,k_1}$.
\end{itemize}
With this definition, we can construct a sijection $\iota \colon \text{MT}(\mathbf{k}) \sijto \text{GMT}(\mathbf{k})$ which is compatible with $\eta_\text{TA}$ in a way similar to Construction \ref{iotamt}.
In \cite{FK1}, Fischer and Konvalinka construct the sijection equivalent to $\iota \colon \text{MT}^\text{(FK)}(\mathbf{k}) \sijto \text{GMT}(\mathbf{k})$. We can check it is also compatible with $\eta_\text{TA}$.
Since the restrictions of $\eta_\text{TA}$ to $X(\mathbf{k})$ or to $\bs{\text{MT}^\text{(FK)}(\mathbf{k})}$ are injective, the two definitions of modified monotone triangles coincide.

In fact, the binary relation $\mathbf{l} \prec' \mathbf{k}$ is a bit different from the original one $\mathbf{l} \prec \mathbf{k}$.
After some cumbersome calculations, we obtain
\begin{multline*}
	\mathbf{l} \prec' \mathbf{k} \Leftrightarrow \mathbf{l} \prec \mathbf{k} \text{ and for any $i \in \{ 1,2,\ldots,n-3\}$, it does not hold that}\\ \text{ $l_i=k_{i+1}-1=l_{i+1}-1=k_{i+2}-2=l_{i+2}-2$}.
\end{multline*}
However, if $\mathbf{l} \not\prec' \mathbf{k}$ and $\mathbf{l} \prec \mathbf{k}$, then $\mathbf{l}$ is partially successive.
Therefore, the difference does not matter in the definition of modified monotone triangles.
Also, we can recover the explicit (namely, without the language of $\mathscr{M}$s) definition of the sign of an element of $\text{MT}(\mathbf{k})$ through direct calculation.
\end{Remark}

\subsubsection{Shifted Gelfand-Tsetlin Patterns}
Next, we define a shifted Gelfand-Tsetlin pattern.
An \textit{arrow pattern} $T=(t_{i,j})_{1 \leq i<j \leq n}$ is a triangular arrangement of $\{\searrow,\swarrow,\sxarrow\}$ that is indexed as below:
\[
\begin{array}{ccccccccc}
&&&&t_{1,n}&&&&\\
&&&t_{1,n-1}&&t_{2,n}&&&\\
&&t_{1,n-2}&&t_{2,n-1}&&t_{3,n}&&\\
&\iddots&&\vdots&&\vdots&&\ddots&\\
t_{1,2}&\cdots&\cdots&\cdots&\cdots&\cdots&\cdots&\cdots&t_{n-1,n}.
\end{array}
\]
The sign of an arrow pattern is defined as $(-1)^{\# of \expsxarrow}$,
or more formally, the signed set of arrow patterns with size $n$ is
\[\mathrm{AP}_n=(\{\searrow,\swarrow\},\{\sxarrow\})^{n-1} \times (\{\searrow,\swarrow\},\{\sxarrow\})^{n-2} \times \cdots \times (\{\searrow,\swarrow\},\{\sxarrow\})^{1}\simeq(\{\searrow,\swarrow\},\{\sxarrow\})^{\binom{n}{2}},\]
where $\mathrm{AP}_1=(\{\emptyset\},\emptyset)$.
An arrow pattern $T=(t_{i,j})_{1 \leq i<j \leq n}$ acts on $\mathbf{k} \in \mathbb{Z}^n$ as:
\begin{align*}
	c_i(T)&=\sum_{j=i+1}^{n} \delta_{\swarrow}(t_{i,j}) - \sum_{j=1}^{i-1} \delta_{\searrow}(t_{i,j}) \\
	T(\mathbf{k})&=(k_1+c_1(T),k_2+c_2(T),\ldots,k_n+c_n(T)) \in \mathbb{Z}^n,
\end{align*}
where $\delta_{\searrow}(\searrow)=\delta_{\searrow}(\sxarrow)=1$, $\delta_{\searrow}(\swarrow)=0$ and $\delta_{\swarrow}(\swarrow)=\delta_{\swarrow}(\sxarrow)=1$, $\delta_{\swarrow}(\searrow)=0$.
Lastly, the signed set of shifted Gelfand-Tsetlin patterns with bottom row $\mathbf{k} \in \mathbb{Z}^n$ is defined as
\[
	\textrm{SGT}(\mathbf{k})=\bigsqcup_{T \in \textrm{AP}_n} \textrm{GT}(T(\mathbf{k})).
\]

We define the top of a shifted Gelfand-Tsetlin pattern $(A,T) \insupp \text{SGT}(\mathbf{k})$ by $\eta_\text{top}((A,T)) = \eta_\text{top}(A)$.
Note that $A$ is an element of $\bs{\text{GT}(T(\mathbf{k}))}$.
\begin{ex}
Let $T=((\searrow,\sxarrow,\swarrow),(\sxarrow,\swarrow),(\swarrow)) \insupp \text{AP}_4$. When we arrange its elements as
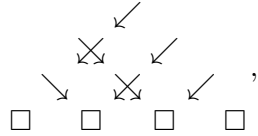
\[
	\setlength{\arraycolsep}{1.5pt} 
	\begin{array}{ccccccc}&&&\swarrow&&&\\&&\sxarrow&&\swarrow&&\\&\searrow&&\sxarrow&&\swarrow&\\ \square&&\square&&\square&&\square \end{array},
\]
$c_i(T)$ is the number of arrows on the anti-diagonal line of the $i$-th square from the left and pointing at it, minus the number of arrows on the diagonal line pointing at the square. Then, we have
\begin{align*}
	c_1(T)=2-0=2,&&c_2(T)=2-1=1,&&c_3(T)=1-2=-1,&&c_4(T)=0-0=0.
\end{align*}
Therefore, when let $\mathbf{k}=(3,1,4,1)$, then $T(\mathbf{k}) = (5,2,3,1)$. Thus, for example we have
\[
	\left(\quad\begin{array}{ccccccc}&&&2&&&\\&&3&&1&&\\&4&&2&&1&\\ 5&&2&&3&&1 \end{array}\quad,
	\setlength{\arraycolsep}{1.5pt} 
	\quad\begin{array}{ccccc}&&\swarrow&&\\&\sxarrow&&\swarrow&\\\searrow&&\sxarrow&&\swarrow \end{array}\quad\right) \insupp \text{SGT}((3,1,4,1)),
\]
and its sign is $(-1)^2 \cdot (-1)^5 = 1$.
Last, the value of $\eta_\text{top}$ on this element is $2$.
\end{ex}
\begin{remark}
A shifted Gelfand-Tsetlin pattern is introduced in \cite{FK1} to explain the combinatorial meaning of an `operator formula' stating that
\[
	\text{MT}^{\text{(g)}}(\mathbf{k}) = \prod_{1 \leq p<q \leq n}(E_{k_p} + E_{k_q}^{-1} - E_{k_p} E_{k_q}^{-1}) \prod_{1 \leq i < j \leq n} \frac{k_j-k_i}{j-i},
\]
where $E_x$ denotes the shift operator, i.e., $E_xp(x)=p(x+1)$ (This is cited from the theorem in the induction of \cite{FK1}).
It holds that $\text{GT}(\mathbf{k}) = \prod_{1 \leq i < j \leq n} \frac{k_j-k_i}{j-i}$, and the action of an arrow pattern on $\mathbf{k}$ corresponds to the operator part $\prod_{1 \leq p<q \leq n}(E_{k_p} + E_{k_q}^{-1} - E_{k_p} E_{k_q}^{-1})$.
In this sense, shifted Gelfand-Tsetlin patterns are a combinatorial realization of this operator formula.
\end{remark}

\subsection{Generalized Inversion Numbers}\label{ssec_GIN}
In this subsection, we define generalized inversion numbers of alternating sign matrices, modified monotone triangles, generalized monotone triangles and shifted Gelfand-Tsetlin patterns.
First, we recall the definition of an inversion number of a permutation and a permutation matrix.
For a permutation $p = (p_1,p_2,\ldots,p_n) \in \mathfrak{S}_n$, its inversion number is the number of pairs $(i,j)$ such that $1 \leq i<j \leq n$ and $p(i)>p(j)$.
A permutation matrix $A=(a_{ij})_{ij}$ corresponding to $p$ is defined by $a_{ij}=\delta_{j,p(i)}$, and its inversion number is that of $p$ itself.
In fact, the inversion number $\eta_\text{inv}(A)$ can be described by using the elements of $A$ as below \cite{BFZ1}:
\[
	\eta_\text{inv}(A) = \sum_{1 \leq i < i' \leq n,\, 1 \leq j' \leq j \leq n} a_{ij}a_{i'j'}.
\]
This is also the definition of the inversion number for alternating sign matrices. Namely, for an alternating sign matrix $A=(a_{ij})_{ij}$, we define its inversion number by the above equation.

\begin{remark}
This statistic is one of four statistics in Behrend, Di Francesco and Zinn-Justin's refined enumeration of alternating sign matrices and descending plane partitions \cite{BFZ1,BFZ2}, as well as one of three statistics in Mills, Robbins and Rumsey Jr.'s conjecture \cite{MRR}.
As stated in \cite{BFZ1}, there exists a variant of an inversion number for alternating sign matrices, \[\eta_\textrm{inv'}(A) = \sum_{1 \leq i < i' \leq n,\, 1 \leq j' < j \leq n} a_{ij}a_{i'j'}.\]
However, we adopt the former definition because the corresponding statistic for shifted Gelfand-Tsetlin patterns can be described simply (See \ref{GIN_SGT}).
\end{remark}

\subsubsection{Generalized Inversion Numbers for Modified Monotone Triangles}
We define the inversion number $\eta_\text{inv}(B)$ of a modified monotone triangle $B = (b_{ij})_{1 \leq j \leq i \leq n}$ by
\[
	\eta_\text{inv}(B) = \# \{ (i,j) \mid 1 \leq j \leq i \leq n-1,\, b_{i+1,j} \leq b_{i,j} = b_{i+1,j+1}-1 \}.
\]
This definition is given in \cite{Fis18} without details.
In the rest of this subsubsection, we will show that this definition is compatible with the bijection introduced in \ref{def_MT}.
Namely, let $A$ be an alternating sign matrix and $B$ the modified monotone triangle corresponding to $A$. We shall prove that $\eta_\text{inv}(A) = \eta_\text{inv}(B)$.
First, we have
\begin{align*}
	\eta_\text{inv}(A) &= \sum_{1 \leq i < i' \leq n,\, 1 \leq j' \leq j \leq n} a_{ij}a_{i'j'} \\
	&=\sum_{1 \leq i' \leq n} \sum_{1 \leq j \leq n} \sum_{1 \leq i < i'} \sum_{1 \leq j' \leq j} a_{ij}a_{i'j'} \\
	&=\sum_{1 \leq i' \leq n} \sum_{1 \leq j \leq n} \left(\sum_{1 \leq i < i'} a_{ij} \right) \left(\sum_{1 \leq j' \leq j} a_{i'j'} \right).
\end{align*}
Let $\displaystyle c_{i,j}=\sum_{1 \leq i' \leq i} a_{i'j}$ and $d_{i,j} = b_{ij}-(j-1)$. Note that these coincide with elements of intermediate products in the bijection between alternating sign matrices and modified monotone triangles (See \ref{def_MT}).
Here, we have $c_{ij} \in \{0,1\}$ by the definition of an alternating sign matrix.
Then, we have
\[
	\sum_{1 \leq i < i'} a_{ij} = 1 \Leftrightarrow c_{i'-1,j}=1 \Leftrightarrow \exists \tilde{j} \in \{1,2,\ldots,i'-1\}, d_{i'-1,\tilde{j}} = j.
\]
Therefore, we have
\begin{align*}
	\eta_\text{inv}(A)
	&= \sum_{1 \leq i' \leq n}\quad \sum_{j=d_{i'-1,1},d_{i'-1,2},\ldots,d_{i'-1,i'-1}} \left(\sum_{1 \leq j' \leq j} a_{i'j'} \right) \\
	&= \sum_{1 \leq \tilde{j} \leq i' \leq n-1} \left(\sum_{1 \leq j' \leq d_{i'\tilde{j}}} a_{i'+1,j'} \right) \\
	&= \sum_{1 \leq j \leq i \leq n-1} \left(\sum_{1 \leq j' \leq d_{ij}} a_{i+1,j'} \right).
\end{align*}
We fix $(i,j)$ such that $1 \leq j \leq i \leq n-1$ and let $k = d_{ij}$.
We prove
\[
	\sum_{1 \leq j' \leq k} a_{i+1,j'} = 1 \Leftrightarrow d_{i+1,j} \leq d_{i,j} = d_{i+1,j+1},
\]
by considering the following two possibilities.
\begin{itemize}
\item Case 1: Assume that $a_{i+1,k}=-1$. By the definition of an alternating sign matrix, we have $\displaystyle \sum_{1 \leq j' \leq k} a_{i+1,j'}=0$ and $c_{i+1,k}=0$.
In particular, $c_{i+1,k}=0$ means $d_{i+1,j+1} \ne k$. Therefore both the RHS and the LHS are false.
\item Case 2: Assume that $a_{i+1,k} \in \{0,1\}$. 
Since $k=d_{i,j}$, we have $c_{i,k}=1$. Therefore, we have $c_{i+1,k}=1$ and $a_{i+1,k}=0$.
In particular, the RHS is equivalent to $\#\{ j' \mid 1 \leq j' \leq k,\, c_{i+1,j'} =1 \} = j+1$,
and then it is equivalent to $\sum_{i'=1}^{i+1} \sum_{j'=1}^k a_{i'j'} = j+1$.
Since $\sum_{i'=1}^{i} \sum_{j'=1}^k a_{i'j'} = \sum_{j'=1}^k c_{ij'} = j$, it is equivalent to the LHS.
\end{itemize}
From the above, we have
\begin{align*}
	\eta_\text{inv}(A) &= \# \{ (i,j) \mid 1 \leq j \leq i \leq n-1,\, d_{i+1,j} \leq d_{i,j} = d_{i+1,j+1} \} \\
	&= \# \{ (i,j) \mid 1 \leq j \leq i \leq n-1,\, b_{i+1,j} \leq b_{i,j} = b_{i+1,j+1}-1 \} =: \eta_\text{inv}(B).
\end{align*}

\subsubsection{Generalized Inversion Numbers for Generalized Monotone Triangles}
We define the inversion number of a generalized monotone triangle as the number of $\nearrow$ and $\nxarrow$ it includes.
More formally,
we define the inversion number of a generalized monotone triangle $$T=\left(\mu^{(n)},\mathbf{k}^{(n-1)},\mu^{(n-1)},\ldots,\mathbf{k}^{(1)},\mu^{(1)}\right) \in \textrm{GMT}(\mathbf{k})$$ by
\[
	\eta_\text{inv}(T) = \#\left\{ (i,j) \mid 1 \leq j \leq i \leq n,\, \mu^{(i)}_j = \nearrow \text{ or } \nxarrow \right\}.
\]
The following proposition guarantees that this definition is indeed a generalization of the inversion number for modified monotone triangles.
For this proof, we define the inversion number of an arrow row $\mu \in \text{AR}_n$ as the number of $\nearrow$ and $\nxarrow$ it includes.
\begin{proposition}\label{prop::GIN_iota}
Let $n \in \mathbb{Z}$ and let $\mathbf{k} \in \mathbb{Z}^n$ be strictly increasing.
The sijection $\iota_\text{MT} \colon \text{MT}(\mathbf{k}) \sijto \text{GMT}(\mathbf{k})$ in Construction \ref{iotamt} is compatible with the inversion number.
\end{proposition}
\begin{proof}
The sijection is constructed by induction on $n$. When $n=1$, the compatibility can be checked easily.
For $n>1$, it is constructed along the following relation (see Construction \ref{iotamt}):
\begin{equation} \label{porp::GIN_iota-eq1}
	\textrm{MT}(\mathbf{k})
	\sijto \bigsqcup_{\mu \in \textrm{AR}_n} \bigsqcup_{\mathbf{l} \in \mu(\mathbf{k})} \textrm{MT}(\mathbf{l})
	\sijto \bigsqcup_{\mu \in \textrm{AR}_n} \bigsqcup_{\mathbf{l} \in \mu(\mathbf{k})} \textrm{GMT}(\mathbf{l})
	= \textrm{GMT}(\mathbf{k}).
\end{equation}
We define the inversion number of $\displaystyle(T,\mathbf{l},\mu) \insupp \bigsqcup_{\mu \in \textrm{AR}_n} \bigsqcup_{\mathbf{l} \in \mu(\mathbf{k})} \textrm{MT}(\mathbf{l})$ by
$$\eta_\text{inv}\left((T,\mathbf{l},\mu)\right) = \eta_\text{inv}(\mu) + \eta_\text{inv}(T),$$
and the inversion number of $\displaystyle(T,\mathbf{l},\mu) \insupp{\bigsqcup_{\mu \in \textrm{AR}_n} \bigsqcup_{\mathbf{l} \in \mu(\mathbf{k})} \textrm{GMT}(\mathbf{l})}$ by
$$\eta_\text{inv}\left((T,\mathbf{l},\mu)\right) = \eta_\text{inv}(\mu) + \eta_\text{inv}(T).$$
Note that the trivial sijection $\displaystyle\bigsqcup_{\mu \in \textrm{AR}_n} \bigsqcup_{\mathbf{l} \in \mu(\mathbf{k})} \textrm{GMT}(\mathbf{l})=\textrm{GMT}(\mathbf{k})$ in $(\ref{porp::GIN_iota-eq1})$
is compatible with the inversion number.
In addition, the compatibility of the sijection $\displaystyle \bigsqcup_{\mu \in \textrm{AR}_n} \bigsqcup_{\mathbf{l} \in \mu(\mathbf{k})} \textrm{MT}(\mathbf{l}) \sijto \bigsqcup_{\mu \in \textrm{AR}_n}\bigsqcup_{\mathbf{l} \in \mu(\mathbf{k})} \textrm{GMT}(\mathbf{l})$ follows from the results for smaller $n$.
According to Lemma \ref{lem:comp_comp}, it is sufficient to show that the sijection $\displaystyle\textrm{MT}(\mathbf{k}) \sijto \bigsqcup_{\mu \in \textrm{AR}_n} \bigsqcup_{\mathbf{l} \in \mu(\mathbf{k})} \textrm{MT}(\mathbf{l})$ is compatible with the inversion number.

Let $\Phi$ be the sijection $\displaystyle\textrm{MT}(\mathbf{k}) \sijto \bigsqcup_{\mu \in \textrm{AR}_n} \bigsqcup_{\mathbf{l} \in \mu(\mathbf{k})} \textrm{MT}(\mathbf{l})$.
Because $\text{MT}(\mathbf{k})^- = \emptyset$, it is sufficient to show that
\begin{equation} \label{porp::GIN_iota-eq2}
\eta_\text{inv}\left((T,\mathbf{l},\mu)\right) = \eta_\text{inv}\left(\Phi((T,\mathbf{l},\mu))\right)
\end{equation}
 for any $\displaystyle(T,\mathbf{l},\mu) \insupp{\bigsqcup_{\mu \in \textrm{AR}_n} \bigsqcup_{\mathbf{l} \in \mu(\mathbf{k})} \textrm{MT}(\mathbf{l})}$.
When $\displaystyle \Phi((T,\mathbf{l},\mu)) \insupp{\bigsqcup_{\mu \in \textrm{AR}_n} \bigsqcup_{\mathbf{l} \in \mu(\mathbf{k})} \textrm{MT}(\mathbf{l})}$,
$(\ref{porp::GIN_iota-eq2})$ follows from the fact that $\Phi((T,\mathbf{l},\mu))$ is given by adding/removing $\nwarrow$ to/from $\mu$ of $(T,\mathbf{l},\mu)$.
In the other case $\mu$ must be $\mu_\mathbf{l}$, where $\mu_\mathbf{l}$ is defined in Construction \ref{iotamt} as follows:
\[
	(\mu_\mathbf{l})_i = \begin{cases} \nearrow & \text{when $i \geq 2$ and $k_i=l_{i-1}+1$}, \\ \nwarrow & \text{otherwise}. \end{cases}
\]
Let $T' = \Phi((T,\mathbf{l},\mu)) \in \bs{\text{MT}(\mathbf{k})}$.
Since $\mathbf{k}$ is strictly increasing, $k_{i-1} \leq l_{i-1} = k_i-1$ holds if and only if $(\mu_\mathbf{l})_i = \nearrow$.
Therefore, the number of $\nearrow$ in $\mu_\mathbf{l}$ coincides with $\eta_\text{inv}(T')-\eta_\text{inv}(T)$.
\end{proof}

\begin{remark}
We extended the definition of (modified) monotone triangles as pure triangular arrays to general bottom rows in Remark \ref{uniqMT},
but we cannot define the inversion number for them for the sijection $\iota \colon \text{MT} \sijto \text{GMT}$ to become compatible with it.
For example, let $\mathbf{k} = (3,1)$ and $\mathbf{l}=(0)$, then we have
\[
	\mathscr{M}_{\mathbf{k},\mathbf{l}} = (\{ \nwarrow, \nearrow \}, \{ \nxarrow \}) \times (\{ \nwarrow \},\{ \nxarrow \}).
\]
Then the unsigned distribution of $\eta_\text{inv}$ on $\bs{\mathscr{M}_{\mathbf{k},\mathbf{l}}}$  is 
\[
\begin{tabular}{ccc} \hline
$\eta_\text{inv}=0$ & $\eta_\text{inv}=1$ & $\eta_\text{inv}=2$ \\ \hline
$1$ & $3$ & $2$ \\\hline
\end{tabular}.
\]
Because the sijection $\iota$ is compatible with $\eta_\text{TA}$, it cannot be compatible with $\eta_\text{inv}$.
\end{remark}

\subsubsection{Generalized Inversion Numbers for Shifted Gelfand-Tsetlin Patterns}\label{GIN_SGT}
We define the inversion number of a shifted Gelfand-Tsetlin pattern
$(A,T) \insupp \bigsqcup_{T \in \textrm{AP}_n} \textrm{GT}(T(\mathbf{k})) = \textrm{SGT}(\mathbf{k})$
as the number of $\swarrow$ and $\sxarrow$ in $T$.
For example, the inversion number of 
\[
	\left( \begin{array}{ccccc}&&3&&\\&1&&4&\\1&&5&&2 \end{array}, \begin{array}{ccc} &\swarrow& \\ \searrow&&\sxarrow \end{array} \right) \insupp \text{SGT}(0,5,3),
\]
is two.
This definition is simple so it is meaningful to ask for compatibility with it.
However, Fischer and Konvalinka's construction is not compatible with this statistic.
For example, according to their paper \cite{FK1}, $$\left( \begin{array}{ccccc}&&2&&\\&1&&4&\\1&&3&&5 \end{array} \right) \insupp \text{MT}(1,3,5)$$ corresponds to 
$$\left(\begin{array}{ccccc}&&2&&\\&2&&4&\\2&&3&&5 \end{array},\, \begin{array}{ccc}&\swarrow&\\ \searrow&&\swarrow \end{array} \right) \insupp \text{SGT}(1,3,5),$$
but the inversion number of the former is $1$, which is different from that of the latter, which is $2$.
The rest of this section is devoted to the construction of a sijection $\text{GMT}(\mathbf{k}) \sijto \text{SGT}(\mathbf{k})$ compatible with these inversion numbers.

\subsection{
A sijection $\text{GMT}(\mathbf{k}) \sijto \text{SGT}(\mathbf{k})$}\label{ASM-SGT::KR}
Fischer and Konvalinka construct a sijection between $\text{GMT}(\mathbf{k})$ and $\text{SGT}(\mathbf{k})$ in \cite{FK1}.
They first construct a sijection $\Phi_{\mathbf{k},x} \colon \bigsqcup_{\mu \in \textrm{AR}_n} \bigsqcup_{\mathbf{l} \in \mu(\mathbf{k})} \textrm{SGT}(\mathbf{l}) \sijto \text{SGT}(\mathbf{k})$,
and then construct the desired sijection by induction on $n$ as (cf. Construction \ref{iotamt}):
\[
	\textrm{GMT}(\mathbf{k})
	= \bigsqcup_{\mu \in \textrm{AR}_n} \bigsqcup_{\mathbf{l} \in \mu(\mathbf{k})} \textrm{GMT}(\mathbf{l})
	\sijto \bigsqcup_{\mu \in \textrm{AR}_n} \bigsqcup_{\mathbf{l} \in \mu(\mathbf{k})} \textrm{SGT}(\mathbf{l})
	\sijto \textrm{SGT}(\mathbf{k}).
\]
For more details, see Problem 10 in \cite{FK1}. Their construction of $\Phi$ is given by composing the four sijections:
\begin{align*}
	&\bigsqcup_{\mu \in \textrm{AR}_n} \bigsqcup_{\mathbf{l} \in \mu(\mathbf{k})} \textrm{SGT}(\mathbf{l}) \\
	&\qquad \overset{\Phi_1}{\sijto} \bigsqcup_{\mu \in \text{AR}_n} \bigsqcup_{T \in \text{AP}_{n-1}} \bigsqcup_{\mathbf{m} \in S_1 \times S_2 \times \cdots S_{n-1}} \text{GT}(m_1+c_1(T),m_2+c_2(T),\ldots,m_{n-1}+c_{n-1}(T),x) \\
	&\qquad \overset{\Phi_2}{\sijto} \bigsqcup_{\mu \in \text{AR}_n} \bigsqcup_{T \in \text{AP}_{n}} \bigsqcup_{\mathbf{m} \in S_1 \times S_2 \times \cdots S_{n-1}} \text{GT}(m_1+c_1(T),m_2+c_2(T),\ldots,m_{n-1}+c_{n-1}(T),x) \\
	&\qquad \overset{\Phi_3}{\sijto} \bigsqcup_{\mu \in \text{AR}_n} \bigsqcup_{T \in \text{AP}_{n}} \bigsqcup_{i=1}^n \text{GT}(k_1+\delta_{\nearrow}(\mu_1)+c_1(T),\ldots,k_{i-1}+\delta_{\nearrow}(\mu_{i-1})+c_{i-1}(T), \\
	& \hspace{7cm} x,k_{i+1}-\delta_{\nwarrow}(\mu_{i+1})+c_{i}(T),\ldots,k_{n}-\delta_{\nwarrow}(\mu_{n})+c_{n-1}(T)) \\
	&\qquad \overset{\Phi_4}{\sijto} \text{SGT}(\mathbf{k}).
\end{align*}
Here, $S_i = (\{ k_{i}+\delta_{\nearrow}(\mu_{i}) \}, \{ k_{i+1}-\delta_{\nwarrow}(\mu_{i+1}) \} )$.
Note that we have slightly modified their result following our use of half-open intervals.
Our goal is to obtain a sijection compatible with the statistics $\eta_\text{top}$ and $\eta_\text{inv}$ by modifying the construction.
To describe our construction, we prepare some notations.
\begin{itemize}
\item Let $\Omega$ be a signed set $(\{0\},\{1\})$.
\item We define a function $m_i \colon \bs{\text{AR}_n} \times \bs{\text{AP}_{n-1}} \times \bs{\Omega} \to \mathbb{Z}$ by
\begin{align*}
m_i(\mu,T,0) = k_i+\delta_\nearrow(\mu_i)+c_i(T),&&m_i(\mu,T,1) = k_{i+1}-\delta_\nwarrow(\mu_{i+1})+c_i(T).
\end{align*}
\item We define an involution $r$ on arrows that acts on them as reversing the direction. For example, 
\begin{align*}
r(\nwarrow) = \searrow,&&r(\sxarrow) = \nxarrow.
\end{align*}
\end{itemize}
Our construction consists of $\Phi_1$, a modified $\Phi_3$ and a modified $\Phi_4$:
\begin{align*}
	&\bigsqcup_{\mu \in \textrm{AR}_n} \bigsqcup_{\mathbf{l} \in \mu(\mathbf{k})} \textrm{SGT}(\mathbf{l}) \\
	&\qquad \overset{\Phi_1}{\sijto} \bigsqcup_{\mu \in \text{AR}_n} \bigsqcup_{T \in \text{AP}_{n-1}} \bigsqcup_{\mathbf{m} \in S_1 \times S_2 \times \cdots S_{n-1}} \text{GT}(m_1+c_1(T),m_2+c_2(T),\ldots,m_{n-1}+c_{n-1}(T),x) \\
	&\qquad \overset{\Phi_3'}{\sijto} \bigsqcup_{i=1}^n \bigsqcup_{\mu \in \text{AR}_n} \bigsqcup_{T \in \text{AP}_{n-1}} 
	\text{GT}(m_1(\mu,T,0),\ldots,m_{i-1}(\mu,T,0), x,m_i(\mu,T,1),\ldots,m_{n-1}(\mu,T,1)) \\
	&\qquad \overset{\Phi_4'}{\sijto} \text{SGT}(\mathbf{k}).
\end{align*}

The sijection $\Phi_3'$ is essentially the same as $\Phi_3$, but the difference between  $\Phi_4'$ and $\Phi_4$ plays a key role for compatibility of $\Phi$.
Therefore, we explain the constructions of $\Phi_1$ and $\Phi_3'$ in this subsection as essentially known results and we explain $\Phi_4'$ in the next subsection.
The results in this subsection are established in \cite{FK1}.
As was the case in section \ref{secofGT}, our explanation is more transparent than the original one because of the use of half-open intervals.

\begin{construction}[See also Problem. 9 in \cite{FK1}] \label{Phi1}
We construct $\Phi_1$.
By the definition of shifted Gelfand-Tsetlin patterns and by applying Construction \ref{du_ord}, we have
\begin{align*}
	\bigsqcup_{\mu \in \textrm{AR}_n} \bigsqcup_{\mathbf{l} \in \mu(\mathbf{k})} \textrm{SGT}(\mathbf{l})
	&= \bigsqcup_{\mu \in \textrm{AR}_n} \bigsqcup_{\mathbf{l} \in \mu(\mathbf{k})} \bigsqcup_{T \in \textrm{AP}_{n-1}} \textrm{GT}(T(\mathbf{l})) \\
	&= \bigsqcup_{\mu \in \textrm{AR}_n} \bigsqcup_{T \in \textrm{AP}_{n-1}}  \bigsqcup_{\mathbf{l} \in \mu(\mathbf{k})}  \textrm{GT}(T(\mathbf{l})) \\
	&= \bigsqcup_{\mu \in \textrm{AR}_n} \bigsqcup_{T \in \textrm{AP}_{n-1}}  \bigsqcup_{\mathbf{l}' \in \bigsqcup_{\mathbf{l} \in \mu(\mathbf{k})} (\{T(\mathbf{l})\},\emptyset)} \textrm{GT}(\mathbf{l}').
\end{align*}
Since
\begin{multline*}
	\bigsqcup_{\mathbf{l} \in \mu(\mathbf{k})} (\{T(\mathbf{l})\},\emptyset)
	=
	\sint{k_1+\delta_{\nearrow}(\mu_1)+c_1(T)}{k_2-\delta_{\nwarrow}(\mu_2)+c_1(T)}\\ \times 
	\cdots\times\sint{k_{n-1}+\delta_{\nearrow}(\mu_{n-1})+c_{n-1}(T)}{k_n-\delta_{\nwarrow}(\mu_n)+c_{n-1}(T)},
\end{multline*}
using $\rho$ from Construction \ref{rho}, we obtain the construction.
\qed
\end{construction}

\begin{construction}[See also Problem. 9 in \cite{FK1}] \label{Phi3}
We construct $\Phi_3'$.
Here, we have
\begin{align*}
	&\bigsqcup_{\mu \in \text{AR}_n} \bigsqcup_{T \in \text{AP}_{n-1}} \bigsqcup_{\mathbf{m} \in S_1 \times S_2 \times \cdots S_{n-1}} \text{GT}(m_1+c_1(T),m_2+c_2(T),\ldots,m_{n-1}+c_{n-1}(T),x) \\
	&\qquad= \bigsqcup_{\mu \in \text{AR}_n}\bigsqcup_{T \in \text{AP}_{n-1}} \bigsqcup_{\omega \in \Omega^{n-1}}
	\text{GT}(m_1(\mu,T,\omega_1),m_2(\mu,T,\omega_2),\ldots,m_{n-1}(\mu,T,\omega_{n-1}),x) \\
	&\qquad= \bigsqcup_{\omega \in \Omega^{n-1}} \bigsqcup_{\mu \in \text{AR}_n}\bigsqcup_{T \in \text{AP}_{n-1}}
	\text{GT}(m_1(\mu,T,\omega_1),m_2(\mu,T,\omega_2),\ldots,m_{n-1}(\mu,T,\omega_{n-1}),x).
\end{align*}
For $1 \leq i \leq n$, we define $\omega^{(i)} \in \bs{\Omega^{n-1}}$ by $\omega^{(i)}_j = \begin{cases} 0 & j < i \\ 1 & j \geq i \end{cases}$.
Note that its sign as an element of $\Omega^{n-1}$ is $(-1)^{(n-i)}$.
Let $\Omega_1 = \left( \{ \omega_n, \omega_{n-2}, \ldots \},\, \{ \omega_{n-1}, \omega_{n-3}, \ldots \} \right)$, where $\bs{\Omega_1} = \{ \omega_1, \omega_2, \ldots, \omega_n\}$.
Furthermore, we define $\Omega_2$ by $\Omega_2^\pm = (\Omega^{n-1})^\pm \setminus \Omega_1^\pm$.
Then, using $\pi$ from Construction \ref{pi}, we have
\begin{align*}
	&\hspace{-4em}\bigsqcup_{\omega \in \Omega_1} \bigsqcup_{\mu \in \text{AR}_n}\bigsqcup_{T \in \text{AP}_{n-1}}
	\text{GT}(m_1(\mu,T,\omega_1),m_2(\mu,T,\omega_2),\ldots,m_{n-1}(\mu,T,\omega_{n-1}),x) \\
	\qquad=& \bigsqcup_{i=1}^n \bigsqcup_{\mu \in \text{AR}_n} \bigsqcup_{T \in \text{AP}_{n-1}}  (-1)^{(n-i)}
	\text{GT}(m_1(\mu,T,0),\ldots,m_{i-1}(\mu,T,0), \\
	&\hspace{240pt}m_i(\mu,T,1),\ldots,m_{n-1}(\mu,T,1),x) \\
	\qquad\overset{\pi}{\sijto}& \bigsqcup_{i=1}^n \bigsqcup_{\mu \in \text{AR}_n} \bigsqcup_{T \in \text{AP}_{n-1}} 
	\text{GT}(m_1(\mu,T,0),\ldots,m_{i-1}(\mu,T,0), x,m_i(\mu,T,1),\ldots,m_{n-1}(\mu,T,1)).
\end{align*}
It is therefore sufficient to construct
\[
	\bigsqcup_{\omega \in \Omega_2} \bigsqcup_{\mu \in \text{AR}_n}\bigsqcup_{T \in \text{AP}_{n-1}}
	\text{GT}(m_1(\mu,T,\omega_1),m_2(\mu,T,\omega_2),\ldots,m_{n-1}(\mu,T,\omega_{n-1}),x)
	\sijto (\emptyset,\emptyset).
\]
If we fix $\omega \insupp \Omega_2$ then, by the definition of $\Omega_2$, there is a minimum index $i \leq n-2$ such that $\omega_i = 1$ and $\omega_{i+1}=0$.
We define an involution on $\text{AR}_n \times \text{AP}_{n-1} \colon (\mu,T) \leftrightarrow (\mu',T')$ as follows:
\begin{itemize}
\item For $1 \leq j \leq i-1$, swap $t_{j,i}$ and $t_{j,i+1}$. Namely, $(t'_{j,i},t'_{j,i+1})=(t_{j,i+1},t_{j,i})$.
\item For $i+2 \leq j \leq n$, swap $t_{i,j}$ and $t_{i+1,j}$. Namely, $(t'_{i,j},t'_{i+1,j})=(t_{i+1,j},t_{i,j})$.
\item Let $t'_{i,i+1}=r(\mu_{i+1})$ and $\mu'_{i+1}=r(t_{i,i+1})$.
\item Other elements are left unchanged.
\end{itemize}
For example, when $n=5$ and $i=2$ the involution can be illustrated as
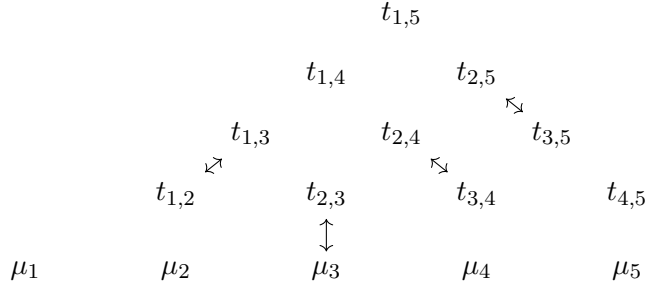
\[
\begin{tikzpicture}
\def\t{0.8}
\node (1) at (0,0) {$\mu_1$};
\node (2) at (2,0) {$\mu_2$};
\node (3) at (4,0) {$\mu_3$};
\node (4) at (6,0) {$\mu_4$};
\node (5) at (8,0) {$\mu_5$};
\foreach \i in {1,2,3,4}
	\tikzmath{
		integer \temp;
		\temp = \i+1;
	}
	\foreach \j in {\temp,...,5} 
		\tikzmath{ \x = \i+\j-1; \y = \t*(\j-\i-1)+1; }
		\node (\i\j) at (\x,\y) {$t_{\i,\j}$};
\draw[<->] (12)--(13);
\draw[<->] (24)--(34);
\draw[<->] (25)--(35);
\draw[<->] (3)--(23);
\end{tikzpicture}.
\]

Then, we have 
\begin{align*}
m_i(\mu,T,1) = k_{i+1}-\delta_\nwarrow(\mu_{i+1})+c_i(T) = k_{i+1}+\delta_\nearrow(\mu'_{i+1})+c_{i+1}(T') = m_{i+1}(\mu',T',0), \\
m_{i+1}(\mu,T,0) = k_{i+1}+\delta_\nearrow(\mu_{i+1})+c_{i+1}(T) = k_{i+1}-\delta_\nwarrow(\mu'_{i+1})+c_{i}(T') = m_{i}(\mu',T',1),
\end{align*}
and since the involution acts on $(\mu,T)$ as permuting elements it preserves the sign.
Therefore, from Construction \ref{pi}, we have
\begin{multline*}
	\text{GT}(m_1(\mu,T,\omega_1),m_2(\mu,T,\omega_2),\ldots,m_{n-1}(\mu,T,\omega_{n-1}),x) \\
	\sijto -\text{GT}(m_1(\mu',T',\omega_1),m_2(\mu',T',\omega_2),\ldots,m_{n-1}(\mu',T',\omega_{n-1}),x).
\end{multline*}
This completes the construction. \qed
\end{construction}

\subsection{
A more natural sijection $\text{GMT}(\mathbf{k}) \sijto \text{SGT}(\mathbf{k})$}\label{ASM-SGT::MR}
In this section we construct $\Phi_4'$ to complete the construction of the sijection between $\text{GMT}(\mathbf{k})$ and $\text{SGT}(\mathbf{k})$.
After that, we prove its compatibility with the top and the inversion number statistics.

\begin{construction} \label{Phi4}
We shall construct $\Phi_4'$.
First, we construct a bijection $\bs{\text{AR}_n} \times \bs{\text{AP}_{n-1}} \to \bs{\text{AR}_1} \times \bs{\text{AP}_n} ; (\mu,T) \mapsto (\mu',T')$ as follows (cf. \cite[Problem. 8]{FK1}):
\begin{gather*}
t'_{p,q} = \begin{cases} t_{p,q} & p<q<i \\ t_{p,q-1} & p<i<q \\ t_{p-1,q-1} & i<p<q \\ r(\mu_p) & p<i=q \\ r(\mu_q) & p=i<q \end{cases} \\
\mu'_1 = \mu_i.
\end{gather*}
For example, when $n=5$ and $i=3$ we have
\begin{align*}
	\mu'=(\mu_3),&& T' = \quad
\begin{tikzpicture}[baseline=30]
\node (12) at (0,0) {$t_{1,2}$};
\node (13) at (1,0.8) {$r(\mu_1)$};
\node (14) at (2,1.6) {$t_{1,3}$};
\node (15) at (3,2.4) {$t_{1,4}$};
\node (23) at (2,0) {$r(\mu_2)$};
\node (24) at (3,0.8) {$t_{2,3}$};
\node (25) at (4,1.6) {$t_{2,4}$};
\node (34) at (4,0) {$r(\mu_4)$};
\node (35) at (5,0.8) {$r(\mu_5)$};
\node (45) at (6,0) {$t_{3,4}$};
\end{tikzpicture}.
\end{align*}

Since this bijection only rearranges and reverses arrows, it preserves the sign. Therefore, it can be recognized as a sijection $\text{AR}_n \times \text{AP}_{n-1} \sijto \text{AR}_1 \times \text{AP}_n$.
In addition we have a sijection $\varphi_{\text{AR}_1} \colon \text{AR}_1 \sijto (\{\cdot\},\emptyset)$ defined by
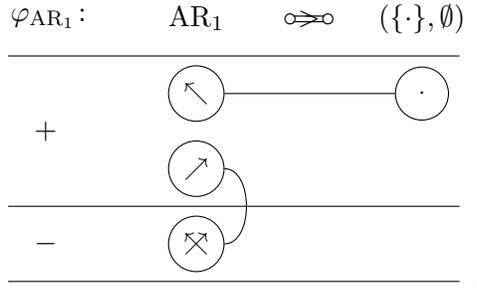
\[
\begin{tikzpicture}[main/.style = {draw, circle}]
\node (S) at (0,3) {$\text{AR}_1$};
\node (T) at (3,3) {$(\{\cdot\},\emptyset)$};
\node (phi) at (-1.9,3) {$\varphi_{\text{AR}_1} \colon$};
\node (sij) at (1.5,3) {$\sijto$};
\node (+) at (-2,1.5) {$+$};
\node (-) at (-2,0) {$-$};
\draw (-2.5,-0.5) to (3.5,-0.5);
\draw (-2.5,0.5) to (3.5,0.5);
\draw (-2.5,2.5) to (3.5,2.5);
\node[scale=0.84,main,minimum size=.84cm] (1) at (3,2) {$\cdot$};
\node[scale=0.84,main] (2) at (0,1) {$\nearrow$};
\node[scale=0.84,main] (3) at (0,0) {$\nxarrow$};
\node[scale=0.84,main] (4) at (0,2) {$\nwarrow$};
\draw[-] (1) -- (4);
\draw[-] (2) to [out=0,in=0] (3);
\end{tikzpicture},
\]
 so we get a sijection $\Psi_{n,i} \colon \text{AR}_n \times \text{AP}_{n-1} \sijto \text{AP}_n$ such that
\begin{align*}
	m_j(\mu,T,0) &= k_j+c_j(\Psi_{n,i}(\mu,T)) &&\hspace{-80pt} \text{ if } j<i, \\
	m_j(\mu,T,1) &= k_{j+1}+c_{j+1}(\Psi_{n,i}(\mu,T)) &&\hspace{-80pt} \text{ if } j \geq i.
\end{align*}
By applying this result and Definition \ref{def_sij_DIS}, we have
\begin{multline*}
	\bigsqcup_{i=1}^n \bigsqcup_{\mu \in \text{AR}_n} \bigsqcup_{T \in \text{AP}_{n-1}} 
	\text{GT}(m_1(\mu,T,0),\ldots,m_{i-1}(\mu,T,0), x,m_i(\mu,T,1),\ldots,m_{n-1}(\mu,T,1)) \\
	\sijto \bigsqcup_{i=1}^n \bigsqcup_{T \in \text{AP}_{n}} \text{GT}(k_1+c_1(T),\ldots,k_{i-1}+c_{i-1}(T),x,k_{i+1}+c_{i+1}(T),\ldots,k_{n}+c_{n}(T)).
\end{multline*}
Then, using Construction \ref{du_ord} and $\tau$ in Construction \ref{tau}, we have
\begin{multline*}
	\bigsqcup_{i=1}^n \bigsqcup_{T \in \text{AP}_{n}} \text{GT}(k_1+c_1(T),\ldots,k_{i-1}+c_{i-1}(T),x,k_{i+1}+c_{i+1}(T),\ldots,k_{n}+c_{n}(T)) \\
	= \bigsqcup_{T \in \text{AP}_{n}} \bigsqcup_{i=1}^n \text{GT}(k_1+c_1(T),\ldots,k_{i-1}+c_{i-1}(T),x,k_{i+1}+c_{i+1}(T),\ldots,k_{n}+c_{n}(T)) \\
	\sijto \bigsqcup_{T \in \text{AP}_{n}} \text{GT}(k_1+c_1(T),k_2+c_2(T),\ldots,k_{n}+c_{n}(T)) = \text{SGT}(\mathbf{k}),
\end{multline*}
which completes the construction. \qed
\end{construction}

Finally, we obtain the desired sijection.
\begin{construction}\label{Gamma}
We construct a sijection $\Gamma \colon \text{GMT}(\mathbf{k}) \sijto \text{SGT}(\mathbf{k})$ for any $\mathbf{k} \in \mathbb{Z}^n$, $n \in \mathbb{Z}$.
When $n=1$, $\text{GMT}(\mathbf{k})=\text{AR}_1$ and $\text{SGT}(\mathbf{k}) = (\{k_1\},\emptyset)$.
Therefore, the construction is given as:
\[
\begin{tikzpicture}[main/.style = {draw, circle}]
\node (S) at (0,3) {$\text{GMT}(\mathbf{k})$};
\node (T) at (3,3) {$\text{SGT}(\mathbf{k})$};
\node (phi) at (-1.9,3) {$\varphi \colon$};
\node (sij) at (1.5,3) {$\sijto$};
\node (+) at (-2,1.5) {$+$};
\node (-) at (-2,0) {$-$};
\draw (-2.5,-0.5) to (3.5,-0.5);
\draw (-2.5,0.5) to (3.5,0.5);
\draw (-2.5,2.5) to (3.5,2.5);
\node[scale=0.84,main] (1) at (3,2) {$k_1$};
\node[scale=0.84,main] (2) at (0,1) {$\nearrow$};
\node[scale=0.84,main] (3) at (0,0) {$\nxarrow$};
\node[scale=0.84,main] (4) at (0,2) {$\nwarrow$};
\draw[-] (1) -- (4);
\draw[-] (2) to [out=0,in=0] (3);
\end{tikzpicture}.
\]
When $n>1$, the construction is given by induction on $n$:
\[
	\textrm{GMT}(\mathbf{k})
	= \bigsqcup_{\mu \in \textrm{AR}_n} \bigsqcup_{\mathbf{l} \in \mu(\mathbf{k})} \textrm{GMT}(\mathbf{l})
	\sijto \bigsqcup_{\mu \in \textrm{AR}_n} \bigsqcup_{\mathbf{l} \in \mu(\mathbf{k})} \textrm{SGT}(\mathbf{l})
	\overset{\Phi'}{\sijto} \textrm{SGT}(\mathbf{k}),
\]
where $\Phi' = \Phi_4' \circ \Phi_3' \circ \Phi_1$. \qed
\end{construction}

Last, we check the compatibility of the sijection $\Gamma$.
\begin{proposition} \label{prop::MR-SGT}
The sijection $\Gamma$ in Construction \ref{gamma} is compatible with the two statistics $\eta_\text{top}$ and $\eta_\text{inv}$.
In particular, it is compatible with the statistic $(\eta_\text{top},\eta_\text{inv})$.
\end{proposition}
\begin{proof}
When $n=1$, this is trivial. In the following, we assume that $n>1$.
First, we define the statistics for the intermediates. For $(T,\mathbf{l},\mu) \insupp \bigsqcup_{\mu \in \textrm{AR}_n} \bigsqcup_{\mathbf{l} \in \mu(\mathbf{k})} \textrm{GMT}(\mathbf{l})$,
we define
\begin{align*}
	\eta_\text{top}((T,\mathbf{l},\mu)) &= \eta_\text{top}(T), \\
	\eta_\text{inv}((T,\mathbf{l},\mu)) &= \eta_\text{inv}(T) + \eta_\text{inv}(\mu),
\end{align*}
and for $(T,\mathbf{l},\mu) \insupp \bigsqcup_{\mu \in \textrm{AR}_n} \bigsqcup_{\mathbf{l} \in \mu(\mathbf{k})} \textrm{SGT}(\mathbf{l})$, we define
\begin{align*}
	\eta_\text{top}((T,\mathbf{l},\mu)) &= \eta_\text{top}(T), \\
	\eta_\text{inv}((T,\mathbf{l},\mu)) &= \eta_\text{inv}(T) + \eta_\text{inv}(\mu).
\end{align*}
The trivial sijection $\textrm{GMT}(\mathbf{k})
	= \bigsqcup_{\mu \in \textrm{AR}_n} \bigsqcup_{\mathbf{l} \in \mu(\mathbf{k})} \textrm{GMT}(\mathbf{l})$
is compatible with these statistics and the compatibility of the second sijection follows from the cases for smaller $n$.
Therefore, it is sufficient to check the compatibility of $\Phi'$.
\begin{itemize}
\item $\Phi_1$: The compatibility with $\eta_\text{top}$ follows from the compatibility of $\rho$ (Proposition \ref{prop::rho}).
The sijection does not act on arrows, and hence it is compatible with inversion numbers.
\item $\Phi_3'$: The compatibility with $\eta_\text{top}$ follows from the compatibility of $\pi$ (Proposition \ref{prop::pi}).
The sijection acts on arrows as reversing and permuting, and hence it is compatible with inversion numbers.
\item $\Phi_4'$: The compatibility with $\eta_\text{top}$ follows from the compatibility of $\text{id}_\text{GT}$ and $\tau$ (Proposition \ref{prop::tau}).
The sijection acts on arrows as a permutation except for the actions of $\varphi_{\text{AR}_1}$. Since the actions of $\varphi_{\text{AR}_1}$ do not affect the inversion number, the sijection is compatible with it.
\end{itemize}
According to Lemma \ref{lem:comp_comp}, the compatibility of $\Phi' = \Phi_4' \circ \Phi_3' \circ \Phi_1$ follows from the above.
\end{proof}

\begin{remark}
One of the advantages of proving refined enumerations with compatibility appears in the way one proves doubly-refined enumerations.
If a sijection is compatible with two statistics $\eta_A$ and $\eta_B$ then it is also compatible with the pair statistic $(\eta_A,\eta_B)$.
In general, however, the result of doubly-refined enumeration with respect to $(\eta_A,\eta_B)$ does not follow immidiately from the results of refined enumeration with respect to each of $\eta_A$ and $\eta_B$ separately.
\end{remark}

\subsection{About the choice of the parameter $x$ in the construction}
Fischer and Konvalinka's construction as well as ours depends on a parameter $x$ which, in fact, can be made to depend on the other parameters in each sijection that appears in the construction, by induction, notably on $\mathbf{k}$.
Here we show that there exists a choice of $x$ that has particularly nice properties at the level of the sijections.
We discuss how to choose the parameter in this subsection.
In particular we prove the following proposition.
\begin{proposition}\label{result::limit}
Let $\mathbf{k} \in \mathbb{Z}^n$. Then, there exist $X_+,X_- \in \mathbb{Z}$ such that
\begin{align*}
	\forall x \geq X_+,\,\Gamma_{\mathbf{k},x} = \Gamma_{\mathbf{k},X_+}, \\
	\forall x \leq X_-,\,\Gamma_{\mathbf{k},x} = \Gamma_{\mathbf{k},X_-},
\end{align*}
where $\Gamma_{\mathbf{k},x} \colon \text{GMT}(\mathbf{k}) \sijto \text{SGT}(\mathbf{k})$ is the sijection constructed in Construction \ref{gamma}.
\end{proposition}
According to the proposition, we define $\Gamma_{\mathbf{k},\pm\infty} = \Gamma_{\mathbf{k},X_\pm}$. Note that $X_\pm$ depend on $\mathbf{k}$.
Let $f_t \colon \mathbb{Z} \to \mathbb{Z}; k \mapsto k+1$ be a translation.
The definition of $f_t$ is extended to $\mathbf{Z}^n$, $\text{GMT}$ and $\text{SGT}$ in a natural way.
The choice of parameters in $\Gamma_{\mathbf{k},\pm\infty}$ is good in the sense that it satisfies
\[
f_t \circ \Gamma_{\mathbf{k},\pm\infty} = \Gamma_{f_t(\mathbf{k}),\pm\infty} \circ f_t.
\]

The rest of this subsection is devoted to prove the proposition and the above relation.
First we define inclusion relations of signed sets and sijections.
Let $S$ and $T$ be signed sets. We say $S$ is included in $T$ and denote it by $S \subset T$ if $S^+ \subset T^+$ and $S^- \subset T^-$ hold.
Let $\varphi_i \colon S_i \sijto T_i$ ($i=1,2$) be sijections. We say $\varphi_1$ is included in $\varphi_2$ and denote it by $\varphi_1 \subset \varphi_2$ if all of the following conditions are met:
\begin{align*}
	S_1 \subset S_2,&& T_1 \subset T_2,&& \forall s \in \bs{S_1 \sqcup T_1},\, \varphi_1(s) = \varphi_2(s).
\end{align*}
We also define the difference of signed sets and sijections as follows:
\begin{itemize}
\item When $S \subset T$, the signed set $T \setminus S$ is $(T^+ \setminus S^+, T^- \setminus S^-)$.
\item When $\varphi_i \colon S_i \sijto T_i$ ($i=1,2$) meet $\varphi_1 \subset \varphi_2$, the sijection $\varphi_2 \setminus \varphi_1 \colon S_2 \setminus S_1 \sijto T_2 \setminus T_1$ is a restriction of $\varphi_2$.
\end{itemize}
Then, we have $S \sqcup (T \setminus S) = T$ and $\varphi_2 = \varphi_1 \sqcup (\varphi_2 \setminus \varphi_1)$.
The converse is also true. Namely, we have $S \subset S \sqcup T$ and $(S \sqcup T) \setminus S = T$ for any signed sets $S$ and $T$ and 
$\varphi \subset \varphi \sqcup \psi$ and $(\varphi \sqcup \psi) \setminus \varphi = \psi$ for any sijections $\varphi$ and $\psi$.
In particular, when $\varphi_i \colon S_i \sijto T_i$, $\psi_i \colon T_i \sijto U_i$, ($i=1,2$) satisfy $\varphi_1 \subset \varphi_2,\,\psi_1 \subset \psi_2$, we have
\[
	\psi_2 \circ \varphi_2 = (\psi_1 \sqcup (\psi_2 \setminus \psi_1)) \circ (\varphi_1 \sqcup (\varphi_2 \setminus \varphi_1)) = (\psi_1 \circ \varphi_1) \sqcup ((\psi_2 \setminus \psi_1) \circ (\varphi_2 \setminus \varphi_1)).
\]
Therefore, we have $\psi_1 \circ \varphi_1 \subset \psi_2 \circ \varphi_2$.
Particularly, if $S_1 = S_2$ and $T_1 = T_2$, then $\psi_1 \circ \varphi_1 = \psi_2 \circ \varphi_2$.

Let $\{\varphi_x\}_{x \in \mathbb{Z}}$ be a sequence of sijections.
We say $\{\varphi_x\}_{x \in \mathbb{Z}}$ makes a \textit{chain} in $x \geq X$, when
\[
	\varphi_X \subset \varphi_{X+1} \subset \varphi_{X+2} \subset \cdots.
\]
Similarly, we say $\{\varphi_x\}_{x \in \mathbb{Z}}$ makes a chain in $x \leq X$, when
\[
	\varphi_X \subset \varphi_{X-1} \subset \varphi_{X-2} \subset \cdots.
\]
\begin{proposition}\label{prop::alpha}
Let $a$, $b$, $x$ be integers and $\varphi_x \colon \sint{a}{b} \sijto \sint{a}{x} \sqcup \sint{x}{b}$ the sijection constructed in Example \ref{ex_sij}.
Then, $\{ \varphi_x\}_{x \in \mathbb{Z}}$ makes a chain in $x \geq \max(a,b)$ and in $x \leq \min(a,b)$.
\end{proposition}
\begin{proof}
For $x \geq \max(a,b)$, we have $\varphi_{x+1} = \varphi_x \sqcup \psi_x$, where $\psi_x$ is the unique sijection between $(\emptyset,\emptyset)$ and $(\{x\},\{x\})$.
Therefore, $\{ \varphi_x\}_{x \in \mathbb{Z}}$ makes a chain in $x \geq \max(a,b)$.
We can prove the case $x \leq \min(a,b)$ similarly.
\end{proof}
We prepare some lemmas to prove Proposition \ref{result::limit}.
\begin{lemma}\label{chain::lem::ids}
Let $S$ and $T$ be signed sets such that $S \subset T$. Then, we have $\text{id}_S \subset \text{id}_T$.
In addition, let $\varphi_S \colon S \sqcup -S \sijto (\emptyset,\emptyset)$ be the sijection induced by $\text{id}_S$ as in Construction \ref{sij_oppo}.
We define $\varphi_T$ similarly. Then, we have $\varphi_S \subset \varphi_T$.
\end{lemma}
\begin{proof}
The proof follows immediately from the definitions.
\end{proof}
\begin{lemma}\label{lem::times}
Let $\varphi_i \colon S_i \sijto T_i$ be sijections for $i=1,2$ and $\psi \colon U \sijto V$ a sijection.
Then, we have
\[
	\varphi_1 \subset \varphi_2 \Rightarrow \psi \sqcup \varphi_1 \subset \psi \sqcup \varphi_2.
\]
and
\[
	\varphi_1 \subset \varphi_2 \Rightarrow \psi \times \varphi_1 \subset \psi \times \varphi_2.
\]
\end{lemma}
\begin{proof}
The proof follows immediately from the definitions.
\end{proof}
When $\varphi_i \colon S_i \sijto T_i$ and $\psi_i \colon U_i \sijto V_i$ are sijections for $i=1,2$ such that $\varphi_1 \subset \varphi_2$ and $\psi_1 \subset \psi_2$, then we have from this lemma that
\[
	\psi_1 \times \varphi_1 \subset \psi_1 \times \varphi_2 \subset \psi_2 \times \varphi_2.
\]
Furthermore, for a disjoint union with signed index, the following lemma holds.
\begin{lemma}\label{lem::DSU}
Let $\psi_i \colon T_i \sijto \tilde{T}_i$ be sijections for $i \in \mathbb{Z}$ such that $\psi_1 \subset \psi_2 \subset \cdots$.
Also, for $S = \bigcup_{i=1}^{\infty} \left(T_i^+\right) \sqcup \bigcup_{i=1}^{\infty} \left(\tilde{T}_i^-\right)$,
let $\left\{\varphi_s\colon U_s \sijto V_s\right\}_{s \in S}$ be a family of sijections.
Then we have
\[
	\bigsqcup_{t \in T_1 \sqcup \tilde{T}_1} \varphi_t \subset \bigsqcup_{t \in T_2 \sqcup \tilde{T}_2} \varphi_t \subset \cdots.
\]
\end{lemma}
\begin{proof}
The proof follows immediately from the definitions.
\end{proof}
\begin{proof}[Proof of Proposition \ref{result::limit}]
The following claims can be checked by tracing the corresponding construction.
\begin{enumerate}
\item The sijection $\beta$ in Subconstruction \ref{beta} makes a chain in $x \geq \max(a_n,b_n)$ and in $x \leq \min(a_n,b_n)$.
This follows from Proposition \ref{prop::alpha}, Lemma \ref{lem::times} and the induction.
\item The sijection $\rho$ in Construction \ref{rho} makes a chain in $x \geq \max(a_n,b_n)$ and in $x \leq \min(a_n,b_n)$.
This follows from the property of $\beta$ and Lemma \ref{lem::DSU}.
\item The sijection $\pi$ in Construction \ref{pi} makes a chain in $k_j \geq \max(k_1,k_2,\ldots,k_{j-1},k_{j+1},\ldots,k_n)$ and in $k_j \leq \min(k_1,k_2,\ldots,k_{j-1},k_{j+1},\ldots,k_n)$ for $j \in \{1,2,\ldots,n\}$.
Also, the sijection $\sigma$ in Construction \ref{pi} makes a chain in $a_j \geq b_j$, in $a_j \leq b_j$, in $b_j \geq a_j$ and in $b_j \leq a_j$ for $j \in \{1,2,\ldots,i-1,i+2,\ldots,n\}$.
Furthermore, let $\mathbf{a}' = (a_1,a_2,\ldots,a_{i-1},x,x,a_{i+2},\ldots,a_n)$ and $\mathbf{b}' = (b_1,b_2,\ldots,b_{i-1},x,x,b_{i+2},\ldots,b_n)$.
Then, $\sigma_{\mathbf{a}',\mathbf{b},x}$ makes a chain in $x \geq \max(b_i,b_{i+1})$ and in $x \leq \min(b_i,b_{i+1})$
and $\sigma_{\mathbf{a},\mathbf{b}',x}$ makes a chain in $x \geq \max(a_i,a_{i+1})$ and in $x \leq \min(a_i,a_{i+1})$.
These properties of $\sigma_n$ follow from Lemma \ref{chain::lem::ids} and the property of $\pi_n$ and the property of $\pi_n$ follows from those of $\sigma_{n-1}$.
Therefore the proof of these properties is given by induction.
\item The sijection $\gamma$ in Subconstruction \ref{gamma} makes a chain in $x \geq \max(\mathbf{k})$ and in $x \leq \min(\mathbf{k})$.
This follows from the lemmas.
\item The sijection $\tau$ in Subconstruction \ref{tau} makes a chain in $x \geq \max(\mathbf{k})$ and in $x \leq \min(\mathbf{k})$.
This follows from the properties of $\gamma$ and $\sigma$.
\item The sijection $\Phi_1$ in Construction \ref{Phi1} makes a chain in $x \geq \max(\mathbf{k})+n$ and in $x \leq \min(\mathbf{k})-n$.
This follows from the property $\rho$.
\item The sijection $\Phi_3'$ in Construction \ref{Phi3} makes a chain in $x \geq \max(\mathbf{k})+n$ and in $x \leq \min(\mathbf{k})-n$.
This follows from the property of $\pi$.
\item The sijection $\Phi_4'$ in Construction \ref{Phi4} makes a chain in $x \geq \max(\mathbf{k})+n$ and in $x \leq \min(\mathbf{k})-n$.
This follows from the Lemma \ref{lem::DSU} and the property of $\tau$.
\item The sijection $\Phi' = \Phi_4' \circ \Phi_3' \circ \Phi_1$ makes a chain in $x \geq \max(\mathbf{k})+n$ and in $x \leq \min(\mathbf{k})-n$.
This follows from the properties of $\Phi_1$, $\Phi_3'$ and $\Phi_4'$.
\item The sijection $\Gamma$ in Construction \ref{Gamma} makes a chain in $x \geq \max(\mathbf{k})+n$ and in $x \leq \min(\mathbf{k})-n$.
This follows from the property of $\Phi'$.
\end{enumerate}
Since the domain and codomain of $\Gamma$ ($\text{GMT}(\mathbf{k})$ and $\text{SGT}(\mathbf{k})$) are independent of the parameter $x$,
we have
\begin{align*}
\Gamma_{X_+} = \Gamma_{X_++1} = \Gamma_{X_++2} = \cdots,&& \text{where $X_+ = \max(\mathbf{k})+n$,} \\
\Gamma_{X_-} = \Gamma_{X_--1} = \Gamma_{X_--2} = \cdots,&& \text{where $X_- = \min(\mathbf{k})-n$.}
\end{align*}
\end{proof}

Last, we check the property of $\Phi_{\pm\infty}$ with respect to translations:
\[
f_t \circ \Gamma_{\mathbf{k},\pm\infty} = \Gamma_{f_t(\mathbf{k}),\pm\infty} \circ f_t.
\]
By the constructions, we have
\[
f_t \circ \Gamma_{\mathbf{k},x} = \Gamma_{f_t(\mathbf{k}),f_t(x)} \circ f_t.
\]
Considering a parameter $x \geq \max(X_+(\mathbf{k}),X_+(f_t(\mathbf{k})))$, we have
\[
f_t \circ \Gamma_{\mathbf{k},+\infty} = \Gamma_{f_t(\mathbf{k}),+\infty} \circ f_t.
\]
We can prove a similar result for $-\infty$.

\subsection{A natural sijection $\text{GMT}(\mathbf{k}) \sijto \text{AR}_n \times \text{SGT}(\mathbf{k})$} \label{ssec::gmt-ar_sgt}

In this subsection, we start with the weighted enumeration formula of a variant of monotone triangles, which is proved in \cite{AF}.
From this formula, we construct a sijection that is compatible with $n+3$ statistics between $\text{GMT}(\mathbf{k})$ and $\text{AR}_n \times \text{SGT}(\mathbf{k})$.
This serves as a good example of transforming a non-constructive proof for a weighted enumeration formula into a constructive one using the concept of compatibility.
First, to align with the terms of \cite{AF}, we will slightly modify the definitions of the combinatorial objects introduced so far.
\begin{itemize}
\item Let $\text{AR}_n=(\{\nwarrow,\nearrow,\nxarrow\},\emptyset)^n$.
The action of an element of $\text{AR}_n$ on $\mathbf{k} \in \mathbb{Z}^n$ is defined as before.
The definition of generalized monotone triangles is also modified accordingly.
\item Let $\mathrm{AP}_n=(\{\searrow,\swarrow,\sxarrow\},\emptyset)^{n-1} \times (\{\searrow,\swarrow,\sxarrow\},\emptyset)^{n-2} \times \cdots \times (\{\searrow,\swarrow,\sxarrow\},\emptyset)^{1}
$.
The action of an element of $\text{AP}_n$ on $\mathbf{k} \in \mathbb{Z}^n$ is defined as before.
The definition of Shifted Gelfand-Tsetlin patterns is also modified accordingly.
\end{itemize}
This modification just changes the signs of $\nxarrow$ and $\sxarrow$.
Therefore, except for the parts involving cancellations with $\nearrow\, \leftrightarrow \nxarrow$ and $\swarrow \leftrightarrow \sxarrow$, the constructions of the sijection described thus far remain unaffected.
In \cite{AF}, weighted enumerations of several variants of monotone triangles are provided.
When rewritten in accordance with the definitions and terminology of this paper, we obtain the following result for the weighted enumeration of generalized monotone triangles.
Here, the weighted enumeration of a signed set $S$ is defined as:
\[
	\sum_{s \in S^+} \text{weight}(s) + \sum_{s \in S^+} \text{weight}(s).
\]

\begin{theorem}[Theorem 2.2 and 3.2 in \cite{AF}] \label{formulae}
For $T = \left(\mathbf{k}^{(n)},\mu^{(n)},\mathbf{k}^{(n-1)},\mu^{(n-1)},\ldots,\mathbf{k}^{(1)},\mu^{(1)}\right) \insupp \textrm{GMT}(\mathbf{k})$, we define its weight as follows:
\begin{itemize}
\item Let $\eta_u(T)$, $\eta_v(T)$ and $\eta_w(T)$ be the number of $\nearrow$, $\nwarrow$ and $\nxarrow$ respectively in $\mu^{(1)},\mu^{(2)},\ldots,\mu^{(n)}$.
\item Let $\eta_{X_i}(T) = \sum \mathbf{k}^{(i)} - \sum \mathbf{k}^{(i-1)} + (\text{$\#$ of $\nearrow$ in $\mu^{(i)}$}) - (\text{$\#$ of $\nwarrow$ in $\mu^{(i)}$})$,
where $\sum \mathbf{k}^{(i)}$ means the summation of entries of $\mathbf{k}^{(i)}$ and we define $\sum \mathbf{k}^{(0)} = 0$.
\item Then, the weight of $T$ is $u^{\eta_u(T)} v^{\eta_v(T)} w^{\eta_w(T)} X_i^{\eta_{X_i}(T)}$.
\end{itemize}
As a result, the weighted enumeration of generalized monotone triangles with bottom row $\mathbf{k}$ is 
\[
	\prod_{i=1}^n \left( uX_i + vX_i^{-1} + w \right) \prod_{1 \leq p < q \leq n} \left( u E_{k_p} + v E_{k_q}^{-1} + w E_{k_p} E_{k_q}^{-1} \right) \tilde{s}_{(k_n,k_{n-1},\ldots,k_1)}(X_1,X_2,\ldots,X_n),
\]
where $E_x$ is a shift operator such that $E_x p(x)=p(x+1)$ and $\tilde{s}_{(k_n,k_{n-1},\ldots,k_1)}$ is a transformed Schur polynomial.
\end{theorem}
The relation between the transformed Schur polynomial $\tilde{s}$ and the extended Schur polynomial $s$ defined in \cite{AF} is
\[
	\tilde{s}_{(k_n,k_{n-1},\ldots,k_1)} = s_{(k_n-(n-1),k_{n-1}-(n-2),\ldots,k_1-0)}.
\]
This difference is the result of our change in convention from closed intervals to half-open ones.
For the definition of extended Schur polynomials, which we omit here as it is not necessary for this paper, 
please see \cite{AF}.
What is important in this paper is that $\tilde{s}{(k_n,k_{n-1},\ldots,k_1)}(X_1,X_2,\ldots,X_n)$ is the weighted enumeration of $\text{GT}(\mathbf{k})$ with appropriate weight.
In fact, for $A \insupp \text{GT}(\mathbf{k})$ we define ${\sum}_i(A) = \sum \eta_\text{row}(A)_i$, which means that ${\sum}_i(A)$ is the sum of elements in the $i$-th row when considering $A$ as a triangular array. Then, we can assign the weight as $\text{weight}(A) = X_i^{\eta_{X_i}(A)}$, where $ \eta_{X_i}(A) = {\sum}_i(A) - {\sum}_{i-1}(A)$.
Based on the observation and the definition of the shift operator, the part of 
$\prod_{1 \leq p < q \leq n} \left( u E_{k_p} + v E_{k_q}^{-1} + w E_{k_p} E_{k_q}^{-1} \right) \tilde{s}_{(k_n,k_{n-1},\ldots,k_1)}(X_1,X_2,\ldots,X_n)$
in the weighted enumeration of $\text{GMT}(\mathbf{k})$ is identified with a weighted enumeration of $\text{SGT}(\mathbf{k})$ with appropriate weight.
Furthermore, the remaining part is recognized as a weighted enumeration of $\text{AR}_n$ with appropriate weight.
To summarize the above discussion, the following corollary can be obtained.

\begin{corollary} \label{cor::gmt-ar-sgt}
There exists a sijection $\text{GMT}(\mathbf{k}) \sijto \text{AR}_n \times \text{SGT}(\mathbf{k})$ which is compatible with the following $n+3$ statistics
$\eta_u$, $\eta_v$, $\eta_w$, $\eta_{X_1}$, $\eta_{X_2}$, $\ldots$, $\eta_{X_n}$:
\begin{itemize}
\item For GMT side, the definitions of statistics are the same in Theorem \ref{formulae}.
\item For $(\mu,(A,T)) \insupp \text{AR}_n \times \text{SGT}(\mathbf{k})$, where $A \insupp \text{GT}(T(\mathbf{k}))$, we define its statistics as follows:
\begin{align*}
\eta_u((\mu,(A,T))) &= (\text{\# of $\nearrow$ in $\mu$}) + (\text{\# of $\swarrow$ in $T$}), \\
\eta_v((\mu,(A,T))) &= (\text{\# of $\nwarrow$ in $\mu$}) + (\text{\# of $\searrow$ in $T$}), \\
\eta_w((\mu,(A,T))) &= (\text{\# of $\nxarrow$ in $\mu$}) + (\text{\# of $\sxarrow$ in $T$}), \\
\eta_{X_i}((\mu,(A,T)) &= \eta_{X_i}(A) + \delta_\nearrow(\mu_i) - \delta_\nwarrow(\mu_i).
\end{align*}
\end{itemize}
\end{corollary}

The rest of this subsection is devoted to actually constructing this sijection.
\begin{construction} \label{construction::asm-sgt}
We shall construct a sijection which satisfies the conditions described in Corollary \ref{cor::gmt-ar-sgt}.
The construction is by induction on $n$.
If $n=1$, it is trivial since we have $\text{GMT}(\mathbf{k})=\text{AR}_1$ and $\text{SGT}(\mathbf{k}) = \text{GT}(\mathbf{k}) = (\{k\},\emptyset)$.
If $n>1$, it is given as follows:
\begin{align*}
\text{GMT}(\mathbf{k}) =& \bigsqcup_{\mu \in \textrm{AR}_n} \bigsqcup_{\mathbf{l} \in \mu(\mathbf{k})} \textrm{GMT}(\mathbf{l}) \\
\sijto& \bigsqcup_{\mu \in \textrm{AR}_n} \bigsqcup_{\mathbf{l} \in \mu(\mathbf{k})} \left( \text{AR}_{n-1} \times \text{SGT}(\mathbf{l}) \right) \\
=& \text{AR}_{n-1} \times \bigsqcup_{\mu \in \textrm{AR}_n} \bigsqcup_{\mathbf{l} \in \mu(\mathbf{k})} \text{SGT}(\mathbf{l}) \\
\overset{\text{id}_{\text{AR}_{n-1}} \times \Phi_1}{\sijto}& \text{AR}_{n-1} \times \bigsqcup_{\mu \in \text{AR}_n} \bigsqcup_{T \in \text{AP}_{n-1}} \bigsqcup_{\mathbf{m} \in S_1 \times S_2 \times \cdots S_{n-1}} \\
&\qquad\qquad\qquad\text{GT}(m_1+c_1(T),m_2+c_2(T),\ldots,m_{n-1}+c_{n-1}(T),x) \\
\overset{\text{id}_{\text{AR}_{n-1}} \times \Phi_3'}{\sijto}& \text{AR}_{n-1} \times \bigsqcup_{i=1}^n \bigsqcup_{\mu \in \text{AR}_n} \bigsqcup_{T \in \text{AP}_{n-1}} \\
&\qquad\qquad\text{GT}(m_1(\mu,T,0),\ldots,m_{i-1}(\mu,T,0), x,m_i(\mu,T,1),\ldots,m_{n-1}(\mu,T,1)) \\
\overset{\Phi_4''}{\sijto}& \text{AR}_{n-1} \times \text{AR}_1 \times \text{SGT}(\mathbf{k}) \\
=& \text{AR}_n \times \text{SGT}(\mathbf{k}),
\end{align*}
where the assumption of induction is used in the second line.
For the definition of $\Phi_1$, $\Phi_3'$, $S_i$ and $m_i$, refer to the Subsection \ref{ASM-SGT::KR}.
In fact, as the constructions of $\Phi_1$ and $\Phi_3'$ do not involve the cancellations $\nearrow\, \leftrightarrow \nxarrow$ and $\swarrow \leftrightarrow \sxarrow$,
these sijections can be applied directly in the current case.
On the other hand, since $\varphi_{\text{AR}_1}$ does involve such cancellations in the construction of $\Phi_4'$, this construction should be modified.
However, the modification is not difficult.
Using the bijection $\text{AR}_n \times \text{AP}_{n-1} \to \text{AR}_1 \times \text{AP}_n$ constructed in Construction \ref{Phi4} instead of $\Psi_{n,i}$, we obtain
\begin{align*}
	&\bigsqcup_{i=1}^n \bigsqcup_{\mu \in \text{AR}_n} \bigsqcup_{T \in \text{AP}_{n-1}} 
	\text{GT}(m_1(\mu,T,0),\ldots,m_{i-1}(\mu,T,0), x,m_i(\mu,T,1),\ldots,m_{n-1}(\mu,T,1)) \\
	&\quad \sijto \bigsqcup_{i=1}^n \bigsqcup_{(\mu,T) \in \text{AR}_1 \times \text{AP}_{n}} \text{GT}(k_1+c_1(T),\ldots,k_{i-1}+c_{i-1}(T),x,k_{i+1}+c_{i+1}(T),\ldots,k_{n}+c_{n}(T)) \\
	&\qquad = \text{AR}_1 \times \bigsqcup_{i=1}^n \bigsqcup_{T \in \text{AP}_{n}} \text{GT}(k_1+c_1(T),\ldots,k_{i-1}+c_{i-1}(T),x,k_{i+1}+c_{i+1}(T),\ldots,k_{n}+c_{n}(T)),
\end{align*}
and the construction of $\Phi_4''$ can be completed as in Construction \ref{Phi4}.

Next, we will verify that the sijection satisfies the compatibility conditions in Corollary \ref{cor::gmt-ar-sgt}.
First, all the sijections used in the construction act on arrows as reversing and permuting, and hence the sijection is compatible with $\eta_u$, $\eta_v$ and $\eta_w$.
In addition, it is compatible with $\eta_{X_1},\eta_{X_2},\ldots,\eta_{X_{n-1}}$ for the same reason as Proposition \ref{prop::MR-SGT}.
Last, we will prove the sijection is compatible with $\eta_{X_n}$.
In fact, the value of $\sum_{i=1}^n \eta_{X_i} -\eta_u + \eta_v$ is const. Indeed,
\begin{itemize}
\item for $T \in \text{GMT}(\mathbf{k})$, we have
\[
	\left(\sum_{i=1}^n \eta_{X_i} -\eta_u + \eta_v\right)(T) = {\sum}_n T = \sum \mathbf{k},
\]
\item and for $ (\mu,(A,T)) \in \text{AR}_n \times \text{SGT}(\mathbf{k})$, we have
\begin{align*}
	\left(\sum_{i=1}^n \eta_{X_i} -\eta_u + \eta_v\right)((\mu,(A,T))) = {\sum}_n A - (\text{\# of $\swarrow$ in $T$}) + (\text{\# of $\searrow$ in $T$}).
\end{align*}
Here, we have ${\sum}_nA = \sum T(\mathbf{k})$ from the definition of $\text{SGT}(\mathbf{k})$. Thus, the value is equal to $\sum \mathbf{k}$, based on the definition of $T(\mathbf{k})$.
\end{itemize}
Therefore, the value of $$\eta_{X_n} = \left(\sum_{i=1}^n \eta_{X_i} -\eta_u + \eta_v\right) - \left(\sum_{i=1}^{n-1} \eta_{X_i} -\eta_u + \eta_v\right)
= \sum \mathbf{k} - \left(\sum_{i=1}^{n-1} \eta_{X_i} -\eta_u + \eta_v\right),$$ is preserved under the action of the sijection.
Thus, the sijection is compatible with all $n+3$ statistics in Corollary \ref{cor::gmt-ar-sgt}.
\qed
\end{construction}

\begin{remark}
In Corollary \ref{cor::gmt-ar-sgt}, we say there are $n+3$ statistics with which the sijection is compatible, but some of them are degenerate, as can be understood from the proof of the compatibility in Construction \ref{construction::asm-sgt}.
In addition to the relation $\sum_{i=1}^n \eta_{X_i} -\eta_u + \eta_v = \sum \mathbf{k}$, we have $\eta_u+\eta_v+\eta_w = \binom{n}{2}$.
Thus, there are at most $n+1$ independent statistics, where we say statistics are independent if there are no relations between them.
Note that the number of independent statistics is not essentially meaningful.
While it may be considered that the appropriateness of a family of statistics can be measured by how fine it is, we will not delve further into this as it is beyond the scope of this paper.
\end{remark}

\section{Conclusions}
In this paper, we introduce the notion of compatibility of sijections and give transparent proofs of basic properties relevant to signed sets, sijections and the notion of compatibility.
In addition, we explain some combinatorial results by means of compatibility.
Our main contributions are as follows.
\begin{itemize}
\item We find and describe in detail the canonical one-to-one correspondences between Gelfand-Tsetlin patterns with a bottom row and that with a permuted bottom row.
This leads us to a new computational proof of the signed enumeration of Gelfand-Tsetlin patterns.
In addition, inspired by this proof, we define a generalization of Gelfand-Tsetlin patterns and extend the signed enumeration to this generalization.
\item We extend the definition of inversion numbers for shifted Gelfand-Tsetlin patterns and construct a more natural sijection than the conventional one between (generalized) monotone triangles and shifted Gelfand-Tsetlin patterns, through the notion of compatibility with the inversion numbers.
\item We give a bijective proof for the refined enumeration of an extension of alternating sign matrices with $n+3$ statistics, first proved in \cite{AF}.
\end{itemize}

We would also like to discuss the possibility of future applications of the notion of compatibility.
In fact, we believe it could be applied to many topics in integrable combinatorics.
On the one hand we have many determinant formulae in this field. (For example, those derived from the Lindstrom-Gessel-Viennot lemma, the Yang-Baxter equation, etc.)
As mentioned in \cite{FK2}, the theory of signed sets and sijections is closely related to linear algebra.
Therefore, known computational proofs of determinant formulae can be translated into bijective proofs.
On the other hand we have many results about refined enumerations in this field.
We should import these results into the bijective proofs through the notion of compatibility,
since this may offer hints to construct more natural (or completely new) sijections, as in Section 5.
In addition, considering compatibility can lead us to novel combinatorial results, as in Section 4.

Besides the above, the knowledge of this paper could be applied to the results obtained in \cite{FK2}, which is the sequel of the paper we mainly based our paper upon.
This might lead to the discovery of novel relations between alternating sign matrices and descending plane partitions.

\section*{Acknowledgements}
I would like to thank my supervisor, Prof. Ralph Willox for helpful discussions and comments on the manuscript.
I would also like to thank Prof. Takafumi Mase for useful discussions.
The author is partially supported by FoPM, WINGS Program, the University of Tokyo and JSPS KAKENHI No.\ 23KJ0795.

\appendix
\section{The details of Construction \ref{pi}}
In this section, we will describe the details of the construction of $\pi$ in Construction \ref{pi}.
We construct $\pi$ from $\sigma$ for one smaller $n$.
For $x_1,x_2,x_3,x_4 \in \mathbb{Z}$, we have sijections compatible with the normal statistics
\begin{align*}
	\sint{x_1}{x_2} \times \sint{x_2}{x_3} &\sijto -\sint{x_2}{x_1} \times ( \sint{x_2}{x_1} \sqcup \sint{x_1}{x_3} ) \\
	&\sijto -\sint{x_2}{x_1} \times \sint{x_2}{x_1} \sqcup -\sint{x_2}{x_1} \times \sint{x_1}{x_3},
\end{align*}
\begin{align*}
	\sint{x_1}{x_2} \times \sint{x_2}{x_3} &\sijto ( \sint{x_1}{x_3} \sqcup \sint{x_3}{x_2} ) \times -\sint{x_3}{x_2} \\
	&\sijto -\sint{x_1}{x_3} \times \sint{x_3}{x_2} \sqcup -\sint{x_3}{x_2} \times \sint{x_3}{x_2}
\end{align*}
and
\begin{align*}
	\sint{x_1}{x_2} \times \sint{x_2}{x_3} \times \sint{x_3}{x_4}
	&\sijto (\sint{x_1}{x_3} \sqcup \sint{x_3}{x_2}) \times -\sint{x_3}{x_2} \times (\sint{x_3}{x_2} \sqcup \sint{x_2}{x_4}) \\
	&\sijto -\sint{x_1}{x_3} \times \sint{x_3}{x_2} \times \sint {x_3}{x_2} \sqcup -\sint{x_1}{x_3} \times \sint{x_3}{x_2} \times \sint {x_2}{x_4} \\
	&\qquad \qquad \sqcup -\sint{x_3}{x_2} \times \sint{x_3}{x_2} \times \sint {x_3}{x_2} \sqcup -\sint{x_3}{x_2} \times \sint{x_3}{x_2} \times \sint {x_2}{x_4}.
\end{align*}
When $i=1$, we have
\begin{align*}
\text{GT}(\mathbf{l}) =& \bigsqcup_{\mathbf{m} \in \sint{l_1}{l_2} \times \sint{l_2}{l_3} \times \cdots \times \sint{l_{n-1}}{l_n}} \text{GT}(\mathbf{m}) \\
	\overset{C. \ref{du_dist}}{\ \sijto\ }& \bigsqcup_{\mathbf{m} \in -\sint{l_2}{l_1} \times \sint{l_2}{l_1} \times \sint{l_3}{l_4} \times \cdots \times \sint{l_{n-1}}{l_n}} \text{GT}(\mathbf{m})
	\sqcup \bigsqcup_{\mathbf{m} \in -\sint{l_2}{l_1} \times \sint{l_1}{l_3} \times \sint{l_3}{l_4} \cdots \times \sint{l_{n-1}}{l_n}} \text{GT}(\mathbf{m}) \\
	\overset{\sigma \sqcup =}{\ \sijto\ }& (\emptyset,\emptyset) \sqcup - \bigsqcup_{\mathbf{m} \in -\sint{l_2}{l_1} \times \sint{l_1}{l_3} \times \sint{l_3}{l_4} \cdots \times \sint{l_{n-1}}{l_n}} \text{GT}(\mathbf{m}) \\
	=& -\text{GT}(l_2,l_1,l_3,\ldots,l_n).
\end{align*}
When $i=n-1$, we have
\begin{align*}
\text{GT}(\mathbf{l}) =& \bigsqcup_{\mathbf{m} \in \sint{l_1}{l_2} \times \sint{l_2}{l_3} \times \cdots \times \sint{l_{n-1}}{l_n}} \text{GT}(\mathbf{m}) \\
	\overset{C. \ref{du_dist}}{\ \sijto\ }& \bigsqcup_{\mathbf{m} \in \sint{l_1}{l_2} \times \cdots \times \sint{l_{n-3}}{l_{n-2}} \times -\sint{l_{n-2}}{l_n} \times \sint{l_n}{l_{n-1}}} \text{GT}(\mathbf{m})
	\sqcup \bigsqcup_{\mathbf{m} \in \sint{l_1}{l_2} \times \cdots \times \sint{l_{n-3}}{l_{n-2}} \times -\sint{l_n}{l_{n-1}} \times \sint{l_n}{l_{n-1}}} \text{GT}(\mathbf{m}) \\
	\overset{= \sqcup \sigma}{\ \sijto\ }&  -\bigsqcup_{\mathbf{m} \in \sint{l_1}{l_2} \times \cdots \times \sint{l_{n-3}}{l_{n-2}} \times \sint{l_{n-2}}{l_n} \times \sint{l_n}{l_{n-1}}} \text{GT}(\mathbf{m}) \sqcup (\emptyset,\emptyset) \\
	=& -\text{GT}(l_1,l_2,\ldots,l_{n-2},l_n,l_{n-1}).
\end{align*}
Otherwise, we have
\begin{align*}
\text{GT}(\mathbf{l}) =& \bigsqcup_{\mathbf{m} \in \sint{l_1}{l_2} \times \sint{l_2}{l_3} \times \cdots \times \sint{l_{n-1}}{l_n}} \text{GT}(\mathbf{m}) \\
	\overset{C. \ref{du_dist}}{\ \sijto\ }&
	\begin{multlined}[t]
	\bigsqcup_{\mathbf{m} \in \sint{l_1}{l_2} \times \cdots \times \sint{l_{i-2}}{l_{i-1}} \times -\sint{l_{i-1}}{l_{i+1}} \times \sint{l_{i+1}}{l_{i}} \times \sint{l_{i+1}}{l_{i}} \times \sint{l_{i+2}}{l_{i+3}} \times \cdots \times \sint{l_{n-1}}{l_{n}}} \text{GT}(\mathbf{m}) \\
	\sqcup \bigsqcup_{\mathbf{m} \in \sint{l_1}{l_2} \times \cdots \times \sint{l_{i-2}}{l_{i-1}} \times -\sint{l_{i-1}}{l_{i+1}} \times \sint{l_{i+1}}{l_{i}} \times \sint{l_{i}}{l_{i+2}} \times \sint{l_{i+2}}{l_{i+3}} \times \cdots \times \sint{l_{n-1}}{l_{n}}} \text{GT}(\mathbf{m}) \\
	\sqcup \bigsqcup_{\mathbf{m} \in \sint{l_1}{l_2} \times \cdots \times \sint{l_{i-2}}{l_{i-1}} \times -\sint{l_{i+1}}{l_{i}} \times \sint{l_{i+1}}{l_{i}} \times \sint{l_{i+1}}{l_{i}} \times \sint{l_{i+2}}{l_{i+3}} \times \cdots \times \sint{l_{n-1}}{l_{n}}} \text{GT}(\mathbf{m}) \\
	\sqcup \bigsqcup_{\mathbf{m} \in \sint{l_1}{l_2} \times \cdots \times \sint{l_{i-2}}{l_{i-1}} \times -\sint{l_{i+1}}{l_{i}} \times \sint{l_{i+1}}{l_{i}} \times \sint{l_{i}}{l_{i+2}} \times \sint{l_{i+2}}{l_{i+3}} \times \cdots \times \sint{l_{n-1}}{l_{n}}} \text{GT}(\mathbf{m})
	\end{multlined} \\
	\overset{\sigma}{\ \sijto\ }&  (\emptyset,\emptyset) \sqcup -\bigsqcup_{\mathbf{m} \in \sint{l_1}{l_2} \times \cdots \times \sint{l_{i-2}}{l_{i-1}} \times \sint{l_{i-1}}{l_{i+1}} \times \sint{l_{i+1}}{l_{i}} \times \sint{l_{i}}{l_{i+2}} \times \sint{l_{i+2}}{l_{i+3}} \times \cdots \times \sint{l_{n-1}}{l_{n}}} \text{GT}(\mathbf{m})
	\sqcup (\emptyset,\emptyset) \sqcup (\emptyset,\emptyset) \\
	=& -\text{GT}(l_1,l_2,\ldots,l_{i-1},l_{i+1},l_{i},l_{i+2},\ldots,l_n). \qed
\end{align*}



\begin{thebibliography}{99}
	\fontsize{10pt}{0cm}\selectfont
	\setlength{\itemsep}{-3pt}
	\bibitem{BFZ1} Behrend, R.E., Di Francesco, P. \& Zinn-Justin, P. (2012). On the weighted enumeration of alternating sign matrices and descending plane partitions. J. Comb. Theory, Ser. A, 119(2), 331-363.
	\bibitem{BFZ2} Behrend, R.E., Di Francesco, P. \& Zinn-Justin, P. (2013). A doubly-refined enumeration of alternating sign matrices and descending plane partitions. J. Comb. Theory, Ser. A, 120(2), 409-432.
	\bibitem{Francesco1} Di Francesco, P. (2013). Integrable combinatorics. In XVIIth International Congress on Mathematical Physics: Aalborg, Denmark, 6-11 August 2012. World Scientific Publishing Co. (pp. 29-51).
	\bibitem{Doyle} Doyle, P. G. (2019). A category for bijective combinatorics. arXiv:1907.09015 [math.CO].
	\bibitem{Fis} Fischer, I. (2012). Sequences of labeled trees related to Gelfand-Tsetlin patterns, Advances in Applied Mathematics, 49(3-5), 165-195.
	\bibitem{Fis18} Fischer, I. (2018). Constant term formulas for refined enumerations of Gog and Magog trapezoids, J. Comb. Theory, Ser. A, 158, 560-604.
	\bibitem{FK1} Fischer, I. \& Konvalinka, M. (2020). A Bijective Proof of the ASM Theorem, Part I: The Operator Formula. Electron. J. Comb., 27(3), 3-35.
	\bibitem{FK2} Fischer, I. \& Konvalinka, M. (2020). A Bijective Proof of the ASM Theorem Part II: ASM Enumeration and ASM-DPP Relation. International Mathematics Research Notices, 2022(10), 7203-7230.
	\bibitem{AF} Fischer, I. \& Schreier-Aigner, F. (2023). The relation between alternating sign matrices and descending plane partitions: $n+3$ pairs of equivalent statistics.
Advances in Mathematics, 413, 108831.
	\bibitem{Kupe} Kuperberg, G. (1996). Another proof of the alternating-sign matrix conjecture. International Mathematics Research Notices, 1996(3), 139-150.
	\bibitem{MRR} Mills, W.H., Robbins, D.P. \& Rumsey, H. (1983). Alternating sign matrices and descending plane partitions, J. Comb. Theory, Ser. A, 34, 340-359.
	\bibitem{EC2} Stanley, R. P. (1999). Enumerative Combinatorics Vol. 2. Cambridge University Press.
	\bibitem{Stan} Stanley, R. P. (2009). Bijective Proof Problems. https://klein.mit.edu/~rstan/bij.pdf
	\bibitem{Zeil} Zeilberger, D. (1996). Proof of the alternating sign matrix conjecture. Electron. J. Comb., 3(2):Research Paper 13.
\end{thebibliography}
\end{document}